\documentclass[a4paper,10pt,english,leqno]{amsart}

\usepackage{color}
\usepackage{dsfont}
\usepackage[colorlinks = false]{hyperref}
\usepackage{mathrsfs}
\usepackage{mathtools}

\numberwithin{equation}{section}
\allowdisplaybreaks[1]

\usepackage{environ}
\makeatletter
\NewEnviron{Lalign}{\tagsleft@true\begin{align}\BODY\end{align}}
\makeatother

\usepackage[shortcuts,cyremdash]{extdash}

\newtheorem{theorem}{Theorem}[section]
\newtheorem{lemma}[theorem]{Lemma}

\newtheorem{proposition}[theorem]{Proposition}
\theoremstyle{definition}\newtheorem{definition}{Definition}[section]
\theoremstyle{remark}\newtheorem{remark}{Remark}[section]

\newcommand{\R}{\mathbb{R}}

\newcommand{\N}{\mathbb{N}}

\newcommand{\dt}{\,dt}

\newcommand{\seq}[1]{\left\{#1\right\}}
\newcommand{\abs}[1]{\left|#1\right|}

\newcommand{\norm}[1]{\left\| #1\right\|}

\DeclareMathOperator*{\esssup}{ess\,sup}

\usepackage{amssymb}
\usepackage{hyperref}
\usepackage{eufrak}
\usepackage{comment}

\newcommand{\En}{\Bbb{I}}

\renewcommand{\P}{\mathbb P} 
\newcommand{\EE}{\mathbb E} 

\newcommand{\LC}[2]{L^{#1}(\Omega; C^0(0,T; #2))}
\newcommand{\LOT}[2]{L^{#1}(\Omega_T;#2)}
\newcommand{\Hone}{H^1(M,h)} 
\newcommand{\Lp}[1]{L^{#1}(M,h)} 
\newcommand{\Ltwo}{\Lp{2}} 
\newcommand{\hS}{L_2(\mathfrak{U};\Lp{2})} 
\newcommand{\Div}{\mbox{div}_h\, }
\newcommand{\Grad}{\mbox{grad}_h\, }
\newcommand{\coppia}[2]{{}_{V^\ast}\langle #1,#2\rangle_{V}} 
\newcommand{\LOTV}{L^2\left(\Omega_T; \overrightarrow{\Ltwo}\right)}
\DeclareMathOperator{\supp}{supp}
\DeclareMathOperator{\dive}{div}
\DeclareMathOperator{\jac}{Jac}

\theoremstyle{example}

\title[Stochastic conservation laws on Riemannian manifolds]
{Well-posedness theory for stochastically forced conservation laws 
on Riemannian manifolds}

\author[L. Galimberti]{Luca Galimberti}
\address[Luca Galimberti]
{\newline Department of mathematics
\newline University of Oslo
\newline P.O. Box 1053, Blindern
\newline N--0316 Oslo, Norway} 
\email[]{lucaga@math.uio.no}

\author[K. H. Karlsen]{Kenneth H. Karlsen}
\address[Kenneth Hvistendahl Karlsen]
{\newline Department of mathematics
\newline University of Oslo
\newline P.O. Box 1053,  Blindern
\newline N--0316 Oslo, Norway} 
\email[]{kennethkarlsen@me.com}

\date{\today}



\thanks{This work was partially supported by the Research Council 
of Norway through the project Stochastic Conservation Laws (250674/F20).}

\begin{document}

\begin{abstract}
We investigate a class of scalar conservation laws on manifolds 
driven by multiplicative Gaussian (It\^o) noise. The Cauchy problem 
defined on a Riemanian manifold is shown to be well-posed. 
We prove existence of generalized kinetic solutions using 
the vanishing viscosity method. A rigidity result \textit{\`ala} 
Perthame is derived, which implies that 
generalized solutions are kinetic solutions and that 
kinetic solutions are uniquely determined by their initial 
data ($L^1$ contraction principle). Deprived of noise, the 
equations we consider coincide with those analyzed by 
Ben-Artzi and LeFloch \cite{Ben-Artzi/LeFloch}, who worked with 
Kru{\v{z}}kov-DiPerna solutions. In the Euclidian case, the 
stochastic equations agree with those examined by 
Debussche and Vovelle \cite{DV2010}.
\end{abstract}

\maketitle

\tableofcontents

\newpage

\section{Introduction}\label{sec: Introduction}
Hyperbolic conservation laws constitute a significant 
class of nonlinear partial differential equations (PDEs) that 
arises in numerous applications. Indeed, the starting 
point of many mathematical models are balance equations 
for physical quantities such as mass, momentum, and energy. 
Prominent examples include the Euler and Saint-Venant 
(shallow water) equations. Many advances in fluid dynamics are 
built upon the mathematical theory of hyperbolic conservation laws, which 
was developed to answer questions regarding existence, uniqueness, and 
structure of weak solutions (shock waves). 
Most aspects of the theory of conservation laws are nicely 
summarized in the monumental book \cite{Dafermos:2016aa}. 

In recent years many researchers added new effects 
and features to conservation laws in order to account for 
additional (or more realistic) physical phenomena. One interesting situation 
arises when the domain of the solution to a 
hyperbolic PDE is a curved manifold, in which case the curvature 
of the domain alters the underlying dynamics. 
Significant applications include geophysical fluid dynamics, e.g.~shallow 
water waves on the surface of a planet (caricature model of 
the atmosphere), and general relativity in which the Einstein-Euler 
equations are posed on a manifold with the metric being one of the unknowns. 
For scalar conservation laws defined on manifolds, the development 
of a theory of well-posedness and numerical approximations 
(of Kru{\v{z}}kov-DiPerna solutions) was initiated 
by LeFloch and co-authors \cite{Amorim:2005aa,Amorim:2008aa,Beljadid:2017aa,Ben-Artzi/LeFloch,Ben-Artzi:2009aa,LeFloch:2014aa,LeFloch:2008aa,LeFloch:2009aa} 
(see also Panov \cite{Panov:1997aa,Panov:2011aa}). 
The subject has been extended in several directions by different authors, 
including Giesselmann \cite{Giesselmann:2009aa}, Dziuk, Kr\"oner, 
and M\"uller \cite{Dziuk:2013aa}, Lengeler and M\"uller \cite{Lengeler:2013aa}, 
Giesselmann and M\"uller \cite{Giesselmann:2014aa}, and 
Kr\"oner, M\"uller, and Strehlau \cite{Kroner:2015aa}, 
and Graf, Kunzinger, and Mitrovic \cite{Graf:2017aa}.

In a different direction, many researchers have made 
attempts to extend the scope of hyperbolic conservation laws (on 
Euclidean domains) by adding ``random" effects. Randomness 
can enter these nonlinear PDEs in different ways, such as 
through stochastic forcing (source term) or in uncertain system 
parameters (flux function). Recently the mathematical study of 
stochastic conservation laws has emerged as an active field of study, 
linking several areas of mathematics, including nonlinear analysis 
and probability theory. Several works have studied the 
effect of It\^{o}-type stochastic forcing on 
scalar conservation laws. With emphasis on questions related to 
existence and uniqueness of generalized solutions, we mention 
Kim \cite{Kim:2003mz} (see also Vallet and 
Wittbold \cite{Vallet:2009uq}), who established the well-posedness 
of Kru{\v{z}}kov solutions in the  additive noise case. 
Feng and Nualart \cite{Feng:2008ul} presented 
a non-trivial modification of the Kru{\v{z}}kov framework 
that ensured the well-posedness for nonlinear noise 
functions (multiplicative noise). Debussche and Vovelle \cite{DV2010} 
advanced a general existence and uniqueness theory 
based on kinetic solutions. Additional results can be 
found in Bauzet, Vallet, and Wittbold \cite{Bauzet:2012kx},
Chen, Ding, and Karlsen \cite{Chen:2012fk}, Hofmanov\'a \cite{Hofmanova:2013aa}, 
Biswas, Karlsen, and Majee \cite{Biswas:2014gd}, Karlsen 
and Storr{\o}sten \cite{Karlsen:2015ab}, Debussche and Vovelle 
\cite{Debussche:2015aa}, Debussche, Hofmanov\'a, and Vovelle \cite{DHV2016}, 
Lv and Wu \cite{Lv:2016aa}, Kobayasi and Noboriguchi \cite{Kobayasi:2016aa}, 
and Dotti and Vovelle \cite{Dotti:2018aa,Dotti:2016ab}.  
A class of scalar conservation laws with ``rough path" flux 
was introduced and analyzed by Lions, Perthame, and Souganidis in a series of 
works \cite{Lions:2013ab,Lions:2013aa,Lions:2014aa}. They developed a 
pathwise well-posedness theory based on kinetic solutions. This theory was 
further extended by Gess and Souganidis \cite{Gess:2015aa,Gess:2017aa}.

No previous work has investigated the combined effect of 
nonlinear domains and Gaussian noise on the dynamics of shock waves.
In this paper we are interested in the well-posedness 
of generalized solutions for a class of scalar 
conservation laws that are posed on a curved manifold and perturbed by a Gaussian 
It\^{o}-type noise term. More precisely, let $(M,h)$ be an 
$n$-dimensional ($n\geq 1$) smooth Riemannian manifold, 
which we assume is compact, connected, oriented, and with no 
boundary ($\partial M = \emptyset$). We study the 
Cauchy problem for stochastically forced conservation laws of the form
\begin{equation}\label{eq:target}
	\begin{split}
		& du + \Div\, f_x(u)\,dt = B(u)\, dW(t), 
		\;\;\; x\in M, \; 0<t<T,
		\\
		& u(0,x) = u_0(x),
		\;\;\; \text{$x\in M$},
	\end{split}
\end{equation}
where $W$ is a cylindrical Wiener process with nonlinear 
noise coefficient (operator) $B(u)$, the flux $f=f_x(\xi)$ is a vector field 
on $M$ depending (nonlinearly) on a real parameter $\xi$ and assumed to 
be geometry-compatible in the sense of Ben-Artzi 
and LeFloch \cite{Ben-Artzi/LeFloch}, 
$\Div$ is the divergence operator linked to $(M,h)$, the initial datum 
$u_0$ is a bounded (random) function, and $u=u(\omega,t,x)$ is the unknown
that is sought up to a fixed final time $T>0$. 

Our investigation of \eqref{eq:target} utilizes firmly established tools 
for the analysis of (deterministic) conservation laws, 
specifically the kinetic formulation \cite{Perthame}. 
As in Debussche and Vovelle \cite{DV2010}, we make use 
of kinetic (and also generalized kinetic) solutions. 
Suppose for the moment that $W(t)$ is a one-dimensional Wiener process, 
and replace the operator $B(u)$ by a scalar function $g(x,u)$ that is (say) 
Lipschitz in both variables. In broad strokes, a process 
$u=u(\omega,t,x)$ is called a kinetic solution of \eqref{eq:target} if 
the associated process
\begin{equation}\label{intro:kinetic-func}
\varrho(\omega,t,x,\xi)=\En_{u(\omega,t,x)>\xi}:=
\begin{cases}
	1, & \text{if $\xi < u(\omega,t,x)$}\\
	0, & \text{if $\xi \ge u(\omega,t,x)$} 
\end{cases}
\end{equation}
satisfies (in the distributional sense) the kinetic equation
\begin{equation}\label{intro:kinetic}
\partial_t \varrho 
+ \left(f'_x(\xi),\nabla \varrho\right)_h
+ g(x,\xi) \partial_\xi \varrho \, \frac{dW}{dt}
= \partial_\xi \left(\frac{(g(x,\xi))^2}{2}
\partial_\xi \varrho \right)
+ \partial_\xi m,
\end{equation}
for some nonnegative (random) measure $m$, where 
$\left(\cdot,\cdot\right)_h$ is the inner product 
induced by the metric $h$. Note the property 
$\partial_\xi \varrho = -\delta(\xi-u)$ (and thus 
$\varrho(t,x,\cdot)\in BV_\xi$). Roughly speaking, the 
difference between a kinetic solution $\varrho$ and a 
generalized kinetic solution $\rho$ is that this structural property is 
replaced by the requirement $\partial_\xi \rho=-\nu$ for some 
Young measure $\nu$ on $\R_\xi$ (that is parameterized over $\omega,t,x$). 
We refer to Section \ref{sec: kinetic and generalized} for details.

Following an approach developed by Perthame \cite{Perthame} (instead 
of the ``doubling of variables" method \cite{DV2010}), we 
establish a rigidity result implying that generalized 
kinetic solutions are in fact kinetic solutions, 
and that they are uniquely determined by their initial 
data ($L^1$ contraction principle). To achieve this, we will 
employ a regularization procedure, commutator arguments 
\textit{\`ala} DiPerna-Lions, and the It\^{o} formula 
(for semimartingales) to show that a generalized kinetic solution $\rho$ and 
its square $\rho^2$ coincide (the ``rigidity result"), 
provided $\rho|_{t=0}=\En_{u_0>\xi}$. In our setting, added 
difficulties arise due to the stochastic forcing term 
and the nonlinear nature of the underlying domain $M$. 
In a nutshell, our strategy regarding the latter 
is the following: with the help of a partition of 
the unity, we will first localize the equation \eqref{eq:target} and 
then ``pull it back'' to the Euclidean space, where the  
regularization procedure (convolution in $x,\xi$) will be carried out. 
This leads to several equations that are subsequently aggregated 
into a single (global) equation, living on the manifold $M$. 
Eventually this global equation is used to derive the rigidity result. 
We note that the solution to this equation is smooth in $x,\xi$ (but not $t$). 
The equation contains a commutator term (regularization error), arising 
as a result of the convective flux $f$ and the 
nonlinear nature of the domain $M$, which converges to 
zero thanks to a proper adaption of the 
DiPerna-Lions commutator lemma \cite{DL89}. 
 
There are additional issues linked to the noise term 
in \eqref{eq:target} and its local time (quadratic covariation), including 
the handling of regularization errors tied to the $x$ and $\xi$ variables. 
For the moment, let us focus on the $\xi$ variable. For simplicity of 
presentation, we consider the Euclidean case and set $f\equiv 0$ (see 
Section \ref{sec: reduction and uniqueness} 
for the general case). To illustrate some of 
the difficulties, consider the kinetic equation 
\begin{equation}\label{intro:kinetic-simple}
\partial_t \varrho  + a(\xi) \partial_\xi  \varrho \, \frac{dW}{dt} 
=  \partial_\xi \left(\frac{(b(\xi))^2}{2} \partial_\xi  \varrho\right)
+ \partial_\xi m, 
\qquad
\text{$\varrho$ of the form \eqref{intro:kinetic-func}}, 
\end{equation}
where $a(\xi)$ and $b(\xi)$ are two, say, Lipschitz 
functions, noting that $b\equiv a$ corresponds to \eqref{intro:kinetic}. 
In order to compare $\varrho$ and $\varrho^2$, we need to 
determine the equation satisfied by $\varrho^2$. 
Let us attempt to do that for \eqref{intro:kinetic-simple}.
Fix $S\in C^2$ (say, $S(f)=f^2$). A 
formal application of the  It\^{o} formula suggests 
the following equation for $S(\varrho)$:
\begin{equation}\label{intro:kinetic-ito}
\partial_t S(\varrho) + a(\xi) \partial_\xi  S(\varrho)\, \frac{dW}{dt} 
=\partial_\xi \left(\frac{(b(\xi))^2}{2} \partial_\xi  S(\varrho)\right) 
+ S'(\varrho) \partial_\xi m+\mathcal{Q},
\end{equation}
where $\mathcal{Q}$ contains the quadratic terms 
coming from the second order differential operator and the covariation 
of the martingale part of the equation \eqref{intro:kinetic-simple}:
\begin{align*}
	\mathcal{Q} & = \left(a^2(\xi) - b^2(\xi)\right) 
	\frac{S''(f)}{2}\left(\partial_\xi \varrho \right)^2.
\end{align*}
At first glance, it may seem that noise induces extra regularity 
in the $\xi$ variable, as a result of the second order 
differential operator in \eqref{intro:kinetic-simple}.
This, however, is not the case. Only under the ``super-parabolicity" 
condition $a^2(\cdot)< b^2(\cdot)$ do we have $Q\le 0$, 
in which case $Q$ represents ``dissipation" (from noise). 
The specific case $b\equiv a$ corresponds to 
the kinetic equation for the stochastic conservation law. 
The perfect cancellation (i.e., $Q=0$) in this case 
is the basic reason why the $L^1$ contraction principle 
(uniqueness) holds for these nonlinear SPDEs.  
Unfortunately the equation \eqref{intro:kinetic-ito} is 
only suggestive (the calculations leading up to it are only formal). 
To make the calculations rigorous we regularize the 
linear equation \eqref{intro:kinetic-simple} using a 
mollifier $\phi_\delta(x,\xi)$,  thereby bringing in 
an additional type of regularization error $\mathcal{R}(\delta)$. 
If $a=b=g$ for some function $g(\cdot)$, then 
$\mathcal{R}(\delta)$ takes the form 
\begin{equation}\label{intro:reg-error}
	\abs{\, \int  \partial_\xi \varrho_{\delta}
	\Bigl(\left(g^2 \partial_\xi \varrho \right)
	\underset{(x,\xi)}{\star} \phi_\delta\Bigr) 
	-\Bigl(\left(g \partial_\xi  \varrho \right) 
	\underset{(x,\xi)}{\star} \phi_\delta \Bigr)^2
	\, d\xi \, dx}, 
	\quad
	\varrho_{\delta}:=\varrho \underset{(x,\xi)}{\star} \phi_\delta.
\end{equation}
Under a suitable regularity assumption on $g$, one can show 
that $\mathcal{R}(\delta)\to 0$ as $\delta\to 0$. 
The relevant assumption is dictated by the following 
derivable expression for $\mathcal{R}(\delta)$:
\begin{align*}
	\mathcal{R}(\delta)
	& = \frac12 \int \abs{g(\zeta)-g(\bar\zeta)}^2 
	\partial_\xi  \varrho(t,y,\zeta)
	\partial_\xi  \varrho(t,\bar y,\bar \zeta)
	\\ & \qquad \qquad
	\times \phi_\delta(x-y,\xi-\zeta)
	\phi_\delta(x-\bar y,\xi-\bar \zeta)
	\, d\zeta\, d\bar \zeta \, dy \, d\bar y \, dx \, d\xi.
\end{align*}
For a standard mollifier $\phi_\delta$, it turns out that this 
expression tends to zero as $\delta\to 0$ 
if $g$ is Lipschitz, or more generally if 
$\abs{g(\xi_1)-g(\xi_2)}^2 \le C \abs{\xi_1-\xi_2}\delta(\abs{\xi_1-\xi_2})$ 
for some continuous function $\delta$ on $\R_+$ with $\delta(0)=0$, which 
is consistent with \cite{DV2010}. Up to this point we have 
tried to extract some of the main ideas behind the uniqueness proof (in a  
simplified situation). Unfortunately, the complete proof in the 
general case is painfully long and technical. 
We refer to Section \ref{sec: reduction and uniqueness} for details.

As part of showing existence of kinetic solutions, we 
will establish the well-posedness of variational solutions 
\cite{Krylov/Rozowskii} for a stochastic parabolic problem, 
obtained by adding to \eqref{eq:target} a small diffusion term 
$\varepsilon\, \Delta_h$ ($\varepsilon>0$) involving 
the Laplace-Beltrami operator $\Delta_h$ on $(M,h)$ 
(cf.~Section \ref{sec: Vanishing} for details). 
Making use of a priori ($L^p$) estimates, we prove that 
there exists a kinetic solution to \eqref{eq:target} 
by arguing that the kinetic function linked to 
the variational solution (of the stochastic parabolic equation) 
converges weakly as $\varepsilon\to 0$ to a generalized kinetic solution 
of \eqref{eq:target} (cf.~Section \ref{Sec: Existence} for details). 
A crucial ingredient in the overall existence proof is a generalized 
It\^{o} formula for weak solutions to a wide class of 
SPDEs on Riemannian manifolds. Indeed, since variational 
solutions of the stochastic parabolic equation are merely $H^1$ 
regular in the $x$ variable, our general setting 
forces us to derive this It\^{o} formula. This is the topic of 
Section \ref{sec: generalized ito}.

\section{Background and hypotheses}\label{sec: Background}

We now provide the precise assumptions on each of the terms appearing in 
the stochastic conservation law \eqref{eq:target}. 
Basic background material on hyperbolic conservation laws can be 
found in the books \cite{Dafermos:2016aa,Perthame}.

\subsection{Geometric framework}
The underlying space is an $n$-dimensional ($n\geq 1$) 
smooth manifold $M$, which we assume to be compact, connected, 
oriented and with no boundary. Moreover, $M$ is 
endowed with a smooth Riemannian metric $h$. By this, we mean that $h$ is a 
positive-definite, 2-covariant tensor field, which thus 
determines for every $x\in M$ an inner product $h_x$ on $T_xM$ (the 
tangent space at $x$). For any two vectors $V_1,V_2\in 
T_xM$, we will henceforth write $h_x(V_1,V_2)=:\left(V_1,V_2\right)_{h_x}$ 
or even $\left(V_1,V_2\right)_h$ if the context is clear. 
We set $|V|_h:=\left(V,V\right)_h^{1/2}$. 
Recall that in local coordinates $x=(x^i)$, the 
derivations $\partial_i:=\frac{\partial}{\partial x^i}$ 
form a basis for $T_xM$, while the differential forms $dx^i$ 
determine a basis for the cotangent space $T_x^\ast M$. 
Therefore, in local coordinates, $h$ reads 
\[
h = h_{ij}dx^idx^j,  \;\; h_{ij}=\left(\partial_i,\partial_j\right)_h, 
\]
Here and elsewhere we employ the Einstein summation convention over 
repeated indices. We will denote by $(h^{ij})$ the inverse of 
the matrix $(h_{ij})$.

We denote by $dV_h$ the Riemannian density associated to $h$, which 
in local coordinates reads 
\[
dV_h = |h|^{1/2}dx^1\cdots dx^n, 
\]
where $|h|$ is the determinant of $h$. 
We recall that integration with respect to $dV_h$ 
is done in the following way: if $u\in C^0(M)$ has support contained in 
the domain of a single chart $\Phi:U\subset M\to \Phi(U)\subset \R^n$, then
\[
\int_Mu(x)\,dV_h(x)=\int_{\Phi(U)}(|h|^{1/2}u)
\circ\Phi^{-1}\,dx^1\cdots dx^n, 
\] 
where $(x^i)$ are the coordinates associated 
to $\Phi$. If $\supp u$ is not contained in 
a single chart domain, then the integral is defined as
\[
\int_Mu(x)\,dV_h(x)=\sum_{i\in\mathcal I}\int_M\alpha_i u\,dV_h(x), 
\]
where $(\alpha_i)_{i\in\mathcal{I}}$ 
is a partition of unity subordinate to some atlas $
\mathcal{A}$. Throughout the paper, we will assume for convenience that 
\[
\mathrm{Vol}(M,h):=\int_MdV_h=1. 
\]
Always in local coordinates, the gradient of a function $u:M\to \R$ 
is the vector field given by the following expression
\[
\Grad u := h^{ij}\partial_iu\,\partial_j. 
\]
The symbol $\nabla$ will indicate the Levi-Civita 
connection of $h$, namely the unique 
linear connection on $M$ that is compatible with $h$ and is symmetric. 
In particular, the 
covariant derivative of a vector field $X=X^\alpha\partial_\alpha$ 
is the $(1,1)$-tensor field which in local coordinates 
takes the following form
\[
(\nabla X)_j^\alpha = X^\alpha_{\;;j}:= \partial_j X^\alpha 
+ \Gamma^\alpha_{kj}X^k, 
\]
where $\Gamma^k_{ij}$ are the Christoffel 
symbols associated to $\nabla$:
\[
\Gamma^k_{ij} = \frac12 h^{kl}(\partial_ih_{jl}
+\partial_jh_{il}-\partial_lh_{ij}). 
\]
The divergence of a vector field $X$ 
is the function defined by
\[
\Div X = \partial_jX^j + \Gamma^j_{kj}X^k. 
\]
We recall that for a function $u\in C^1(M)$ and 
a smooth vector field $X$, the following 
integration by parts formula holds:
\[
\int_M \left(\Grad u,X\right)_h\,dV_h 
= -\int_M u \,\Div X \, dV_h. 
\]

We assume that $f=f_x(\xi)$ is a vector field 
on $M$ depending on the real parameter $
\xi$. More precisely, $f:M\times\R\to TM$, where $TM$ 
is the tangent bundle of $M$, and 
$f$ is smooth in both $x$ and $\xi$. 
We will call $f$ the \textit{flux} on $M$. Following 
\cite{Ben-Artzi/LeFloch}, we assume 
that $f$ is \textit{geometry-compatible}, 
in the sense that
\begin{equation}\label{eq: gc}
\Div f_x(\xi)=0,  
\;\;\;\; \xi\in\R, \;\;x\in M.
\end{equation}
Moreover, we impose the following polynomial growth conditions 
on $f$ and the derivative $f':=\partial_\xi f$:
\begin{equation}\label{eq: growth condition 0}
\begin{cases}
\abs{f_x(\xi)}_h \leq C_0\left(1+\abs{\xi}^r\right), 
& \xi\in \R, \; x\in M, 
\\
\abs{f'_x(\xi)}_h \leq C_0\left(1+\abs{\xi}^{r-1}\right), 
& \xi\in \R, \; x\in M, 
\end{cases}
\end{equation}
for some constants $C_0>0,r\geq 1$. 

We denote by $\Lp{p}, p\geq 1$ 
the usual Lebesgue spaces on $(M,h)$. The Sobolev 
spaces $H^k(M,h), k\geq 1$, are defined as 
the completion of $C^\infty(M)$ with respect to the norm
\[
\norm{u}_{H^k(M,h)} = \left(\sum_{j=0}^k
\int_M \abs{\nabla^ju}^2_h\,dV_h\right)^\frac12, 
\]
where 
\[
\abs{\nabla^ju}^2_h= h^{a_1b_1}\cdots 
h^{a_jb_j}(\nabla^ju)_{a_1\cdots a_j}(\nabla^ju)_{b_1\cdots b_j}
\]
(in a local chart) and $\nabla^j$ designates the $j^{th}$ 
covariant derivative of $u$. 
Note that the spaces $H^k(M,h)$ are Hilbert spaces. 
For further details, we refer to 
\cite{Aubin} and \cite{Hebey}. 

Note that, for $x\in M$ and $V\in T_xM$, it holds 
$(\nabla u)(V)=\left(\Grad u, V\right)_h$. For this reason, we will 
in this paper slightly abuse the notation by always writing 
$\left(\nabla u, V\right)_h$ instead of 
$\left(\Grad u, V\right)_h$; that is, we identify 
$\Grad u$ and $\nabla u$.

\subsection{Stochastic framework}
For background material on stochastic analysis and SPDEs, we refer to 
the books \cite{DZ,PR}. For a topological space $(X,\tau)$, the symbol 
$\mathcal{B}(X)$ will indicate its Borel $\sigma$-algebra. 
Given two measurable spaces $(X_i,\mathcal{M}_i),i=1,2$, 
and a map $f:X_1\to X_2$, the expression ``$f$ 
is $\mathcal{M}_1/\mathcal{M}_2$ measurable'' (or 
simply ``$f$ is $\mathcal{M}_1/\mathcal{M}_2$'') 
means that $f^{-1}(B)\in\mathcal{M}_1$ for all $B\in\mathcal{M}_2$.

Regarding the stochastic term, we are given a complete 
probability space $\left(\Omega,\mathcal{F},\P \right)$, along 
with a complete right-continuous filtration 
$(\mathcal{F}_t)_{t\geq 0}$. We denote by $\mathcal{P}$ 
the predictable $\sigma$-algebra on 
$\Omega_T:=\Omega\times[0,T]$ (associated 
to $(\mathcal{F}_t)_{t\geq 0}$).
By this, we mean 
\begin{align*}
\mathcal{P}&=\sigma\left(\left\{(s,t]\times F_s: 0\leq s < t\leq T, 
F_s\in \mathcal{F}_s\right\}\cup \left\{
\{0\}\times F_0: F_0\in\mathcal{F}_0\right\} \right)\\
&=\sigma\left(Y:\Omega_T\to\R: Y \mbox{ is left-continuous 
and adapted to } \mathcal{F}_t,t\in [0,T]\right). 
\end{align*}
Given an arbitrary separable Hilbert 
space $\tilde{H}$, a map $Y:\Omega_T\to\tilde{H}$ 
that is $\mathcal{P}/\mathcal{B}(\tilde{H})$ measurable 
will be called $\tilde{H}$-predictable. 
If $(X,\mathcal{M})$ is a measure space, a map $Y:\Omega_T\times 
X\to\tilde{H}$ will be called progressively measurable (with 
respect to $(\mathcal{F}_t)_{t\geq 0}$) if for all $t\in[0,T]$ the 
map $Y|_{\Omega\times [0,t]\times X}$ is $
\mathcal{F}_t\otimes\mathcal{B}([0,t])
\otimes\mathcal{M}/\mathcal{B}(\tilde{H})$ measurable.

Whenever we write that a statement holds true ``for a.e.~$(\omega,t)$", we 
will be referring to the product measure between $\P$ and 
the Lebesgue measure on $[0,T]$. 

The initial datum $u_0$ is in general a random variable, namely $u_0$ 
is $\mathcal{F}_0$-measurable and in $L^p(\Omega;\Lp{p})$ 
for some $p\in[1,\infty)$. The driving process $W$ is a cylindrical 
Wiener process, i.e., $W(t)=\sum_{k\geq 1}\beta_k(t)e_k$, where
\begin{enumerate}

\item $(e_k)_{k\geq 1}$ is an orthonormal basis for 
a separable Hilbert space $\mathfrak{U}$;

\item $(\beta_k(t))_{k\geq 1}$ are mutually independent real-valued standard 
Wiener processes relative to $(\mathcal{F}_t)_{t\geq 0}$;

\item the sum converges in $\mathcal{M}_T^2(\mathfrak{U}_0)$, 
the space of $\mathfrak{U}_0$-valued continuous, square 
integrable martingales, where $\mathfrak{U}_0$ is 
the auxiliary Hilbert space defined as
\[
\mathfrak{U}_0:=\left\{v=\sum_{k\geq 1}a_ke_k: 
\sum_{k\geq 1}\frac{a_k^2}{k^2}<\infty\right\}, 
\]
endowed with the norm
\[
\norm{v}^2_{\mathfrak{U}_0}
=\sum_{k\geq 1}\frac{a_k^2}{k^2},
\]
such that the embedding $\mathfrak{U}\hookrightarrow\mathfrak{U}_0$ 
is Hilbert-Schmidt (cf.~\cite{PR} for details).

\end{enumerate}

In our setting, we can assume without loss of generality 
that the $\sigma$-algebra $\mathcal{F}$ is countably generated and 
$(\mathcal{F}_t)_{t\in [0,T]}$ is the filtration 
generated by $u_0$ and $W$.

For each $z\in \Lp{2}$, we consider a mapping 
$B(z):\mathfrak{U}\to\Lp{2}$ defined by 
$B(z)e_k:=g_k(\cdot,z(\cdot)), k\in\N$, with $g_k\in C^0(M\times\R)$. 
We assume the following conditions on 
$\left\{g_k\right\}_{k\in\N}$: there exist 
positive constants $D_1,D_2$ such that
\begin{align}
\label{eq: definition of G}
& G^2(x,\xi):=\sum_{k\geq 1}\abs{g_k(x,\xi)}^2
\leq D_1\left(1+\abs{\xi}^2\right), 
\;\; x\in M, \;\xi\in\R,  \\ &
\label{eq: gk are lipschitz}
\sum_{k\geq 1}\abs{g_k(x,\xi)-g_k(y,\zeta)}^2
\leq D_2\left(d_h^2(x,y)+\abs{\xi-\zeta}^2\right), 
\;\; x,y\in M,  \;\xi,\zeta\in\R, 
\end{align}
where $d_h$ is the distance function on $(M,h)$. 
From \eqref{eq: definition of G}, it easily follows that
\[
B:\Lp{2}\to\hS, 
\]
where $\hS$ denotes the (separable Hilbert) space of Hilbert-Schmidt 
operators from $\mathfrak{U}$ to $\Lp{2}$. We also observe that, 
in view of \eqref{eq: gk are lipschitz}, $B$ is Lipschitz on $\Lp{2}$, 
because it holds, for any $z_1,z_2\in \Lp{2}$,
\begin{equation}\label{eq: B is Lipschitz on L2}
\begin{split}
& \norm{B(z_1)-B(z_2)}_{\hS}^2 
= \sum_{k\geq 1}\norm{B(z_1)e_k-B(z_2)e_k}_{\Ltwo}^2
\\ & \quad 
=\int_M\sum_{k\geq 1}\abs{g_k(x,z_1(x))-g_k(x,z_2(x))}^2\, d V_h(x)
\\ & \quad 
\leq \int_M D_2|z_1(x)-z_2(x)|^2 \, dV_h(x)
=D_2\norm{z_1-z_2}_{\Ltwo}^2. 
\end{split}
\end{equation}

For later use, we note the following simple result:

\begin{lemma}\label{lemma: G^2 is locally Lipschitz }
Assume that \eqref{eq: definition of G} 
and \eqref{eq: gk are lipschitz} hold. Then
\begin{align*}
& \abs{G^2(x_1,\xi_1)-G^2(x_2,\xi_2)} 
\\ & \quad
\leq \sqrt{D_1D_2}
\left\{\sqrt{1+\xi_1^2}+\sqrt{1+\xi_2^2}\right\}
\sqrt{d_h^2(x_1,x_2)+\abs{\xi_1-\xi_2}^2}, 
\end{align*}
for all $x_1,x_2\in M$ and $\xi_1,\xi_2\in\R$.
\end{lemma} 
\begin{proof}
A direct computation gives
\begin{align*}
& \abs{G^2(x_1,\xi_1)-G^2(x_2,\xi_2)}
\leq \sum_{k\geq 1}\abs{g_k^2(x_1,\xi_1)-g_k^2(x_2,\xi_2)}
\\ & \quad \leq \left( \, \sum_{k\geq 1}
\abs{g_k(x_1,\xi_1)+g_k(x_2,\xi_2)}^2\right)^{1/2} 
\left(\, \sum_{k\geq 1}
\abs{g_k(x_1,\xi_1)-g_k(x_2,\xi_2)}^2\right)^{1/2}
\\ & \quad \leq \sqrt{D_1D_2}
\left\{\sqrt{1+\xi_1^2}+\sqrt{1+\xi_2^2}\right\}
\sqrt{d_h^2(x_1,x_2)+\abs{\xi_1-\xi_2}^2}. 
\end{align*}
\end{proof}

In \eqref{eq:target}, ``$B(u)\, dW$" is understood as 
an It\^{o} stochastic integral. 
We refer to \cite{DZ,PR} for a detailed 
construction of the stochastic integral
$$
N_t := \int_0^t H \, dW 
= \sum_{k\ge 1}  \int_0^t H_k \, d\beta_k, 
\qquad H_k:=H e_k,
$$
for any predictable $L^2(M,h)$-valued process 
$$
H\in L^2\big(\Omega,\mathcal{F};L^2(0,T;\hS)\big).
$$
We will frequently make use of the Burkholder-Davis-Gundy 
inequality \cite[App.~D]{PR}, which applied to $N_t$ reads
\begin{equation}\label{eq: BDG}
\EE\left[ \sup_{t\in [0,T]} \norm{\sum_{k\ge 1} 
\int_0^t H_k \, d \beta_k}_{L^2(M,h)}^p \right]
\le C\, \EE\left[\left(\int_0^T \sum_{k\ge 1}
\norm{H_k}_{L^2(M,h)}^2 \dt\right)^{\frac{p}{2}} \right],
\end{equation} 
where $C$ is a constant depending on $p\ge 1$. 

\begin{remark}
Condition \ref{eq: gk are lipschitz} is used for convenience and 
simplicity of presentation. It can be replaced by the 
following more general condition (conforming to \cite{DV2010}):
\[
\sum_{k\geq 1}\abs{g_k(x,\xi)-g_k(y,\zeta)}^2
\leq D_2\left(d_h^2(x,y)+\abs{\xi-\zeta}
\delta(\abs{\xi-\zeta})\right), 
\]
for $x,y\in M$ and $\xi,\zeta\in\R$, where 
$\delta: [0,\infty)\to[0,\infty)$ is a continuous 
non-decreasing function such that $\delta(0)=0$. The relevant
proofs remain the same modulo some notational changes.
\end{remark}

\section{Kinetic solutions and main result}\label{sec: kinetic and generalized}
Following Debussche and Vovelle \cite{DV2010}, we 
introduce the concepts of kinetic and  generalized 
kinetic solutions for stochastic conservation laws 
defined on a manifold. We start with the notion of kinetic measure.

\begin{definition}[kinetic measure]\label{def: kinetik measure}
We say that a map $m$ from $\Omega$ to the set of 
non-negative finite measures over $[0,T]\times M\times\R$ 
is a \textit{kinetic measure} if
\begin{enumerate}

\item $m$ is measurable, that is, for each 
$\phi\in C_b^0([0,T]\times M\times\R)$, $m(\phi):\Omega\to\R$ 
is measurable, where $m(\phi)$ denotes the action of $m$ on $\phi$;

\item $m$ is integrable, that is,
\[
\EE\, m([0,T]\times M\times \R)<\infty\,;
\]

\item $m$ vanishes for large $\xi$, that is, 
if $B_R^c=\left\{\xi\in\R: \abs{\xi}\geq R\right\}$, then 
\[
\lim_{R\to\infty}\EE\, m([0,T]\times M\times B_R^c)=0;
\]

\item for all $\phi\in C^0_b(M\times\R)$, the process
\[
t\mapsto \int_{[0,t]\times M\times\R}\phi(x,\xi)\, m(ds,dx,d\xi)
\in L^2(\Omega\times [0,T])
\]
admits a predictable representative.

\end{enumerate}
\end{definition} 

\begin{definition}[kinetic solutions]\label{def: solution}
With $u_0\in L^\infty(\Omega,\mathcal{F}_0;L^\infty(M,h))$, 
set $\rho_0:=\En_{u_0>\xi}$. A measurable function 
$u:\Omega\times [0,T]\times M\to \R$ is said to be 
a \textit{kinetic solution} of \eqref{eq:target} 
with initial data $u_0$ if $(u(t))_{t\in [0,T]}$ is predictable;
$\forall p\in [1,\infty)$, there exists a positive 
constant $C_p$ such that
\[
\EE\left(\esssup_{t\in [0,T]} \norm{u(t)}^p_{\Lp{p}}\right)\leq C_p;
\]
and there exists a kinetic measure $m$ such that 
$\P$-a.s.~the function $\rho:=\En_{u>\xi}$ satisfies 
\begin{equation*}
\begin{split}
\int_0^T & \int_M\int_\R\rho\,\partial_t\psi \, d\xi \,dV_h(x)\,dt
+\int_M\int_\R\rho_0\,\psi(0,x,\xi)\, d\xi \,dV_h(x)\\
&+\int_0^T\int_M\int_\R\rho\left(f'_x(\xi),\nabla \psi\right)_h 
\,d\xi \,dV_h(x)\,dt
=m(\partial_\xi\psi)\\
& -\sum_{k\geq 1}\int_0^T\int_M g_k(x,u(t,x))\,\psi(t,x,u(t,x))
\,dV_h(x)\,d\beta_k(t)\\
&-\frac12\int_0^T\int_M \partial_\xi\psi(t,x,u(t,x)) 
\,G^2(x,u(t,x))\,dV_h(x)\,dt, 
\end{split}
\end{equation*}
for all $\psi\in C^1_c([0,T)\times M\times\R)$.
\end{definition} 

Let $(X,\mu)$ be a finite measure space, 
and denote by $\mathrm{Prob}(\R)$ the set of 
probability measures on $\R$. A map $\nu :X\to\mathrm{Prob}(\R)$ is a 
\textit{Young measure} on $X$ if, for all $\phi\in C_b(\R)$, 
the map $z\mapsto \nu_z(\phi)$ from $X$ to $\R$ is measurable. 
We say that a Young measure $\nu$ \textit{vanishes at infinity} if, for 
every $p\in [1,\infty)$,
\begin{equation}\label{def: Young measure}
\int_X\int_\R \abs{\xi}^p\, 
\nu_{z}(d\xi)\,\mu(dz)<\infty. 
\end{equation}

Let $(X,\mu)$ be a finite measure space. 
A measurable function $\rho:X\times\R\to [0,1]$ is said 
to be a \textit{generalized kinetic function} 
if there exists a Young measure $\nu$ on $X$ vanishing at infinity 
such that, for $\mu$-a.e.~$z\in X$ and for all $\xi\in\R$, 
$\rho(z,\xi)=\nu_z(\xi,\infty)$. 

We say that $\rho$ is a \textit{kinetic function} if 
there exists a measurable function $u:X\to\R$ such 
that $\rho(z,\xi)=\En_{u(z)>\xi}$ a.e., or, 
equivalently, $\nu_z=\delta_{u(z)}$ for $\mu$-a.e.~$z\in X$.


A generalized kinetic function $\rho$ satisfies 
$\partial_\xi \rho=-\nu$. If $\rho$ is a kinetic function, then 
$\partial_\xi \rho=-\delta_{u}$. Let $\rho$ be a 
generalized kinetic function. Note that the function
$\chi_\rho(z,\xi):=\rho(z,\xi)-\En_{0>\xi}$ 
is, contrary to $\rho$, integrable on $\R_\xi$.

\begin{definition}[generalized kinetic solution]\label{def: generalized solution}
Fix a generalized kinetic function $\rho_0:\Omega\times M\times \R\to [0,1]$. 
We call $\rho:\Omega\times [0,T]\times M\times \R\to [0,1]$ 
a \textit{generalized kinetic solution} of \eqref{eq:target} 
with initial data $\rho_0$ if $\chi_\rho=\rho-\En_{0>\xi}$ 
is $\mathcal{P}/\mathcal{B}(L^2(M\times\R))$ 
measurable and for all $p\in [1,\infty)$ there exists $C_p> 0$ such that
\[
\EE\left(\esssup_{t\in[0,T]}
\int_M\int_\R\abs{\xi}^p\, 
\nu_{\omega,t,x}(d\xi)\, dV_h(x)\right)\leq C_p, 
\]  
where $\nu=-\partial_\xi\rho$ is a Young measure, and 
if there exists a kinetic measure $m$ 
such that $\P$-a.s.~($f'=\partial_\xi f$)
\begin{equation}\label{eq: generalized solution}
\begin{split}
\int_0^T&\int_M\int_\R\rho\,\partial_t\psi
\, d\xi \,dV_h(x)\,dt
+\int_M\int_\R\rho_0\,\psi(0,x,\xi)
\, d\xi \,dV_h(x) \\ & 
+\int_0^T\int_M\int_\R\rho\left(f'_x(\xi),
\nabla \psi\right)_h \,d\xi \,dV_h(x)\,dt
=m(\partial_\xi\psi)\\
& -\sum_{k\geq 1}\int_0^T\int_M\int_\R g_k(x,\xi)\,\psi
\,\nu_{\omega,t,x}(d\xi)\,dV_h(x)\,d\beta_k(t)\\
&-\frac12\int_0^T\int_M\int_\R\partial_\xi\psi \,G^2(x,\xi)
\,\nu_{\omega,t,x}(d\xi)\,dV_h(x)\,dt,
\end{split}
\end{equation}
for all $\psi\in C^1_c([0,T)\times M\times\R)$.
\end{definition} 

We note the following result, which is identical 
to \cite[Proposition 10]{DV2010} 
(see also \cite[Lemma 1.3.3]{Dafermos:2016aa}). 
It tells us that a generalized kinetic solution 
possesses (weak) left and right limits 
at every instant of time. 

\begin{lemma}\label{lemma: left and right limits}
Let $\rho$ be a generalized kinetic solution to 
\eqref{eq:target} with initial data $\rho_0$. 
For any $t_\ast\in[0,T]$, there exist 
generalized kinetic functions $\rho^{\ast,\pm}$ on 
$\Omega\times M\times\R$ such that $\P$-a.s.,
\begin{align*}
&\iint_{M\times\R} \rho(t_\ast-h)\,\psi \,d\xi\,dV_h(x)
\overset{h\downarrow 0}{\to} 
\iint_{M\times\R}\rho^{\ast,-} \,\psi \,d\xi\,dV_h(x),
\\ &
\iint_{M\times\R}\rho(t_\ast+h) \,\psi d\xi\,dV_h(x)
\overset{h\downarrow 0}{\to} 
\iint_{M\times\R}\rho^{\ast,+} \,\psi \,d\xi\,dV_h(x),
\end{align*}
for all $\psi\in C^1_c(M\times\R)$. Moreover, $\P$-a.s.,
\[
\iint_{M\times\R}(\rho^{\ast,+}-\rho^{\ast,-}) 
\,\psi\,d\xi\,dV_h(x) 
= -\int_{[0,T]\times M\times\R}\partial_\xi \psi(x,\xi)
\,\En_{\{t_\ast\}}(t) \,m(dt,dx,d\xi). 
\]
In particular, $\P$-a.s., the set 
$\seq{t_\ast\in [0,T]:\rho^{\ast,+}\neq\rho^{\ast,-}}$ 
is at most countable.
\end{lemma}

For a proof of this result see the above-mentioned reference. 
For a generalized kinetic solution $\rho$, we henceforth define 
the accompanying functions $\rho^\pm$ by setting 
$\rho^\pm(t_*)=\rho^{\ast,\pm}(t_*)$ for $t_*\in[0,T]$. 
Clearly, $\rho^+(t)=\rho^-(t)=\rho(t)$ for a.e.~$t\in [0,T]$.

As in \cite[page 14]{DV2010}, we can replace \eqref{eq: generalized solution} 
by a weak formulation that is pointwise in time: 
$\forall t\in[0,T]$ and $\forall \psi\in C^1_c(M\times\R)$,
\begin{equation}\label{eq: generalized solution (2)}
\begin{split}
-&\int_M\int_\R\rho^+(t)\,\psi\, d\xi \,dV_h(x)
+\int_M\int_\R\rho_0\,\psi\, d\xi \,dV_h(x)
\\ &\qquad\quad 
+\int_0^t\int_M\int_\R\rho^+(s)
\left(f'_x(\xi),\nabla \psi\right)_hd\xi \,dV_h(x)\, ds
\\ & \quad =-\sum_{k\geq 1}\int_0^t\int_M\int_\R g_k(x,\xi)
\,\psi\,\nu_{\omega,s,x}(d\xi)\,dV_h(x)\,d\beta_k(s)\\
&\qquad\quad -\frac12\int_0^t\int_M\int_\R
\partial_\xi\psi \,G^2(x,\xi)\,
\nu_{\omega,s,x}(d\xi)\,dV_h(x)\,ds
\\ &\qquad\quad
+\int_{[0,t]\times M\times\R}
\partial_\xi\psi \,m(ds,dx,d\xi),  
\quad \P\mbox{-almost surely}. 
\end{split}
\end{equation}
This ``pointwise in time'' formulation will be 
utilized in the uniqueness proof.

The next theorem contains the main result of the paper, namely 
the existence, uniqueness, and stability of kinetic solutions.
Moreover, the trajectories (in $L^p$) of 
kinetic solution are continuous, $\P$-almost surely. 
The proof of the theorem is scattered across 
Section \ref{sec: reduction and uniqueness} (uniqueness) 
and Section \ref{Sec: Existence} (existence).  

\begin{theorem}[well-posedness]\label{thm:well-posed}
Suppose \eqref{eq: growth condition 0}, \eqref{eq: definition of G}, 
\eqref{eq: gk are lipschitz} hold. There exists a 
unique kinetic solution $u$ of \eqref{eq:target} with 
initial datum $u_0\in L^\infty(\Omega,\mathcal{F}_0;L^\infty(M,h))$. 
If $u_1, u_2$ are kinetic solutions of \eqref{eq:target} with 
initial data $u_{1,0}, u_{2,0}$, respectively, then
the following $L^1$ contraction property holds: 
\begin{equation}\label{L1-contraction}
	\EE \int_M \abs{\left(u_1-u_2\right)(t,x)} \, \,dV_h(x)
	\leq \EE \int_M \abs{\left(u_{1,0}-u_{2,0}\right)(x)}\,dV_h(x),
\end{equation} 
for $t\in [0,T]$. Besides, $u$ has 
a representative in $L^p(\Omega;L^\infty(0,T;L^p(M,h)))$ 
with continuous trajectories in $L^p(M,h)$, $\P$-a.s., for any $p\in [1,\infty)$.
\end{theorem}

\section{Rigidity and uniqueness results}\label{sec: reduction and uniqueness}

The aim of this section is to show that generalized kinetic solutions are in fact 
kinetic solutions and that they are unique. To achieve this, we will employ a 
regularization procedure, which will enable us to compare $(\rho^{\pm})^2$ 
and $\rho^{\pm}$, following \cite{Perthame} (see also \cite{Dalibard}). 
Our strategy is the following: with the help of a 
smooth partition of unity subordinate to a finite atlas 
$\mathcal{A}=\left\{\kappa:X_\kappa\to\tilde{X}_\kappa \right\}_\kappa$, 
we localize the equation \eqref{eq: generalized solution (2)} 
on $X_\kappa\subset M$, thereby obtaining $\sharp(\mathcal{A})$ equations, 
indexed by $\kappa$, that are ``pulled back'' to $\tilde{X}_\kappa \subset\R^n$ 
and regularized with a mollifier on $\R^n\times\R$. 
Subsequently, we aggregate the regularized equations to 
arrive at a single SPDE, parameterized by $(x,\xi)\in M\times\R$. 
We renormalize this equation using It\^{o}'s formula. 
This enables us to analyze the difference between the regularized 
versions of $\rho^+$ and $(\rho^+)^2$, eventually 
concluding the rigidity result. 

The key result of this section is

\begin{proposition}[rigidity result]\label{prop: Reduction and uniqueness}
Suppose \eqref{eq: growth condition 0}, \eqref{eq: definition of G}, 
and \eqref{eq: gk are lipschitz} hold. Let $\rho$ be a 
generalized kinetic solution to \eqref{eq:target} 
with initial data $\rho_0$. Then, for all $t\in [0,T]$,
$$
\EE \int_{M\times\R} \left( \rho^+-(\rho^+)^2\right) (t)
\,d\xi \,dV_h(x)
\leq \EE\int_{M\times\R}\left(\rho_0-\rho_0^2\right)
\,d\xi \,dV_h(x),
$$
where the temporal right limit $\rho^+$ of $\rho$ is defined 
in Lemma \ref{lemma: left and right limits}.
\end{proposition}

Note that if $\rho_0$ is a kinetic function, 
i.e., $\rho_0=\En_{u_0>\xi}$ for some bounded random 
function $u_0$, then $\rho^+-(\rho^+)^2=0$ (rigidity) 
and accordingly $\rho^+$ takes values in $\seq{0,1}$.
Consequently, $\rho^+$ is a kinetic function, i.e., there exists 
a measurable function $u$ such that $\rho=\En_{u>\xi}$ (and Theorem 
\ref{thm:well-posed} will follow from this).

The proof of Proposition \ref{prop: Reduction and uniqueness} 
will be laid out in several subsections.

\subsection{Localized equations}

To prove Proposition \ref{prop: Reduction and uniqueness}, 
we need some preparational material. We will work with a finite atlas 
$\mathcal{A}=\left\{\kappa:X_\kappa\to\tilde{X}_\kappa \right\}_\kappa$, 
where $X_\kappa\subset M$ is an open subset of $M$ and 
$\tilde{X}_\kappa \subset\R^n$ is an open 
subset of $\R^n$. The typical point of $\R^n$ will be denoted by $z$. 
We take a smooth partition of the unity 
$\{\alpha_\kappa\}_{\kappa\in\mathcal{A}}$ subordinate to 
$\mathcal{A}$, such that
\begin{enumerate}

\item $\alpha_\kappa\geq 0,\,\sum_{\kappa\in\mathcal{A}}\alpha_\kappa=1$,

\item $\alpha_\kappa\in C^\infty(M)$, and

\item $\mbox{supp }\alpha_\kappa\subset X_\kappa$ (and compact).
 
\end{enumerate}

Let $\rho$ be a generalized kinetic solution with initial data 
$\rho_0$ (not necessarily of the form $I_{u_0>\xi}$). 
As $\alpha_\kappa\in C^\infty(M)$, it follows from 
\eqref{eq: generalized solution (2)} that the 
function $\alpha_\kappa\rho^+(t)$ solves for all $
\psi\in C^1_c(M\times\R)$ and $t\in[0,T]$,
\begin{equation}\label{eq: what alfakapparo solves}
\begin{split}
& -\int_M\int_\R\alpha_\kappa(x)\rho^+(t)\,\psi \, d\xi \,dV_h(x)+
\int_M\int_\R\alpha_\kappa(x)\rho_0\,\psi\, d\xi \,dV_h(x)\\
& \quad \qquad
+\int_0^t\int_M\int_\R\alpha_\kappa(x)\rho^+(s)
\left(f'_x(\xi),\nabla \psi\right)_hd\xi \,dV_h(x)\,ds\\
&\quad \qquad \quad
+\int_0^t\int_M\int_\R\rho^+(s)\,\psi
\left(f'_x(\xi),\nabla \alpha_\kappa\right)_h \, d\xi \,dV_h(x)\, ds\\
& \quad =-\sum_{k\geq 1}\int_0^t\int_M\int_\R g_k(x,\xi)\,\psi\,\alpha_\kappa(x)\,
\nu_{\omega,s,x}(d\xi)\,dV_h(x)\,d\beta_k(s)\\
& \quad \qquad 
-\frac12\int_0^t\int_M\int_\R\partial_\xi\psi \,G^2(x,\xi)
\,\alpha_\kappa(x)\,\nu_{\omega,s,x}(d\xi)\,dV_h(x)\,ds\\
&\quad \qquad \quad
+\int_{[0,t]\times M\times\R}\partial_\xi\psi\,\alpha_\kappa(x)
\,m(ds,dx,d\xi), \quad\P\mbox{-almost surely}. 
\end{split}
\end{equation}

We define 
\[
\rho_\kappa^+(\omega,t,z,\xi):=
\alpha_\kappa(z)\,\rho^+(\omega,t,z,\xi)\,\abs{h_\kappa(z)}^{1/2}
\]
and
\[
\rho_{0,\kappa}(\omega,z,\xi):=
\alpha_\kappa(z)\,\rho_0(\omega,z,\xi)\,\abs{h_\kappa(z)}^{1/2}, 
\]
with $\omega\in\Omega,t\in[0,T],z\in\tilde{X}_\kappa \subset\R^n,\xi\in\R$. 

\begin{remark}
Most of the time, but not always, we will use the 
convention of not explicitly writing the chart: 
for example, writing $\alpha_\kappa(z)$ 
instead of $\alpha_\kappa(\kappa^{-1}(z))$). 
Furthermore, we write $\abs{h_\kappa(z)}^{1/2}$ (instead of 
$|h(z)|^{1/2}$) to remind us that it is a local expression on 
$X_\kappa$, and not a global one on $M$. In other words,
\[
\tilde{X}_\kappa \ni z\mapsto \abs{h_\kappa(z)}^{1/2}
\]
is a smooth function, and so is its inverse $\abs{h_\kappa(z)}^{-1/2}$. 
Given a function $v$ compactly supported in $\tilde{X}_\kappa $, we 
can ``lift'' to $M$ the function $\frac{v}{\abs{h_\kappa}^{1/2}}$, 
obtaining a global function on $M$, compactly supported in $X_\kappa$ 
(outside $X_\kappa$ the function is set to zero). 
\end{remark}

We observe that for fixed $\omega\in\Omega,t\in[0,T],\xi\in\R$, 
\[
\supp \rho_\kappa^+(\omega,t,\cdot,\xi)
\subset \kappa(\supp\alpha_\kappa)\subset\subset 
\tilde{X}_\kappa \subset \R^n
\] 
and
\[
\supp \rho_{0,\kappa}(\omega,\cdot,\xi)\subset 
\kappa(\supp\alpha_\kappa)\subset\subset \tilde{X}_\kappa \subset \R^n, 
\] 
and thus they may be seen as global functions on $\R^n$. 

\subsection{Regularization of localized equations}
Let $\phi_1$, $\phi_2$ be a standard mollifiers on $\R^n$ and 
$\R$, respectively, and define the function
\[
\phi_\varepsilon(z,\xi) := 
\varepsilon^{-n}\phi_1\left(\frac{z}{\varepsilon}\right) 
\varepsilon^{-1}\phi_2\left(\frac{\xi}{\varepsilon}\right),
\quad z\in\R^n, \;\xi\in\R, 
\]
whose support is contained in 
$\overline{B_\varepsilon(0)}\times[-\varepsilon,\varepsilon]$. 

We define regularizations of $\rho_{\kappa}^+$ and 
$\rho_{0,\kappa}$. For $\omega\in\Omega$, $t\in[0,T]$, 
$z\in\R^n$, $\xi\in\R$,
\begin{align*}
& (\rho_{\kappa}^+)_\varepsilon (\omega,t)(z,\xi) 
\\ & \quad :=\int_{\R^n\times\R}\rho_{\kappa}^+(\omega,t,\bar{z},\bar{\xi})\,
\varepsilon^{-n}\phi_1\left(\frac{z-\bar{z}}{\varepsilon}\right)
\varepsilon^{-1} \phi_2\left(\frac{\xi-\bar{\xi}}{\varepsilon}\right)
\,d\bar{\xi}\,d\bar{z}
\\ & \quad =\int_{\R^n\times\R}\alpha_\kappa(\bar{z})\,
\rho^+(\omega,t,\bar{z},\bar{\xi})\,
\abs{h_\kappa(\bar{z})}^{1/2}\varepsilon^{-n}\phi_1
\left(\frac{z-\bar{z}}{\varepsilon}\right)\varepsilon^{-1} 
\phi_2\left(\frac{\xi-\bar{\xi}}{\varepsilon}\right)\,d\bar{\xi}\,d\bar{z}, 
\\ & (\rho_{0,\kappa})_\varepsilon(\omega)(z,\xi) 
\\ & \quad := \int_{\R^n\times\R}\rho_{0,\kappa}(\omega,\bar{z},\bar{\xi})\,
\varepsilon^{-n}\phi_1\left(\frac{z-\bar{z}}{\varepsilon}\right)
\varepsilon^{-1} \phi_2\left(\frac{\xi-\bar{\xi}}{\varepsilon}\right)
\,d\bar{\xi}\,d\bar{z} 
\\ & \quad =\int_{\R^n\times\R}\alpha_\kappa(\bar{z})\,
\rho_0(\omega,\bar{z},\bar{\xi})\,\abs{h_\kappa(\bar{z})}^{1/2}\varepsilon^{-n}
\phi_1\left(\frac{z-\bar{z}}{\varepsilon}\right)
\varepsilon^{-1} \phi_2\left(\frac{\xi-\bar{\xi}}{\varepsilon}\right)
\,d\bar{\xi}\,d\bar{z}.
\end{align*}
We set $\varepsilon_\kappa :=
\mbox{dist}\left(\kappa(\supp\alpha_\kappa), 
\partial\tilde{X}_\kappa \right)>0$. 

The main properties of $(\rho_{\kappa}^+)_\varepsilon$ and 
$(\rho_{0,\kappa})_\varepsilon$ are listed in 
\begin{lemma}\label{lemma: properties of rho kappa epsilon}
Let $\kappa\in\mathcal{A}$. Then
\begin{enumerate}

\item $(\rho_{\kappa}^+)_\varepsilon(\omega,t)\in 
C^\infty(\R^n_z\times\R_\xi)$, for all $(\omega,t)\in\Omega_T$.

\item for $\varepsilon < \varepsilon_\kappa$ and for any 
$\omega\in\Omega,t\in[0,T],\xi\in\R$,
\[
\supp (\rho_{\kappa}^+)_\varepsilon(\omega,t)(\cdot,\xi) \subset 
\kappa(\supp\alpha_\kappa) +\overline{B_\varepsilon(0)}
\subset\subset \tilde{X}_\kappa. 
\] 
This implies in particular that for any $(\omega,t)\in\Omega_T$, 
the function $(\rho_{\kappa}^+)_\varepsilon(\omega,t)$ can 
be seen as an element of $C^\infty(M\times\R)$, provided 
that we set it equal to zero outside $X_\kappa\times\R$.

\item $\lim_{\varepsilon\to 0}(\rho_{\kappa}^+)_\varepsilon(\omega,t)
(\cdot,\cdot)=\rho_{\kappa}^+(\omega,t,\cdot,\cdot)$ 
in $\mathcal{D}'(\R^n\times\R)$ for any $(\omega,t)\in\Omega_T$. 
In particular, for $\psi\in \mathcal{D}(X_\kappa\times\R)$ we have
\begin{align*}
\int_{M\times\R}\psi(x,\xi)&\,
\frac{(\rho_{\kappa}^+)_\varepsilon(\omega,t)(x,\xi)}
{\abs{h_\kappa(x)}^{1/2}}\,d\xi\,dV_h(x)\\
&\overset{\varepsilon\downarrow 0}{\longrightarrow} 
\int_{M\times\R}\psi(x,\xi)\,\alpha_\kappa(x) \rho^+(\omega,t,x,\xi)\,d\xi\,dV_h(x), 
\end{align*}
which holds for $\psi\in \mathcal{D}(M\times\R)$ as well, because 
the functions $\frac{(\rho_{\kappa}^+)_\varepsilon}{\abs{h_\kappa}^{1/2}}$ 
and $\alpha_\kappa\rho^+$ are supported in $X_\kappa\times\R$.

\item For any $1\leq p < \infty$ and $(\omega,t)\in\Omega_T$,
\[
\frac{(\rho_{\kappa}^+)_\varepsilon(\omega,t)(\cdot,\cdot)}
{\abs{h_\kappa(\cdot)}^{1/2}}
\overset{\varepsilon\downarrow 0}{\longrightarrow} 
\alpha_\kappa(\cdot) \rho^+(\omega,t,\cdot,\cdot) 
\;\;\;\mbox{in $L^p_{\mathrm{loc}}(M\times\R)$}.  
\]

\end{enumerate}
The listed properties hold true for 
$(\rho_{0,\kappa})_\varepsilon$ as well.
\end{lemma}

\begin{proof}
Claims (1) and (2) follow from standard properties of convolution. 
Claim (3) is an easy consequence of $(2)$ and convergence properties of 
convolution. To show $L^p$-convergence, we argue like this: 
for any $L>0$, using on $X_\kappa$ the coordinates 
given by $\kappa$, we obtain
\begin{align*}
& \int_{M\times [-L,L]}
\abs{\frac{(\rho_{\kappa}^+)_\varepsilon(\omega,t)(x,\xi)}{\abs{h_\kappa(x)}^{1/2}}
-\alpha_\kappa(x) \rho^+(\omega,t,x,\xi)}^p
\,d\xi\,dV_h(x)\\
& \quad =\int_{X_\kappa\times [-L,L]}
\abs{\frac{(\rho_{\kappa}^+)_\varepsilon(\omega,t)(x,\xi)}{\abs{h_\kappa(x)}^{1/2}}
-\alpha_\kappa(x) \rho^+(\omega,t,x,\xi)}^p
\,d\xi\,dV_h(x)\\
& \quad =\int_{\tilde{X}_\kappa \times [-L,L]}
\abs{\frac{(\rho_{\kappa}^+)_\varepsilon(\omega,t)(z,\xi)}{\abs{h_\kappa(z)}^{1/2}}
-\alpha_\kappa(z) \rho^+(\omega,t,z,\xi)}^p\abs{h_\kappa(z)}^{1/2}
\, d\xi\,dz\\
& \quad =\int_{\tilde{X}_\kappa \times [-L,L]}
\abs{(\rho_{\kappa}^+)_\varepsilon(\omega,t)(z,\xi)
-\rho_\kappa^+(\omega,t,z,\xi)}^p\abs{h_\kappa(z)}^{\frac12(1-p)}
\, d\xi\,dz\\
& \quad \leq C(\mathcal{A},h,p)\int_{\tilde{X}_\kappa \times [-L,L]}
\left|(\rho_{\kappa}^+)_\varepsilon(\omega,t)(z,\xi)
-\rho_\kappa^+(\omega,t,z,\xi)\right|^p 
\, d\xi\,dz . 
\end{align*}
The last integral converges to zero as $\varepsilon$ 
goes to zero by standard properties of convolution, 
since $\rho_\kappa^+(\omega,t,\cdot,\cdot)$ is 
in $L^\infty(\R^n_z\times\R_\xi)$. 
\end{proof}

For $(\omega,t)\in\Omega_T$ we define the following 
finite Borel measure on $[0,t]\times M\times\R$:
\[
C^0_b([0,t]\times M\times \R)\ni \phi \mapsto \int_{[0,t]\times M\times\R}
\phi(s,x,\xi)\,\alpha_\kappa(x)\,m(ds,dx,d\xi),
\]
denoted by $(\alpha_\kappa m)(\omega,t)$. 
By definition, $\supp\left((\alpha_\kappa m)(\omega,t)\right)
\subset [0,t]\times\supp(\alpha_\kappa)\times \R$. 
We define its pushforward $(\alpha_\kappa m)_\sharp(\omega,t)$ 
via the homeomorphism 
\[
K:[0,t]\times X_\kappa\times\R 
\to [0,t]\times\tilde{X_k}\times \R, \quad
(s,x,\xi)\mapsto (s,\kappa(x),\xi). 
\]
Hence, its action is given by
\[
C^0_b([0,t]\times\tilde{X}_\kappa \times \R) 
\ni \phi \mapsto \int_{[0,t]\times\tilde{X}_\kappa \times\R}
\phi(s,z,\xi)\,\alpha_\kappa(z)\,m_\sharp(ds,dz,d\xi), 
\]
where $m_\sharp$ is the pushforward of $m$ via $K$. 
With a little abuse of notation, we will also write 
$(\alpha_\kappa m)_\sharp(\omega,t)$ for the finite Borel 
on $\tilde{X}_\kappa \times\R$ determined by 
\[
C^0_b(\tilde{X}_\kappa \times\R)\ni \phi \mapsto 
\int_{[0,t]\times\tilde{X}_\kappa \times\R}\phi(z,\xi)\,\alpha_\kappa(z)
\,m_\sharp(ds,dz,d\xi), 
\]
Consequently, $(\alpha_\kappa m)_\sharp(\omega,t)$ is a finite 
Borel measure on $\tilde{X}_\kappa \times\R$ that is 
supported in $\kappa(\supp \alpha_\kappa) \times\R 
\subset \tilde{X}_\kappa \times\R$, and thus it may be 
naturally viewed as a finite Borel measure on $\R^n\times\R$. 

We regularize $(\alpha_\kappa m)_{\sharp}$ using the 
mollifier $\phi_\varepsilon$: for 
$\omega\in\Omega,t\in[0,T],z\in\R^n,\xi\in\R$ we define
\begin{align*}
((\alpha_\kappa m)_{\sharp})_\varepsilon(\omega,t)(z,\xi)
& := (\alpha_\kappa m)_\sharp(\omega,t)\left( \varepsilon^{-n}
\phi_1\left(\frac{z-\cdot}{\varepsilon}\right) 
\varepsilon^{-1}\phi_2\left(\frac{\xi-\cdot}{\varepsilon}\right)\right)
\\ &=\int_{[0,t]\times\tilde{X}_\kappa \times\R}\phi_\varepsilon(z-\bar{z},\xi
-\bar{\xi})\,\alpha_\kappa(\bar{z})\,m_\sharp(ds,d\bar{z},d\bar{\xi}). 
\end{align*}
The main properties of 
$((\alpha_\kappa m)_{\sharp})_\varepsilon(\omega,t)$ are listed in
\begin{lemma}\label{lemma: properties of alpha m  epsilon}
Let $\kappa\in\mathcal{A}$. Then
\begin{enumerate}

\item $((\alpha_\kappa m)_{\sharp})_\varepsilon(\omega,t)\in 
C^\infty(\R^n_z\times\R_\xi)$ for all $(\omega,t)\in\Omega_T$.

\item for $\varepsilon < \varepsilon_\kappa$ 
and any $\omega\in\Omega,t\in[0,T],\xi\in\R$.
\[
\supp \left(((\alpha_\kappa m)_{\sharp})_\varepsilon(\omega,t)(\cdot,\xi)\right)
\subset \kappa(\supp \alpha_\kappa) +\overline{B_\varepsilon(0)}
\subset\subset \tilde{X}_\kappa. 
\] 
This entails that for fixed $(\omega,t)\in\Omega_T$, the 
function $((\alpha_\kappa m)_{\sharp})_\varepsilon(\omega,t)$ 
can be seen as an element of $C^\infty(M\times\R)$, 
provided we set it to zero outside $X_\kappa\times\R$.

\item $\lim_{\varepsilon\to 0}((\alpha_\kappa m)_{\sharp})_\varepsilon(\omega,t)
=(\alpha_\kappa m)_{\sharp}(\omega,t)$ in the sense of measures for any 
$(\omega,t)\in\Omega_T$. In particular, for $\psi\in C^0_c(X_\kappa\times\R)$,
\begin{align*}
\int_{M\times\R}\psi(x,\xi)
& \,\frac{((\alpha_\kappa m)_{\sharp})_\varepsilon(\omega,t)(x,\xi)}
{\abs{h_\kappa(x)}^{1/2}}\,d\xi\,dV_h(x)\\
&\overset{\varepsilon\downarrow 0}{\longrightarrow} 
\int_{M\times\R}\psi(x,\xi)\,(\alpha_\kappa m)(\omega,t)(dx,d\xi), 
\end{align*}
which is equal to 
\[
\int_{[0,t]\times M\times\R}\psi(x,\xi)\,\alpha_\kappa(x)
\,m(ds,dx,d\xi). 
\]
This result holds for $\psi\in C^0_c(M\times\R)$ as well, 
because the function $\frac{((\alpha_\kappa m)_{\sharp})_\varepsilon}
{\abs{h_\kappa}^{1/2}}$ and the measure 
$(\alpha_\kappa m)(\omega,t)$ are supported in $X_\kappa\times\R$.

\item for any $(\omega,t)\in\Omega_T$ and $\psi\in C^1_c(M\times\R)$,
\begin{align*}
\bigg|\int_{M\times\R}\psi(x,\xi)\partial_\xi 
& \left[\frac{((\alpha_\kappa m)_{\sharp})_\varepsilon(\omega,t)(x,\xi)}
{\abs{h_\kappa(x)}^{1/2}}\right]d\xi\,dV_h(x)\bigg|
\\ & \leq \norm{\partial_\xi\psi}_{L^\infty(M\times\R)}
\int_{[0,t]\times M\times\R}\alpha_\kappa(x)\,m(ds,dx,d\xi). 
\end{align*}

\end{enumerate}
\end{lemma}

\begin{proof}
The proof is identical to the proof of Lemma 
\ref{lemma: properties of rho kappa epsilon}, except for the last point. 
Let $\psi\in C^1_c(X_\kappa\times\R)$. 
Using on $X_\kappa$ the coordinates given by $\kappa$, it follows that
\begin{align*}
-\int_{M\times\R}& \psi (x,\xi)\,\partial_\xi
\left[\frac{((\alpha_\kappa m)_{\sharp})_\varepsilon(\omega,t)(x,\xi)}
{\abs{h_\kappa(x)}^{1/2}}\right]
\, d\xi\,dV_h(x)
\\ & =\int_{M\times\R}\partial_\xi\psi(x,\xi)
\,\frac{((\alpha_\kappa m)_{\sharp})_\varepsilon(\omega,t)(x,\xi)}
{\abs{h_\kappa(x)}^{1/2}}
\,d\xi\,dV_h(x)
\\ & =\int_{\tilde{X}_\kappa \times\R}\partial_\xi\psi(z,\xi)
\,((\alpha_\kappa m)_{\sharp})_\varepsilon(\omega,t)(z,\xi)
\,d\xi\,dz
\\ &=(\alpha_\kappa m)_{\sharp}
(\omega,t\left((\partial_\xi\psi)_\varepsilon\right), 
\end{align*}
where $(\partial_\xi\psi)_\varepsilon$ is the convolution of 
$\partial_\xi\psi$ with $\phi_\varepsilon$. 
By basic estimates for convolution, 
$\norm{(\partial_\xi\psi)_\varepsilon}_{L^\infty(\tilde{X}_\kappa \times\R)}
\leq \norm{\partial_\xi\psi}_{L^\infty(\tilde{X}_\kappa \times\R)}
\leq \norm{\partial_\xi\psi}_{L^\infty(M\times\R)}$, 
where we observe that $\partial_\xi\psi$ can be unambiguously 
defined on all of $M\times\R$. 
Hence, we obtain, by the definition of pushforward,
\begin{align*}
\bigg|\int_{M\times\R}\psi(x,\xi)& \partial_\xi  
\left[\frac{((\alpha_\kappa m)_{\sharp})_\varepsilon(\omega,t)(x,\xi)}
{\abs{h_\kappa(x)}^{1/2}}\right]d\xi\,dV_h(x)\bigg|\\
&\leq \norm{\partial_\xi\psi}_{L^\infty(M\times\R)}
\int_{\tilde{X}_\kappa \times\R}
(\alpha_\kappa m)_{\sharp}(\omega,t)(dz,d\xi)\\
&\leq \norm{\partial_\xi\psi}_{L^\infty(M\times\R)}
\int_{[0,t]\times M\times\R}\alpha_\kappa(x)\,m(ds,dx,d\xi). 
\end{align*}
In view of the compactness of the supports, the last 
estimate holds for any $\psi\in C^1_c(M\times\R)$, and the 
lemma is therefore proved.
\end{proof}

To regularize \eqref{eq: what alfakapparo solves}, we have 
to consider the following map:
\begin{align*}
\nu_\kappa:&\,\Omega\times [0,T]\times\tilde{X}_\kappa 
\to \{\mbox{finite Borel measures on }\R\}, \\
(\nu_\kappa&)_{\omega,t,z}(\cdot):= \alpha_\kappa(z) 
\,\abs{h_\kappa(z)}^{1/2}\,\nu_{\omega,t,\kappa^{-1}(z)}(\cdot).
\end{align*}
Since, if $z\notin \kappa(\supp\alpha_\kappa)$, 
then $(\nu_\kappa)_{\omega,t,z}$ is the null 
measure on $\R$, we can extend $(\nu_\kappa)$ 
on $\Omega\times [0,T]\times\R^n$ in a natural way. This map may 
be transformed into a Radon measure on $\R^n\times\R$ as follows: 
for all $(\omega,t)\in\Omega_T$ and $\psi\in C^0_c(\R^n\times\R)$, set
\begin{align*}
(\nu_\kappa)_{\omega,t}(\psi)
& :=\int_{\R^n}(\nu_\kappa)_{\omega,t,z}(\psi(z,\cdot))\,dz
\\ & = \int_{\kappa(\supp \alpha_\kappa)}\alpha_\kappa(z)
\,\abs{h_\kappa(z)}^{1/2}\left(\int_\R\psi(z,\xi)
\,\nu_{\omega,t,\kappa^{-1}(z)}(d\xi)\right)\, dz. 
\end{align*}
The support of this measure is contained in 
$\kappa(\supp \alpha_\kappa)\times \R$. Once again, we regularize 
$(\nu_\kappa)_{\omega,t}$ using the mollifier $\phi_\varepsilon$.
For $(\omega,t)\in\Omega_T$, $(z,\xi)\in\R^n\times\R$, we set
\[
(\nu_\kappa)_\varepsilon(\omega,t)(z,\xi)
:=(\nu_\kappa)_{\omega,t}\left( \varepsilon^{-n}
\phi_1\left(\frac{z-\cdot}{\varepsilon}\right) 
\varepsilon^{-1}\phi_2\left(\frac{\xi-\cdot}{\varepsilon}\right)\right). 
\]
The main properties of $(\nu_\kappa)_\varepsilon$ are listed in 
\begin{lemma}\label{lemma: properties of nu kappa epsilon}
Let $\kappa\in\mathcal{A}$. Then
\begin{enumerate}

\item $(\nu_\kappa)_\varepsilon(\omega,t)\in C^\infty(\R^n_z\times\R_\xi)$ 
for all $(\omega,t)\in\Omega_T$.
 
\item for $\varepsilon < \varepsilon_\kappa$ and 
any $\omega\in\Omega,t\in[0,T],\xi\in\R$,
\[
\supp \left((\nu_\kappa)_\varepsilon(\omega,t)(\cdot,\xi)\right)
\subset \kappa(\supp \alpha_\kappa) +\overline{B_\varepsilon(0)}
\subset\subset \tilde{X}_\kappa. 
\] 
Hence, for fixed $(\omega,t)\in\Omega_T$, the function 
$(\nu_\kappa)_\varepsilon(\omega,t)$ can be seen as an element of 
$C^\infty(M\times\R)$, setting it to zero outside of $X_\kappa\times\R$.

\item $\lim_{\varepsilon\to 0}(\nu_\kappa)_\varepsilon(\omega,t)
= (\nu_\kappa)_{\omega,t}$ in the sense of measures for any 
$(\omega,t)\in\Omega_T$. In particular, for 
$\psi\in C^0_c(X_\kappa\times\R)$,
\begin{align*}
\int_{M\times\R}\psi(x,\xi) & 
\frac{(\nu_\kappa)_\varepsilon(\omega,t)(x,\xi)}
{\abs{h_\kappa(x)}^{1/2}} \,d\xi\,dV_h(x) 
\\ & \overset{\varepsilon\downarrow 0}{\longrightarrow} 
\int_M\alpha_\kappa(x)\left(\int_\R\psi(x,\xi)
\nu_{\omega,t,x}(d\xi)\right) \, dV_h(x). 
\end{align*}
This result holds for $\psi\in C^0_c(M\times\R)$ as well, because 
the function $\frac{(\nu_\kappa)_\varepsilon}{\abs{h_\kappa}^{1/2}}$ 
and the measure $\alpha_\kappa\nu\, dV_h$ are supported 
in $X_\kappa\times\R$.

\item For a.e.~$(\omega,t)\in\Omega_T$, for all 
$(x,\xi)\in M\times\R$, and for all $\varepsilon<\varepsilon_\kappa$,
\[
\partial_\xi\left[
\frac{(\rho_{\kappa}^+)_\varepsilon(\omega,t)(x,\xi)}{\abs{h_\kappa(x)}^{1/2}}
\right] = -\,\frac{(\nu_{\kappa})_\varepsilon(\omega,t)(x,\xi)}
{\abs{h_\kappa(x)}^{1/2}}, 
\]
where $(\rho^+_\kappa)_\varepsilon$ is given 
by Lemma \ref{lemma: properties of rho kappa epsilon}.

\end{enumerate}
\end{lemma}

\begin{proof}
Only the last claim requires a proof. 
Let $\psi$ be in $\mathcal{D}(X_\kappa\times\R)$. 
By the definition of a generalized kinetic solution, 
for $\P\otimes dt$-a.e.~$(\omega,t)\in\Omega_T$,
\begin{align*}
\int_{M\times\R} & \alpha_\kappa(x)
\rho^+(\omega,t,x,\xi)\,\partial_\xi\psi(x,\xi)
\,d\xi\,dV_h(x)\\
&=\int_M\alpha_\kappa(x)\int_\R\psi(x,\xi)\,
\nu_{\omega,t,x}(d\xi)\,dV_h(x).
\end{align*}
Using the coordinates given by $\kappa$, this means
\begin{align*}
\int_{\kappa(\supp\alpha_\kappa)\times\R} & 
\rho^+_\kappa(\omega,t,z, \xi)\,\partial_\xi\psi(z,\xi)\,d\xi\,dz
\\ & =\int_M\alpha_\kappa(z)\int_\R\psi(z,\xi)
\,\nu_{\omega,t,\kappa^{-1}(z)}(d\xi)
\,\abs{h_\kappa(z)}^{1/2}\,dz. 
\end{align*}
In other words, for all $\psi\in\mathcal{D}(\tilde{X}_\kappa \times\R)$ 
and for $\P\otimes dt$-a.e.~$(\omega,t)\in\Omega_T$,
\[
\int_{\kappa(\supp\alpha_\kappa)\times\R}
\rho^+_\kappa(\omega,t,z,\xi)\,\partial_\xi\psi(z,\xi)
\,dz \,d\xi = (\nu_\kappa)_{\omega,t}(\psi), 
\]
which holds for any $\psi\in\mathcal{D}(\R^n\times\R)$, because 
of the supports of $\rho^+_\kappa$ and $(\nu_\kappa)_{\omega,t}$. 

Therefore, by standard properties of convolution, it follows 
for all $\varepsilon<\varepsilon_\kappa$ and 
for all $(z,\xi)\in\R^n\times\R$ that
\[
\partial_\xi(\rho^+_\kappa)_\varepsilon(\omega,t)(z,\xi)
=-(\nu_\kappa)_\varepsilon(\omega,t)(z,\xi), 
\]
for a.e.~$(\omega,t)\in\Omega_T$. 
This property remains true if we divide both sides 
by $\abs{h_\kappa(z)}^{1/2}$. The claim now follows. 
\end{proof}

As a further step towards the regularization 
of \eqref{eq: what alfakapparo solves}, we need to 
define an additional Radon measure on $\R^n\times\R$:
\[
(G^2\nu_\kappa)_{\omega,t}(\psi):=
\int_{\kappa(\supp \alpha_\kappa)}\int_\R\alpha_\kappa(z)G^2(z,\xi)
\,\abs{h_\kappa(z)}^{1/2}\psi(z,\xi)
\, \nu_{\omega,t,\kappa^{-1}(z)}(d\xi)\,dz, 
\]
with $\psi\in C^0_c(\R^n\times\R)$ and $G^2$ is 
defined in \eqref{eq: definition of G}. 
In view of Lemma \ref{lemma: G^2 is locally Lipschitz }, 
we have $G^2\in C^0(M\times\R)$, and thus, because the 
support of $\alpha_\kappa$ is compact, 
$\alpha_\kappa G^2\abs{h_\kappa}^{1/2}$ 
may be seen as a (continuous) Borel function on $\R^n\times\R$, if 
extended to zero outside of $\tilde{X}_\kappa \times\R$. We 
regularize $(G^2\nu_\kappa)_{\omega,t}$ by convolution. 
For $(\omega,t)\in\Omega_T$, $(z,\xi)\in\R^n\times\R$, we set
\[
(G^2\nu_\kappa)_\varepsilon(\omega,t)(z,\xi)
:=(G^2\nu_\kappa)_{\omega,t}\left( \varepsilon^{-n}
\phi_1\left(\frac{z-\cdot}{\varepsilon}\right) 
\varepsilon^{-1}\phi_2\left(\frac{\xi-\cdot}{\varepsilon}\right)\right). 
\]
The main properties are listed in the following lemma, 
whose proof is identical to those of 
Lemmas \ref{lemma: properties of rho kappa epsilon} 
and \ref{lemma: properties of alpha m  epsilon}.
\begin{lemma}\label{lemma: properties of G^2 nu kappa epsilon}
Let $\kappa\in\mathcal{A}$. Then
\begin{enumerate}

\item $(G^2\nu_\kappa)_\varepsilon(\omega,t)\in C^\infty(\R^n_z\times\R_\xi)$ 
for all $(\omega,t)\in\Omega_T$

\item for $\varepsilon < \varepsilon_\kappa$ and any 
$\omega\in\Omega,t\in[0,T],\xi\in\R$,
\[
\supp (G^2\nu_\kappa)_\varepsilon(\omega,t)(\cdot,\xi)
\subset \kappa(\supp\alpha_\kappa)
+\overline{B_\varepsilon(0)}\subset\subset \tilde{X}_\kappa . 
\] 
Thus, for fixed $(\omega,t)\in\Omega_T$, the function 
$(G^2\nu_\kappa)_\varepsilon(\omega,t)$ can be seen as an element 
of $C^\infty(M\times\R)$, provided it is set to 
zero outside of $X_\kappa\times\R$.

\item $\lim_{\varepsilon\to 0}(G^2\nu_\kappa)_\varepsilon(\omega,t)
= (G^2\nu_\kappa)_{\omega,t}$ in the sense of measures, for 
$(\omega,t)\in\Omega_T$. In particular, 
for $\psi\in C^0_c(X_\kappa\times\R)$,
\begin{align*}
\int_{M\times\R}\psi(x,\xi) & 
\frac{(G^2\nu_\kappa)_\varepsilon(\omega,t)(x,\xi)}
{\abs{h_\kappa(x)}^{1/2}} \,d\xi\,dV_h(x)\\ 
&\overset{\varepsilon\downarrow 0}{\longrightarrow} 
\int_M\alpha_\kappa(x)\left(\int_\R \psi(x,\xi) 
G^2(x,\xi)\,\nu_{\omega,t,x}(d\xi)\right)\, dV_h(x). 
\end{align*}
This result holds for any $\psi\in C^0_c(M\times\R)$, since the 
function $\frac{(G^2\nu_\kappa)_\varepsilon}{\abs{h_\kappa}^{1/2}}$ and the 
measure $\alpha_\kappa G^2\,\nu\, dV_h$ are supported in $X_\kappa\times\R$.

\item for any $(\omega,t)\in\Omega_T$ and $\psi\in C^1_c(M\times\R)$,
\begin{align*}
\bigg|\int_{M\times\R}&\psi(x,\xi)
\partial_\xi\left[\frac{(G^2\nu_\kappa)_\varepsilon(\omega,t)(x,\xi)}
{\abs{h_\kappa(x)}^{1/2}}\right]\,d\xi\,dV_h(x)\bigg|
\\ & \leq \norm{\partial_\xi\psi}_{L^\infty(M\times\R)}
\int_M\alpha_\kappa(x)\left(\int_\R 
G^2(x,\xi)\,\nu_{\omega,t,x}(d\xi)\right)\,dV_h(x). 
\end{align*}

\end{enumerate}
\end{lemma}

For $(\omega,t)\in\Omega_T$, consider the function 
$A_\kappa(\omega,t)(\cdot,\cdot)$ defined as
\[
A_\kappa(\omega,t)(z,\xi):=
\begin{cases}
\rho^+(\omega,t,z,\xi)
\left(f'_z(\xi),\nabla\alpha_\kappa\right)_h
\,\abs{h_\kappa(z)}^{1/2}, 
& (z,\xi)\in \tilde{X}_\kappa \times\R, \\
0, & (z,\xi)\notin \tilde{X}_\kappa \times\R, \\
\end{cases}
\]
which we regularize using the mollifier $\phi_\varepsilon$. 
\begin{lemma}\label{lemma: properties of A kappa epsilon}
Let $\kappa\in\mathcal{A}$. Then
\begin{enumerate}

\item $(A_{\kappa})_\varepsilon(\omega,t)\in C^\infty(\R^n_z\times\R_\xi)$, 
for all $(\omega,t)\in\Omega_T$.

\item for $\varepsilon < \varepsilon_\kappa$ and for 
any $\omega\in\Omega,t\in[0,T],\xi\in\R$,
\[
\supp (A_{\kappa})_\varepsilon(\omega,t)(\cdot,\xi) 
\subset \kappa(\supp\alpha_\kappa) +\overline{B_\varepsilon(0)}
\subset\subset \tilde{X}_\kappa. 
\] 
This implies in particular that for any $(\omega,t)\in\Omega_T$, 
the function $(A_{\kappa})_\varepsilon(\omega,t)$ can be 
seen as an element of $C^\infty(M\times\R)$, provided 
that we set it equal to zero outside of $X_\kappa\times\R$.

\item $\lim_{\varepsilon\to 0}(A_{\kappa})_\varepsilon(\omega,t)
= A_{\kappa}(\omega,t)(\cdot,\cdot)$ in $\mathcal{D}'(\R^n\times\R)$ 
for any $(\omega,t)\in\Omega_T$. In particular, 
for $\psi\in \mathcal{D}(X_\kappa\times\R)$, we have
\begin{align*}
\int_{M\times\R}\psi(x,\xi) & 
\,\frac{(A_{\kappa})_\varepsilon(\omega,t)(x,\xi)}
{\abs{h_\kappa(x)}^{1/2}}\,d\xi\,dV_h(x)
\\ & \overset{\varepsilon\downarrow 0}{\longrightarrow}
\int_{M\times\R}\psi(x,\xi)\,\rho^+(\omega,t,x,\xi)(f'_x(\xi)
\nabla\alpha_\kappa)_h\,d\xi\,dV_h(x). 
\end{align*}
This result holds for any $\psi\in \mathcal{D}(M\times\R)$ 
as well, since the functions involved are supported in $X_\kappa\times\R$.

\item For any $1\leq p < \infty$ and $(\omega,t)\in\Omega_T$,
\[
\frac{(A_{\kappa})_\varepsilon(\omega,t)(\cdot,\cdot)}
{\abs{h_\kappa(\cdot)}^{1/2}}
\overset{\varepsilon\downarrow 0}{\longrightarrow} 
\rho^+(\omega,t,\cdot,\cdot)
(f'_\cdot(\cdot),\nabla\alpha_\kappa)_h \;\;\;
\mbox{in $L^p_{\mathrm{loc}}(M\times\R)$}. 
\]

\end{enumerate}
\end{lemma}

\begin{proof}
Only claim $(4)$ requires a proof, as the other claims are obvious. 
From the very definition of $A_\kappa(\omega,t)$, 
it follows, for $(z,\xi)\in \tilde{X}_\kappa \times\R$,
\begin{align*}
\abs{A_\kappa(\omega,t)(z,\xi)} &=
\abs{\rho^+(\omega,t,z,\xi)(f'_z(\xi),
\nabla\alpha_\kappa)_h\,\abs{h_\kappa(z)}^{1/2}}
\\ & \leq C(M,h,\mathcal{A})\, 
\abs{(f'_z(\xi))^l\,\partial_{z^l}\alpha_\kappa(z)}, 
\end{align*}
where $l$ sums over $1$ to $n$. Thanks to \eqref{eq: growth condition 0}, 
we infer that $A_\kappa(\omega,t)\in L^p_{\mathrm{loc}}(\R^n\times\R)$. 
Hence, arguing as in the proof of 
Lemma \ref{lemma: properties of rho kappa epsilon}, 
we arrive at claim $(4)$.  
\end{proof}

We eventually need to define a finite signed Borel 
measure on $\R^n\times\R$ denoted by $(g_k\nu_\kappa)_{\omega,t},k\in\N$, 
where $g_k\in C^0(M\times\R)$ (see equation \eqref{eq: definition of G}). 
We proceed like this: given $\psi\in C^0_b(\R^n\times\R)$, 
after extending $\alpha_\kappa g_k\,\abs{h_\kappa}^{1/2}$ by 
zero outside of $\tilde{X}_\kappa \times\R$, we estimate 
the following quantity:
\begin{align*}
\int_{\kappa(\supp \alpha_\kappa)} 
&  \int_\R\alpha_\kappa(z)|g_k(z,\xi)|
\,\abs{h_\kappa(z)}^{1/2}|\psi(z,\xi)|
\,\nu_{\omega,t,\kappa^{-1}(z)}(d\xi)\,dz
\\ \leq \norm{\psi}_{L^\infty} & \sqrt{D_1}\int_M\alpha_\kappa(x)
\int_\R \left(1+\abs{\xi}\right)\,\nu_{\omega,t,x}(d\xi)\,dV_h(x)
\\ \leq \norm{\psi}_{L^\infty} &\sqrt{D_1}
\int_M\int_\R \left(1+\abs{\xi}\right)
\,\nu_{\omega,t,x}(d\xi)\,dV_h(x), 
\end{align*}
where we have used \eqref{eq: definition of G}. 
Hence, by the ``vanishing at infinity" assumption 
\eqref{def: Young measure}, we see that the 
last quantity is finite for all $(\omega,t)\in F$, where 
$F\in\mathcal{P}$, $(\P\otimes dt)(F^c)=0$. 
In view of this, we define the finite signed 
Borel measure $(g_k\nu_\kappa)_{\omega,t}$ as the 
one given by the continuous linear functional
\[
C^0_b(\R^n\times\R)\ni\psi\mapsto \int_{\R^n}
\int_\R\alpha_\kappa(z)g_k(z,\xi)\,\abs{h_\kappa(z)}^{1/2}\psi(z,\xi)
\,\nu_{\omega,t,\kappa^{-1}(z)}(d\xi)\,dz, 
\]
if $(\omega,t)\in F$; otherwise, $(g_k\nu_\kappa)_{\omega,t}$ is set 
to be the null measure on $\R^n\times\R$. 
Observe that the supports of these signed measures 
are contained in $\kappa(\mbox{supp }\alpha_\kappa)\times\R$. 

We regularize these measures with the mollifier $\phi_\varepsilon$. 
For $(\omega,t)\in\Omega_T$, $(z,\xi)\in\R^n\times\R$, and $k\in\N$, set
\[
(g_k\nu_\kappa)_\varepsilon(\omega,t)(z,\xi)
:=(g_k\nu_\kappa)_{\omega,t}\left( \varepsilon^{-n}
\phi_1\left(\frac{z-\cdot}{\varepsilon}\right) 
\varepsilon^{-1}\phi_2\left(\frac{\xi-\cdot}{\varepsilon}\right)\right). 
\]
The following result, whose proof is now obvious, holds:  
\begin{lemma}\label{lemma: properties of gk nu kappa epsilon}
Let $\kappa\in\mathcal{A}$. Then
\begin{enumerate}

\item $(g_k\nu_\kappa)_\varepsilon(\omega,t)\in C^\infty(\R^n_z\times\R_\xi)$ 
for all $(\omega,t)\in\Omega_T$.

\item for $\varepsilon < \varepsilon_\kappa$ and 
any $\omega\in\Omega,t\in[0,T],\xi\in\R$,
\[
\supp \left((g_k\nu_\kappa)_\varepsilon(\omega,t)(\cdot,\xi)\right)
\subset \kappa(\supp \alpha_\kappa)
+\overline{B_\varepsilon(0)}\subset\subset \tilde{X}_\kappa. 
\] 
This entails that for fixed $\omega\in\Omega,t\in[0,T]$ 
the function $(g_k\nu_\kappa)_\varepsilon(\omega,t)$ 
can be seen as an element of $C^\infty(M\times\R)$, provided 
that it is set equal to zero outside of $\tilde{X}_\kappa \times\R$.

\item $\lim_{\varepsilon\to 0}(g_k\nu_\kappa)_\varepsilon(\omega,t)= 
(g_k\nu_\kappa)_{\omega,t}$ in the sense of measures for any $(\omega,t)
\in\Omega_T$. In particular, for $\psi\in C^0_c(X_\kappa\times\R)$,
\begin{align*}
\int_{M\times\R}\psi(x,\xi)& 
\frac{(g_k\nu_\kappa)_\varepsilon(\omega,t)(x,\xi)}
{\abs{h_\kappa(x)}^{1/2}} \,d\xi\,dV_h(x)\\ 
&\overset{\varepsilon\downarrow 0}{\longrightarrow} \int_M\alpha_\kappa(x)
\left(\int_\R\psi(x,\xi)g_k(x,\xi)\nu_{\omega,t,x}(d\xi)\right)dV_h(x). 
\end{align*}
This result holds for $\psi\in C^0_c(M\times\R)$ as well, because the 
function $\frac{(g_k\nu_\kappa)_\varepsilon}{\abs{h_\kappa}^{1/2}}$ and 
the signed measure $\alpha_\kappa g_k\,\nu\, dV_h$ 
are supported in $X_\kappa\times\R$.

\end{enumerate}
\end{lemma}

In view of these results, \eqref{eq: what alfakapparo solves} 
with $\psi\in C^1_c(X_\kappa\times\R)$ can be written as
\begin{equation}\label{eq: what alfakapparo solves (2)}
\begin{split}
-&\int_{\tilde{X}_\kappa \times\R}\rho^+_\kappa(t)\,\psi\, d\bar{\xi}\,d\bar{z}+
\int_{\tilde{X}_\kappa \times\R}\rho_{0,\kappa}\,\psi\, d\bar{\xi}\,d\bar{z}\\
&\hspace{0.5cm}+\int_0^t\int_{\tilde{X}_\kappa \times\R}\rho^+_\kappa(s)(f'_{\bar{z}}
(\bar{\xi}))^l\partial_{\bar{z}^l} \psi\, d\bar{\xi}\,d\bar{z}\, ds\\
&=-\sum_{k\geq 1}\int_0^t\int_{\tilde{X}_\kappa \times\R}\psi\, 
(g_k\nu_\kappa)_{\omega,s}(d\bar{z},d\bar{\xi})\,d\beta_k(s)\\
&\hspace{0.5cm}-\frac12\int_0^t\int_{\tilde{X}_\kappa \times\R}\partial_{\bar{\xi}}\psi 
\,(G^2\nu_\kappa)_{\omega,s}(d\bar{z},d\bar{\xi})\,ds\\
&\hspace{0.5cm}+\int_{\tilde{X}_\kappa \times\R}
\partial_{\bar{\xi}}\psi\,(\alpha_\kappa m)_\sharp(\omega,t)(d\bar{z},d\bar{\xi}) \\
&\hspace{0.5cm}-\int_0^t\int_{\tilde{X}_\kappa \times\R}
\psi A_\kappa (s) d\bar{z}\, d\bar{\xi}\,ds, 
\qquad \text{$\P$-almost surely},
\end{split}
\end{equation}
where $(f'_{\cdot}(\cdot))^l$ is the $l$th component of 
$f'$ in the local coordinates given by $\kappa$.

We set for convenience $\mathcal{U}_\kappa:=
\kappa(\supp \alpha_\kappa) +B_{\varepsilon_\kappa/2}(0)$ 
and from now on we will 
consider only $\varepsilon<\varepsilon_\kappa/4$.
Let us introduce the following family of test functions 
$\psi_{z,\xi,\varepsilon}$ (parametrized by $z\in\tilde{X}_\kappa$, 
$\xi\in\R$, and $\varepsilon<\varepsilon_\kappa/4$):
\[
\psi_{z,\xi,\varepsilon}(\cdot,\cdot):=
\begin{cases}
\phi_\varepsilon(z-\cdot,\xi-\cdot),  
& \mbox{if } 
B_\varepsilon(z)\cap \partial\tilde{X}_\kappa =\emptyset
\\ 0,  & \mbox{otherwise} 
\end{cases}.
\]
We observe that these functions may be seen as elements of 
$C^1_c(X_\kappa\times\R)$ and that, for fixed $\omega\in\Omega$, $t\in[0,T]$, 
and $\xi\in\R$,
\[
\supp \left((\cdots)_\varepsilon(\omega,t)(\cdot,\xi)\right)
\subset \kappa(\supp \alpha_\kappa) +\overline{B_\varepsilon(0)}
\subset\subset \mathcal{U}_\kappa,
\] 
where $\cdots$ is any of the objects defined above. 
Moreover, for $\omega\in\Omega$, $t\in[0,T]$, $\xi\in\R$, and 
$z\in\tilde{X}_\kappa \setminus\mathcal{U}_\kappa$, we have that 
$(\cdots)_\varepsilon(\omega,t)(z,\xi)$ is equal to zero and 
it coincides with the action of $(\cdots)(\omega,t)$ on the 
function $\psi_{z,\xi,\varepsilon}$.

We make use of $\psi_{z,\xi,\varepsilon}$ as test function in
\eqref{eq: what alfakapparo solves (2)}, resulting in the equation
\begin{equation}\label{eq: regolarizzata 1}
\begin{split}
-( & \rho^+_{\kappa})_\varepsilon(\omega,t)(z,\xi)
+(\rho_{0,\kappa})_\varepsilon(\omega)(z,\xi)
\\ &\hspace{0.8cm}
-\int_0^t\int_{\tilde{X}_\kappa \times\R}
\rho^+_\kappa(s)(f'_{\bar{z}}(\bar{\xi}))^l
\left(\partial_l\phi_\varepsilon\right)(z-\bar{z},\xi-\bar{\xi})
\, d\bar{\xi}\,d\bar{z}\, ds
\\ &
=-\sum_{k\geq 1}\int_0^t (g_k\nu_\kappa)_\varepsilon(\omega,s)(z,\xi)
\,d\beta_k(s)
\\ &\hspace{0.8cm}
+\frac12\int_0^t\partial_\xi
\left[(G^2\nu_\kappa)_\varepsilon(\omega,s)(z,\xi)\right]\, ds
-\partial_\xi[((\alpha_\kappa m)_{\sharp})_\varepsilon(\omega,t)(z,\xi)] 
\\ & \hspace{0.8cm}
-\int_0^t (A_{\kappa})_\varepsilon(\omega,s)(z,\xi)\, ds, 
\quad \text{$\P$-almost surely},
\end{split}
\end{equation}
valid for $\varepsilon<\varepsilon_{\kappa}/4$, $\xi\in\R$, 
and $z\in\tilde{X}_\kappa $.

For the transport term involving $f'$ we need the following 
following version of the commutator lemma due to 
DiPerna and Lions \cite{DL89}: 

\begin{lemma}[DiPerna-Lions commutator estimate]\label{lemma: commutator}
Fix $\kappa\in\mathcal{A}$, and let $\mathcal{E}$ be the 
smooth vector field on $\tilde{X}_\kappa \times\R$ defined by
$\mathcal{E}(\bar{z},\bar{\xi})=
\left( (f'_{\bar{z}}(\bar{\xi}))^1, \ldots, 
(f'_{\bar{z}}(\bar{\xi}))^n,0\right)$.
Set 
\[
r_{\kappa,\varepsilon}(\omega,s)(z,\xi):=\dive_{\R^{n+1}}
\left(\rho^+_\kappa(\omega,s)\mathcal{E}\right)_\varepsilon(z,\xi)
-\dive_{\R^{n+1}}\left((\rho^+_{\kappa})_\varepsilon(\omega,s)
\mathcal{E}\right)(z,\xi), 
\]
with $(\omega,s)\in\Omega_T$, $z\in\tilde{X}_\kappa$, and $\xi\in\R$. 
Then $r_{\kappa,\varepsilon}\to 0$ in 
$L^1(\Omega_T\times\tilde{X}_\kappa \times(-L,L))$ 
as $\varepsilon\to 0$, for any $L>0$. Furthermore, for 
$\varepsilon<\varepsilon_\kappa$ and for any $(\omega,s)\in\Omega_T$, the 
function $r_{\kappa,\varepsilon}(\omega,s)$ 
can be seen as an element of $C^\infty(M\times\R)$, provided it is set to 
zero outside of $X_\kappa\times\R$.
\end{lemma}

\begin{proof}
By previous observations, for any $(\omega,s)\in\Omega_T$ 
and $\xi\in\R$, the quantity $r_{\kappa,\varepsilon}(\omega,s)(\cdot,\xi)$ 
is zero on the set $\tilde{X}_\kappa \setminus \mathcal{U}_\kappa$. 
Thus, it is sufficient to examine the points 
$z\in \mathcal{U}_\kappa$ only. By setting 
$w:=(z,\xi)$ and $\bar{w}:=(\bar{z},\bar{\xi})$, 
$r_{\kappa,\varepsilon}$ may be rewritten as
\begin{align*}
r_{\kappa,\varepsilon}&(\omega,s)(z,\xi)
\\ & =\sum_{i=1}^{n+1}
\int_{\substack{\abs{\bar{z}}<\varepsilon\\ \abs{\bar{\xi}}<\varepsilon}}
\rho^+_\kappa(\omega,s,w-\bar{w})
\mathcal{E}_i(w-\bar{w})(\partial_{i} \phi_\varepsilon)(\bar{w})\, d\bar{w}
\\ &\hspace{1.8cm} 
- \left(\nabla_{\R^{n+1}}(\rho^+_\kappa)_\varepsilon\cdot \mathcal{E}\right)(w)
-(\rho^+_\kappa)_\varepsilon(\omega,s)(w)\dive_{\R^{n+1}}\mathcal{E}(w)
\\ & =\sum_{i=1}^{n+1}
\int_{\substack{\abs{\bar{z}}<\varepsilon\\\abs{\bar{\xi}}<\varepsilon}}
\rho^+_\kappa(\omega,s,w-\bar{w})(\partial_{i} \phi_\varepsilon)(\bar{w}) 
\left[\mathcal{E}_i(w-\bar{w})-\mathcal{E}_i(w)\right]\, d\bar{w} 
\\ & \hspace{1.8cm}
-(\rho^+_\kappa)_\varepsilon(\omega,s)(w)
\dive_{\R^{n+1}}\mathcal{E}(w)
\\ & =\sum_{i,j=1}^{n+1} \int_0^1
\int_{\substack{\abs{\bar{z}}<\varepsilon\\\abs{\bar{\xi}}<\varepsilon}}
\left[\partial_j\mathcal{E}_i(w)-\partial_j\mathcal{E}_i(w-\tau\bar{w})\right]
\rho^+_\kappa(\omega,s,w-\bar{w})\bar{w}_j(\partial_{i} 
\phi_\varepsilon)(\bar{w}) \, d\bar{w}\,d\tau 
\\ &\hspace{0.8cm}
-\sum_{i,j=1}^{n+1}\partial_j\mathcal{E}_i(w)
\int_{\substack{\abs{\bar{z}}<\varepsilon\\\abs{\bar{\xi}}<\varepsilon}}
\rho^+_\kappa(\omega,s,w-\bar{w})
\bar{w}_j(\partial_{i} \phi_\varepsilon)(\bar{w}) \, d\bar{w} 
\\ &\hspace{4.2cm}
-(\rho^+_\kappa)_\varepsilon(\omega,s)(w)\dive_{\R^{n+1}}\mathcal{E}(w)
\\ &=: a_\varepsilon(\omega,s;z,\xi)+b_\varepsilon(\omega,s;z,\xi)
-(\rho^+_\kappa)_\varepsilon(\omega,s)(w)\dive_{\R^{n+1}}\mathcal{E}(w). 
\end{align*}
Note that for the vector field $\mathcal{E}$ the indices 
are put in a low position, so that the Einstein 
summation convention does not apply. Let us deal with the term 
$a_\varepsilon$ first. For $L>0$ we have
\begin{align*}
\EE&\int_0^T\int_{\mathcal{U}_\kappa\times(-L,L)}
\abs{a_\varepsilon} \,dw \,ds
\\ & \leq\sum_{i,j=1}^{n+1}\EE\int_0^1\int_0^T \, 
\int_{\mathcal{U}_\kappa\times(-L,L)}  \\
&\hspace{0.75cm}\times
\int_{\substack{\abs{\bar{z}}<\varepsilon\\ \abs{\bar{\xi}}<\varepsilon}}
\abs{\partial_j\mathcal{E}_i(w)
-\partial_j\mathcal{E}_i(w-\tau\bar{w})} \,C(M,h,\mathcal{A})\,
\abs{\bar{w}_j(\partial_{i} \phi_\varepsilon)(\bar{w})} 
\, d\bar{w}\, dw\,ds\,d\tau \\ & \leq C(M,h,\mathcal{A})\sum_{i,j=1}^{n+1}
\EE\int_0^1\int_0^T\, 
\int_{\substack{\abs{\bar{z}}<\varepsilon\\ \abs{\bar{\xi}}<\varepsilon}} \\ &\hspace{0.75cm}\times
\abs{\bar{w}_j(\partial_{i} \phi_\varepsilon)(\bar{w})}\,
\norm{\partial_j\mathcal{E}_i(\cdot)-\partial_j
\mathcal{E}_i(\cdot-\tau\bar{w})}_{L^1(\mathcal{U}_\kappa\times(-L,L))} 
\, d\bar{w}\,ds\,d\tau \\ & \leq C(M,h,\mathcal{A})\,\cdot T\sum_{i,j=1}^{n+1}
\int_0^1 \norm{\bar{w}_j(\partial_{i} \phi_\varepsilon)(\cdot)}_{L^1(\R^{n+1})} 
\\ &\hspace{0.75cm} \times 
\sup_{\substack{\abs{\bar{z}}<\varepsilon\\ \abs{\bar{\xi}}<\varepsilon}}
\norm{\partial_j\mathcal{E}_i(\cdot)
-\partial_j\mathcal{E}_i(\cdot-\tau\bar{w})}_{L^1(\mathcal{U}_\kappa\times(-L,L))}
\, d\tau \\ &\leq C \sum_{i,j=1}^{n+1} 
\sup_{\substack{\abs{\bar{z}}<\varepsilon\\ 
\abs{\bar{\xi}}<\varepsilon}}
\norm{\partial_j \mathcal{E}_i(\cdot)
-\partial_j\mathcal{E}_i(\cdot-\bar{w})}_{L^1(\mathcal{U}_\kappa\times(-L,L))}
\overset{\varepsilon\downarrow 0}{\longrightarrow} 0,
\end{align*}
where $C(M,h,\mathcal{A})$ is a positive constant 
depending on $M,h$, and $\mathcal{A}$. Hence, we have the 
convergence $a_\varepsilon\to 0$ in 
$L^1(\Omega_T\times\tilde{X}_\kappa \times(-L,L))$.

For the term $b_\varepsilon$, by following the proof 
of \cite[Lemma II.1]{DL89}, we immediately infer 
that $b_\varepsilon\to  \rho^+_\kappa\dive_{\R^{n+1}}\mathcal{E}$ 
in $L^1(\Omega_T\times\tilde{X}_\kappa \times(-L,L))$ 
as $\varepsilon\to 0$.  

It remains to deal with the term 
$-(\rho^+_\kappa)_\varepsilon(\omega,s)(w)\dive_{\R^{n+1}}\mathcal{E}(w)$. 
As $(\rho^+_\kappa)_\varepsilon(\omega,s)(\cdot,\cdot)$ belongs to 
$L^\infty(\R^n\times\R)$ for any $(\omega,s)\in\Omega_T$, it 
follows that 
\[
\norm{(\rho^+_\kappa)_\varepsilon(\omega,s)(\cdot,\cdot)
-\rho^+_\kappa(\omega,s,\cdot,\cdot)}_{L^1(\tilde{X}_\kappa \times(-L,L))}
\overset{\varepsilon\downarrow 0}{\longrightarrow} 0. 
\]
Furthermore, by basic $L^p$ convolution estimates,
\begin{align*}
\norm{(\rho^+_\kappa)_\varepsilon(\omega,s)(\cdot,\cdot)}_{L^1(\tilde{X}_\kappa 
\times(-L,L))} &\leq \mathcal{L}^{n+1}(\tilde{X}_\kappa \times(-L,L)) \,
\norm{(\rho^+_\kappa)_\varepsilon(\omega,s)(\cdot,\cdot)}_{L^\infty(\R^n\times\R)}
\\ & \leq \mathcal{L}^{n+1}(\tilde{X}_\kappa \times(-L,L))\,
\norm{\rho^+_\kappa(\omega,s,\cdot,\cdot)}_{L^\infty(\R^n\times\R)}\\
&\leq C(M,h,\mathcal{A},L), 
\end{align*}
uniformly in $(\omega,s)\in\Omega_T$. 
Thus, by H\"older's inequality and smoothness of $\mathcal{E}$,
\begin{align*}
\EE&\int_0^T\int_{\tilde{X}_\kappa \times(-L,L)}
\abs{\dive_{\R^{n+1}}\mathcal{E}(z,\xi)}
\,\abs{(\rho^+_\kappa)_\varepsilon(\omega,s)(z,\xi)
-\rho^+_\kappa(\omega,s,z,\xi)} \,d\xi\,dz\,ds 
\\ & \leq \EE\int_0^T 
\norm{\dive_{\R^{n+1}}\mathcal{E}}_{L^\infty(\tilde{X}_\kappa \times(-L,L))} 
\, \norm{(\rho^+_\kappa)_\varepsilon(\omega,s)
-\rho^+_\kappa(\omega,s,\cdot,\cdot)}_{L^1(\tilde{X}_\kappa \times(-L,L))}
\,ds\\ &\leq C\,\EE\int_0^T\, \norm{(\rho^+_\kappa)_\varepsilon(\omega,s)
-\rho^+_\kappa(\omega,s,\cdot,\cdot)}_{L^1(\tilde{X}_\kappa \times(-L,L))}\,ds 
\overset{\varepsilon\downarrow 0}{\longrightarrow} 0, 
\end{align*}
by means of dominated convergence. We conclude that 
$$
-(\rho^+_\kappa)_\varepsilon\dive_{\R^{n+1}}
\mathcal{E}\overset{\varepsilon\downarrow 0}{\longrightarrow}
 -\rho^+_\kappa\dive_{\R^{n+1}}\mathcal{E}
\quad \text{in $L^1(\Omega_T\times\tilde{X}_\kappa \times(-L,L))$},
$$ 
and thus the desired convergence of 
$r_{\kappa,\varepsilon}$ to zero follows. 

The remaining statements are now obvious and 
thus the lemma is proved.
\end{proof}

In view of Lemma \ref{lemma: commutator}, equation 
\eqref{eq: regolarizzata 1} becomes
\begin{equation}\label{eq: regolarizzata 2}
\begin{split}
-&(\rho^+_{\kappa})_\varepsilon(\omega,t)(z,\xi)
+(\rho_{0,\kappa})_\varepsilon(\omega)(z,\xi)
\\ &\hspace{0.6cm}
-\int_0^t\dive_{\R^n}((\rho^+_{\kappa})_\varepsilon(\omega,s)(z,\xi)
f'_{z}(\xi))\,ds
\\ & =-\sum_{k\geq 1}\int_0^t 
(g_k\nu_\kappa)_\varepsilon(\omega,s)(z,\xi)\,d\beta_k(s)\\
&\hspace{0.6cm}
+\frac12\int_0^t\partial_\xi
\left[(G^2\nu_\kappa)_\varepsilon(\omega,s)(z,\xi)\right]\, ds
-\partial_\xi[((\alpha_\kappa m)_{\sharp})_\varepsilon(\omega,t)(z,\xi)] 
\\ &\hspace{0.6cm}
-\int_0^t  (A_{\kappa})_\varepsilon (\omega,s)(z,\xi)\, ds
+\int_0^tr_{\kappa,\varepsilon}(\omega,s)(z,\xi)\,ds, 
\quad \text{$\P$-almost surely},
\end{split}
\end{equation}
valid for $\varepsilon<\varepsilon_{\kappa}/4$, $\xi\in\R$, 
and $z\in\tilde{X}_\kappa $. 

We recall that, for fixed $(\omega,s)\in\Omega_T$ and $\xi\in\R$, 
the smooth vector field 
\[
\tilde{X}_\kappa \ni z 
\mapsto (\rho^+_{\kappa})_\varepsilon(\omega,s)(z,\xi)f'_{z}(\xi)
\in\R^n
\]  
is compactly supported in $\kappa(\supp \alpha_\kappa) 
+\overline{B_\varepsilon(0)}\subset\subset \tilde{X}_\kappa $. 
The same property must be enjoyed by the smooth vector field   
\[
\tilde{X}_\kappa \ni z \mapsto 
\frac{(\rho^+_{\kappa})_\varepsilon(\omega,s)(z,\xi)
f'_{z}(\xi)}{\abs{h_\kappa(z)}^{1/2}}\in\R^n. 
\]
Therefore, for any $(\omega,s)\in\Omega_T$ and $\xi\in\R$, we 
pushforward this vector field to $X_\kappa$ via 
the diffeomorphism $\kappa^{-1}:\tilde{X}_\kappa \to X_\kappa$. 
Because of the compactness of its support, this new 
vector field may be seen as a global smooth vector field on $M$, 
if extended to zero outside of its support (cf.~\cite[Prop.~4.10]{LeeSmooth}): 
\[
M\ni x \mapsto 
\frac{(\rho^+_{\kappa})_\varepsilon(\omega,s)(x,\xi)f'_{x}(\xi)}
{\abs{h_\kappa(x)}^{1/2}}\in TM. 
\]
Moreover, in the coordinates given by $\kappa$, it holds 
\begin{align*}
\Div\left( \frac{(\rho^+_{\kappa})_\varepsilon(\omega,s)(x,\xi)
f'_{x}(\xi)}{\abs{h_\kappa(x)}^{1/2}}\right)
& = \frac{1}{\abs{h_\kappa(z)}^{1/2}}\,\partial_{z^l}
\left(\abs{h_\kappa(z)}^{1/2}
\frac{(\rho^+_{\kappa})_\varepsilon(\omega,s)(z,\xi)(f'_{z}(\xi))^l}
{\abs{h_\kappa(z)}^{1/2}}\right)
\\ &=\frac{1}{\abs{h_\kappa(z)}^{1/2}}\,\dive_{\R^n}
\left((\rho^+_{\kappa})_\varepsilon(\omega,s)(z,\xi)f'_{z}(\xi)\right) 
\;\;\; \mbox{in } X_\kappa,
\end{align*}
from the very definition of $\Div$.

We divide \eqref{eq: regolarizzata 2} by $\abs{h_\kappa(z)}^{1/2}$ and 
then ``pull back" to $M$. The result is
\begin{equation}\label{eq: regolarizzata 3}
\begin{split}
-&\frac{(\rho^+_{\kappa})_\varepsilon(\omega,t)(x,\xi)}{\abs{h_\kappa(x)}^{1/2}}
+\frac{(\rho_{0,\kappa})_\varepsilon(\omega)(x,\xi)}{\abs{h_\kappa(x)}^{1/2}}
\\ &\hspace{0.6cm}-\int_0^t
\Div\left( \frac{(\rho^+_{\kappa})_\varepsilon(\omega,s)(x,\xi)f'_{x}(\xi)}
{\abs{h_\kappa(x)}^{1/2}}\right)\,ds
\\ &=-\sum_{k\geq 1}\int_0^t \frac{(g_k\nu_\kappa)_\varepsilon(\omega,s)(x,\xi)}
{\abs{h_\kappa(x)}^{1/2}}\,d\beta_k(s)\\
&\hspace{0.6cm} +\frac12\int_0^t\partial_\xi
\left[\frac{(G^2\nu_\kappa)_\varepsilon(\omega,s)(x,\xi)}
{\abs{h_\kappa(x)}^{1/2}}\right]\, ds
-\partial_\xi\left[
\frac{((\alpha_\kappa m)_{\sharp})_\varepsilon(\omega,t)(x,\xi)}
{\abs{h_\kappa(x)}^{1/2}}\right] 
\\ &\hspace{0.6cm}
-\int_0^t  \frac{(A_{\kappa})_\varepsilon 
(\omega,s)(x,\xi)}{\abs{h_\kappa(x)}^{1/2}}\, ds 
+\int_0^t\frac{r_{\kappa,\varepsilon}
(\omega,s)(x,\xi)}{\abs{h_\kappa(x)}^{1/2}}\,ds, 
\quad \text{$\P$-almost surely},
\end{split}
\end{equation}
which is valid for $\varepsilon<\varepsilon_{\kappa}/4$, $\xi\in\R$, 
$x\in M$, and $t\in[0,T]$. In some of the terms we have used the relation 
$\partial_\xi(\cdots)\abs{h_\kappa(x)}^{-1/2}
=\partial_\xi(\cdots \abs{h_\kappa(x)}^{-1/2})$.

In the proof of Lemma \ref{lemma: simmetry argument} (for example) 
we will have to work with certain Radon measures associated 
to $g^2_k$, which we define now. For all $k\in\N$, set 
\[
(g^2_k\nu_\kappa)_{\omega,t}(\psi):= \int_{\kappa(\supp \alpha_\kappa)}
\int_\R\alpha_\kappa(z)g^2_k(z,\xi)\,\abs{h_\kappa(z)}^{1/2}
\psi(z,\xi)\,\nu_{\omega,t,\kappa^{-1}(z)}(d\xi)\,dz, 
\]
with $\psi\in C^0_c(\R^n\times\R)$. 
By assumption $g^2_k\in C^0(M\times\R)$ and because 
the support of $\alpha_\kappa$ is compact, 
$\alpha_\kappa g^2_k\abs{h_\kappa}^{1/2}$ may 
be seen as a (continuous) Borel function on $\R^n\times\R$, 
if extended to zero outside of $\tilde{X}_\kappa \times\R$.

As always, we regularize these measures using a mollifier. 
For $(\omega,t)\in\Omega_T$ and $(z,\xi)\in\R^n\times\R$, set
\[
(g^2_k\nu_\kappa)_\varepsilon(\omega,t)(z,\xi)
:=(g^2_k\nu_\kappa)_{\omega,t}
\left( \varepsilon^{-n}\phi_1\left(\frac{z-\cdot}{\varepsilon}\right) 
\varepsilon^{-1}\phi_2\left(\frac{\xi-\cdot}{\varepsilon}\right)\right). 
\]
The main properties are listed in the following lemma, whose 
proof is obvious in view of earlier (similar) lemmas.

\begin{lemma}\label{lemma: properties of g^2_k nu kappa epsilon}
Let $\kappa\in\mathcal{A}$. Then
\begin{enumerate}

\item $(g^2_k\nu_\kappa)_\varepsilon(\omega,t)\in C^\infty(\R^n_z\times\R_\xi)$ 
for all $(\omega,t)\in\Omega_T$;

\item for $\varepsilon < \varepsilon_\kappa$ and any $\omega\in\Omega, 
t\in[0,T], \xi\in\R$,
\[
\supp \left((g^2_k\nu_\kappa)_\varepsilon(\omega,t)(\cdot,\xi)\right)\subset 
\kappa(\supp\alpha_\kappa) 
+\overline{B_\varepsilon(0)}\subset\subset \tilde{X}_\kappa . 
\] 
Thus, for fixed $(\omega,t)\in\Omega_T$, the function 
$(g^2_k\nu_\kappa)_\varepsilon(\omega,t)$ can be seen as an 
element of $C^\infty(M\times\R)$, provided it is set zero outside of 
$X_\kappa\times\R$.

\item $\lim_{\varepsilon\to 0}(g^2_k\nu_\kappa)_\varepsilon(\omega,t)
= (g^2_k\nu_\kappa)_{\omega,t}$ in the sense of measures 
for any $(\omega,t)\in\Omega_T$. In particular, for 
$\psi\in C^0_c(X_\kappa\times\R)$,
\begin{align*}
\int_{M\times\R}\psi(x,\xi)& 
\frac{(g^2_k\nu_\kappa)_\varepsilon(\omega,t)(x,\xi)}
{\abs{h_\kappa(x)}^{1/2}}
\,d\xi\,dV_h(x)\\ 
&\overset{\varepsilon\downarrow 0}{\longrightarrow}
\int_M\alpha_\kappa(x)\left(\int_\R\psi(x,\xi)g^2_k(x,\xi)
\,\nu_{\omega,t,x}(d\xi)\right)\, dV_h(x). 
\end{align*}
This result holds for $\psi\in C^0_c(M\times\R)$ as well, since 
the function $\frac{(g^2_k\nu_\kappa)_\varepsilon}{\abs{h_\kappa}^{1/2}}$ 
and the measure $\alpha_\kappa g^2_k\,\nu \,dV_h$ 
are supported in $X_\kappa\times\R$.

\end{enumerate}
\end{lemma}

We need to define the following functions, for 
$(\omega,t,x,\xi)\in\Omega\times[0,T]\times M\times\R$ and 
$\varepsilon<\bar{\varepsilon}:=
\frac14 \min_{\kappa\in\mathcal{A}}\{\varepsilon_\kappa\}$:
\begin{align*}
& \Gamma_\varepsilon(\omega,t)(x,\xi)
:= \sum_{\kappa\in\mathcal{A}}
\frac{\Gamma_{\kappa,\varepsilon}(\omega,t)(x,\xi)}
{\abs{h_\kappa(x)}^{1/2}},
\\ &\Gamma_\varepsilon =
\rho^+_{\varepsilon}, 
\rho_{0,\varepsilon},
(g_k\nu)_\varepsilon,
(G^2\nu)_\varepsilon,
(g^2_k\nu)_\varepsilon,
m_\varepsilon,
A_{\varepsilon},
r_{\varepsilon},
\nu_\varepsilon,
\\ 
& \Gamma_{\kappa,\varepsilon} =
 (\rho^+_\kappa)_\varepsilon,
(\rho_{0,\kappa})_\varepsilon,
(g_k\nu_\kappa)_\varepsilon, 
(G^2\nu_\kappa)_\varepsilon,
(g^2_k\nu_\kappa)_\varepsilon, 
((\alpha_\kappa m)_{\sharp})_\varepsilon,
(A_{\kappa})_\varepsilon,
r_{\kappa,\varepsilon},
(\nu_\kappa)_\varepsilon.
\end{align*}

The properties of these functions are listed in
\begin{lemma}\label{properties of the regularized global objects}
For $\varepsilon<\bar{\varepsilon}$ and $(\omega,t)\in\Omega_T$, 
the functions defined above belong to $C^\infty(M\times\R)$. 
Furthermore, for $(\omega,t)\in\Omega_T$, the 
following convergence hold as $\varepsilon\to 0$:
\begin{enumerate}

\item $\rho^+_{\varepsilon}(\omega,t)\to \rho^+(\omega,t)$ in
$L^p_{\mathrm{loc}}(M\times\R)$, for all $1\leq p < \infty$. 

\item $\rho_{0,\varepsilon}(\omega)\to \rho_0(\omega)$ in 
$L^p_{\mathrm{loc}}(M\times\R)$, for all $1\leq p < \infty$.

\item $(g_k\nu)_\varepsilon(\omega,t)\to g_k\,\nu_{\omega,t} \,dV_h$ 
in the sense of measures.

\item $(G^2\nu)_\varepsilon(\omega,t)\to G^2\,\nu_{\omega,t} \,dV_h$ 
in the sense of measures.

\item $(g^2_k\nu)_\varepsilon(\omega,t)\to g^2_k\,\nu_{\omega,t} \,dV_h$ 
in the sense of measures.

\item $m_\varepsilon(\omega,t)\to m(\omega)\big|_{[0,t]\times M\times\R}$ 
in the sense of measures; the restriction of the 
measure $m(\omega)$ to $[0,t]\times M\times\R$ is denoted 
by $m(\omega)\big|_{[0,t]\times M\times\R}$.

\item $A_{\varepsilon}(\omega,t)\to 0$ in 
$L^p_{\mathrm{loc}}(M\times\R)$, for all $1\leq p < \infty$.

\item $r_{\varepsilon}\to 0$ in $L^1(\Omega_T\times M\times(-L,L))$, 
for all $L>0$.

\item $\nu_\varepsilon(\omega,t)\to \nu_{\omega,t} \,dV_h$ in the sense of 
measures; moreover, for a.e.~$(\omega,t)\in\Omega_T$ 
and all $(x,\xi)\in M\times\R$,  $\partial_\xi\rho^+_{\varepsilon}(\omega,t)(x,\xi)
=-\nu_\varepsilon(\omega,t)(x,\xi)$. 

\item for fixed $\varepsilon<\bar{\varepsilon}$ 
and any $(\omega,t)\in\Omega_T$, 
\[
\sum_{k=1}^p(g^2_k\nu)_\varepsilon(\omega,t)(x,\xi) 
\overset{p\to\infty}{\longrightarrow} 
(G^2\nu)_\varepsilon(\omega,t)(x,\xi),  
\qquad \forall (x,\xi)\in M \times \R.
\]
\end{enumerate}
\end{lemma}

\begin{proof}
The smoothness of the functions as well claims $(1)$ to $(6)$, 
and $(9)$ are direct consequences of 
Lemmas \ref{lemma: properties of rho kappa epsilon}, \ref{lemma: 
properties of alpha m  epsilon}, 
\ref{lemma: properties of G^2 nu kappa epsilon}, 
\ref{lemma: properties of A kappa epsilon}, 
\ref{lemma: properties of gk nu kappa 
epsilon}, and \ref{lemma: commutator}, and the fact 
that the partition of unity is finite.

Furthermore, from Lemma \ref{lemma: properties of A kappa epsilon}, 
the limit for $A_\varepsilon(\omega,t)$ is 
\[
\sum_{\kappa\in\mathcal{A}}\rho^+(\omega,t,\cdot,\cdot)
(f'_\cdot(\cdot),\nabla\alpha_\kappa)_h=\rho^+(\omega,t,\cdot,\cdot)
(f'_\cdot(\cdot),\nabla 1)_h=0. 
\]

To verify claim (8), we notice from 
Lemma \ref{lemma: commutator} that
\begin{align*}
\int_{\Omega_T\times M\times(-L,L)}&
\frac{\abs{r_{\kappa,\varepsilon}(\omega,t)(x,\xi)}}{\abs{h_\kappa(x)}^{1/2}}
\, d\xi\,dV_h(x)\,dt\,\P(d\omega) 
\\ & =\int_{\Omega_T\times(-L,L)}
\int_{\tilde{X}_\kappa}
\abs{r_{\kappa,\varepsilon}(\omega,t)(z,\xi)}
\,dz \,d\xi\,dt\,\P(d\omega)
\overset{\varepsilon\downarrow 0}{\longrightarrow} 0. 
\end{align*}
Summing over $\kappa$, we obtain the desired conclusion.

To prove claim $(10)$, we begin by observing that 
Lemma \ref{lemma: G^2 is locally Lipschitz } 
holds for $G^2_p(x,\xi):=\sum_{k=1}^p\abs{g_k(x,\xi)}^2$, 
namely, $G^2_p$ is locally Lipschitz, uniformly in $p\in\N$. 
Hence, by means of the Ascoli-Arzel\'a theorem, $G^2_p\to G^2$ 
uniformly on compact sets. It is therefore clear 
that for any $\psi\in C^0_c(\R^n\times\R)$, $(\omega,t)\in\Omega_T$, 
and $\kappa\in\mathcal{A}$, we have
\[
\sum_{k=1}^p(g^2_k\nu_\kappa)_{\omega,t}(\psi) 
\overset{p\to\infty}{\longrightarrow} (G^2\nu_\kappa)_{\omega,t}(\psi). 
\]
We conclude hence that, for fixed $\varepsilon<\bar{\varepsilon}$ 
and $(\omega,t)\in\Omega_T$,
\[
\sum_{k=1}^p(g^2_k\nu)_\varepsilon(\omega,t)(\cdot,\cdot) 
\overset{p\to\infty}{\longrightarrow}
(G^2\nu)_\varepsilon(\omega,t)(\cdot,\cdot),  
\;\;\; \mbox{pointwise on } M\times\R. 
\]
\end{proof}

\subsection{Global equation and renormalization}

To obtain a single (global) equation, we 
sum over $\kappa$ in \eqref{eq: regolarizzata 3}, arriving at   
\begin{equation}\label{eq: regolarizzata 4}
\begin{split}
-\rho^+_{\varepsilon}&(\omega,t)(x,\xi)+\rho_{0,\varepsilon}(\omega)(x,\xi)
\\ & =\int_0^t\Div(\rho^+_{\varepsilon}(\omega,s)(x,\xi)f'_{x}(\xi))\,ds
\\ & \qquad -\sum_{k\geq 1}\int_0^t (g_k\nu)_\varepsilon(\omega,s)(x,\xi)
\,d\beta_k(s) \\ & \qquad\quad
+\frac12\int_0^t\partial_\xi
\left[(G^2\nu)_\varepsilon(\omega,s)(x,\xi)\right]
\,ds-\partial_\xi\left[m_\varepsilon(\omega,t)(x,\xi)\right] 
\\ & \qquad\quad\quad 
-\int_0^t  A_{\varepsilon} (\omega,s)(x,\xi)\, ds
+\int_0^tr_{\varepsilon}(\omega,s)(x,\xi)\,ds
\\ & =: -H_\varepsilon(t)(x,\xi)-I_\varepsilon(t)(x,\xi), 
\quad \text{$\P$-almost surely},
\end{split}
\end{equation}
valid for $\varepsilon<\bar{\varepsilon}$, 
$(x,\xi)\in M\times\R$, and $t\in[0,T]$, where 
\[
H_\varepsilon(t)(x,\xi):=\sum_{k\geq 1}
\int_0^t(g_k\nu)_\varepsilon(\omega,s)(x,\xi)\,d\beta_k(s)
\]
and
\begin{align*}
I_\varepsilon(t)(x,\xi)&:=-\int_0^t
\Div(\rho^+_{\varepsilon}(\omega,s)(x,\xi)f'_{x}(\xi))\,ds
- \frac12\int_0^t\partial_\xi 
\left[(G^2\nu)_\varepsilon(\omega,s)(x,\xi)\right]
\, ds \\ &\hspace{0.5cm}
+\partial_\xi\left[m_\varepsilon(\omega,t)(x,\xi)\right] 
+\int_0^t A_{\varepsilon} (\omega,s)(x,\xi)\,ds
-\int_0^tr_\varepsilon(\omega,s)(x,\xi)\,ds. 
\end{align*}
Observe that $I_\varepsilon(\cdot)(x,\xi)$ is in $BV[0,T]$, 
c{\`a}dl{\`a}g, and adapted. In fact, we have that 
$I_\varepsilon(\cdot)(x,\xi)-\partial_\xi
\left[m_\varepsilon(\cdot,\cdot)(x,\xi)\right]$ 
is absolutely continuous. 

Applying It\^{o}'s formula for semimartingales \cite{Protter}, 
we obtain the following equation for 
$(\rho^+_{\varepsilon}(\cdot,\cdot)(x,\xi))^2$: $\P$-almost surely,
\begin{align*}
(\rho^+_{\varepsilon}&(\omega,t)(x,\xi))^2
=(\rho_{0,\varepsilon}(\omega)(x,\xi))^2+2\int_0^t
\rho^+_{\varepsilon}(\omega,s_-)(x,\xi)\, dH_\varepsilon(s)(x,\xi)
\\ &+2\int_0^t\rho^+_{\varepsilon}(\omega,s_-)(x,\xi)\, 
I_\varepsilon(ds)(x,\xi)+\int_0^td\left[H_\varepsilon(s)(x,\xi)\right]
\\ &+\int_0^td\left[I_\varepsilon(s)(x,\xi)\right]+2\int_0^t\, 
d\left[H_\varepsilon(s)(x,\xi),I_\varepsilon(s)(x,\xi)\right]
\\ & +\sum_{0<s\leq t}\left\{(\rho^+_{\varepsilon}(\omega,s)(x,\xi))^2
-(\rho^+_{\varepsilon}(\omega,s_-)(x,\xi))^2\right.
\\ &\hspace{1.6cm}\left. 
-2\rho^+_{\varepsilon}(\omega,s_-)(x,\xi)
\Delta \rho^+_{\varepsilon}(\omega,s)(x,\xi) 
- (\Delta\rho^+_{\varepsilon}(\omega,s)(x,\xi))^2\right\}, 
\end{align*}
valid for $\varepsilon<\bar{\varepsilon}$, $(x,\xi)\in M\times \R$, 
and $t\in[0,T]$. Here $d[\cdots]$ denotes 
quadratic covariation, and $\Delta f(s)$ 
indicates the jump of a function $f$ at $s$, that is, 
$\Delta f(s)=f(s)-f(s_-)$. We observe that 
$\sum_{0<s\leq t}\{\cdots\}$ is identically zero, and so is 
$\int_0^td\left[H_\varepsilon(s),I_\varepsilon(s)\right]$, 
because $H_\varepsilon$ is a continuous martingale 
and $I_\varepsilon$ is of bounded variation. 
Note that $\int_0^td [I_\varepsilon(s)(x,\xi)]\geq 0$.
Hence, the equation for $(\rho^+_{\varepsilon})^2$ simplifies to
\begin{equation}\label{eq: regolarizzata al quadrato}
\begin{split}
(\rho^+_{\varepsilon}&(\omega,t)(x,\xi))^2
=(\rho_{0,\varepsilon}(\omega)(x,\xi))^2
+2\int_0^t\rho^+_{\varepsilon}(\omega,s_-)(x,\xi)
\, dH_\varepsilon(s)(x,\xi)
\\ &
+2\int_0^t\rho^+_{\varepsilon}(\omega,s_-)(x,\xi)
\, I_\varepsilon(ds)(x,\xi)+\int_0^t
\, d\left[H_\varepsilon(s)(x,\xi)\right]
\\ & +\int_0^t\, d\left[I_\varepsilon(s)(x,\xi)\right], 
\quad \text{$\P$-almost surely}.
\end{split}
\end{equation}

\subsection{Proof of Proposition \ref{prop: Reduction and uniqueness}}
We add \eqref{eq: regolarizzata al quadrato} 
and \eqref{eq: regolarizzata 4}, and multiply the 
result by $0\leq\theta\in D(M\times\R)$; afterwards, we integrate with 
respect to $d\xi\,dV_h(x)\,d\P$. 
In view of the (stochastic) Fubini theorem, we can 
interchange integrals as we see fit. Consequently,
\begin{align*}
\EE&\int_{M\times\R}H_\varepsilon(t)(x,\xi)\theta 
\,d\xi\,dV_h(x)=0, \\
\EE&\int_{M\times\R}\theta\int_0^t\rho^+_{\varepsilon}(\omega,s)(x,\xi)
\,dH_\varepsilon(s)(x,\xi)\,d\xi\,dV_h(x)=0, 
\end{align*}
and as a result we obtain
\begin{equation}\label{eq: difference between rho_epsilon^2 and rho_epsilon}
\begin{split}
J_0& :=\EE\int_{M\times\R}\left[(\rho^+_\varepsilon(\omega,t)(x,\xi))^2
-\rho^+_\varepsilon(\omega,t)(x,\xi)\right]\theta\,d\xi\,dV_h(x)\\
& =\EE\int_{M\times\R}\left[(\rho_{0,\varepsilon}(\omega)(x,\xi))^2
-\rho_{0,\varepsilon}(\omega)(x,\xi)\right]\theta\,d\xi\,dV_h(x)
\\ & \quad
-2\EE\int_0^t\int_{M\times\R}\theta\,\rho^+_\varepsilon(\omega,s)(x,\xi)\,
\Div(\rho^+_{\varepsilon}(\omega,s)(x,\xi)f'_{x}(\xi))
\,d\xi\,dV_h(x)\,ds 
\\ & \quad 
-\EE\int_0^t\int_{M\times\R}\theta\, \rho^+_{\varepsilon}(\omega,s)(x,\xi)
\partial_\xi\left[(G^2\nu)_\varepsilon(\omega,s)(x,\xi)\right] 
\,d\xi\,dV_h(x)\,ds
\\ & \quad
+2\EE\int_{M\times\R}\int_0^t \rho^+_{\varepsilon}(\omega,s_-)(x,\xi) 
\,\theta\,\partial_\xi\left[m_\varepsilon(\omega,ds)(x,\xi)\right] 
\,d\xi\,dV_h(x)
\\ & \quad 
+2\EE\int_0^t\int_{M\times\R}\theta\, \rho^+_{\varepsilon}(\omega,s)(x,\xi)
A_\varepsilon(\omega,s)(x,\xi) \,d\xi\,dV_h(x)\,ds
\\ & \quad 
-2\EE\int_0^t\int_{M\times\R}\theta\, \rho^+_{\varepsilon}(\omega,s)(x,\xi)
r_\varepsilon(\omega,s)(x,\xi) \,d\xi\,dV_h(x)\,ds
\\ & \quad 
+\EE\int_{M\times\R}\int_0^t
\sum_{k\geq 1}\left((g_k\nu)_\varepsilon(\omega,s)(x,\xi)\right)^2
\, ds\,\theta\,d\xi\,dV_h(x)
\\ & \quad 
+\EE\int_{M\times\R}\theta\int_0^td\left[I_\varepsilon(s)(x,\xi)\right]
\,d\xi\,dV_h(x)
\\ & \quad +\EE\int_0^t\int_{M\times\R}\theta\,
\Div(\rho^+_{\varepsilon}(\omega,s)(x,\xi)f'_{x}(\xi))\,d\xi\,dV_h(x)\,ds
\\ & \quad 
+\frac12\,\EE\int_0^t\int_{M\times\R}\theta\, 
\partial_\xi\left[(G^2\nu)_\varepsilon(\omega,s)(x,\xi)\right] 
\,d\xi\,dV_h(x)\,ds
\\ & \quad 
-\EE\int_{M\times\R} \theta\,\partial_\xi
\left[m_\varepsilon(\omega,t)(x,\xi)\right] \,d\xi\,dV_h(x)\\
&\quad -\EE\int_0^t\int_{M\times\R}\theta\, 
A_\varepsilon(\omega,s)(x,\xi) \,d\xi\,dV_h(x)\,ds
\\ & \quad +\EE\int_0^t\int_{M\times\R}\theta\, 
r_\varepsilon(\omega,s)(x,\xi) \,d\xi\,dV_h(x)\,ds
=:J_1+\cdots +J_{13}, 
\end{split}
\end{equation}
valid for $\varepsilon<\bar{\varepsilon}$ 
and $0\leq\theta\in D(M\times\R)$.

We are going to send the parameter $\varepsilon$ to $0$. 
Let us analyze the behavior of the terms 
in \eqref{eq: difference between rho_epsilon^2 and rho_epsilon} 
separately, starting with the left hand side. 
By Lemma \ref{properties of the regularized global objects}, 
$\rho^+_{\varepsilon}(\omega,t)$ converges 
to $\rho^+(\omega,t)$ in $L^p_{\mathrm{loc}}(M\times\R)$ 
as $\varepsilon\to 0$, for all $(\omega,t)\in\Omega_T$ 
and $p\in [1,\infty)$. Moreover, as we have seen in 
the proof of Lemma \ref{lemma: commutator}, we can bound 
$\abs{(\rho^+_\kappa)_\varepsilon(\omega,t)(z,\xi)}\leq C(M,h,\mathcal{A})$, 
$\forall \kappa\in\mathcal{A}$, uniformly in $\varepsilon$, 
$(\omega,t,z,\xi)\in\Omega_T\times\R^n\times\R$. 
Therefore, for a different constant, 
\begin{equation}\label{eq: universal bound for rho epsilon}
	\abs{\rho^+_\varepsilon(\omega,t)(x,\xi)}
	\leq C(M,h,\mathcal{A}), 
\end{equation}
for all $(\omega,t,x,\xi)\in\Omega_T\times M\times\R$. 
By the dominated convergence theorem,
\[
\lim_{\varepsilon\to 0}J_0=\EE\int_{M\times\R}
\left[\rho^+(\omega,t,x,\xi)^2-\rho^+(\omega,t,x,\xi)\right]
\theta\,d\xi\,dV_h(x). 
\]
An analogous result holds for $J_1$.

From Lemma \ref{properties of the regularized global objects}, 
we know that $\lim\limits_{\varepsilon\to 0} r_{\varepsilon}= 0$ in 
$L^1(\Omega_T\times M\times(-L,L))$, for all $L>0$. 
This property, coupled with \eqref{eq: universal bound for rho epsilon}, 
enables us to conclude that 
$$
J_{13} \overset{\varepsilon\downarrow 0}{\longrightarrow} 0,
\qquad J_6 \overset{\varepsilon\downarrow 0}{\longrightarrow} 0. 
$$

We now deal with the terms $J_5$ and $J_{12}$. 
By Lemma \ref{properties of the regularized global objects}, 
$\lim_{\varepsilon\to 0} A_{\varepsilon}(\omega,s)= 0$ 
in $L^p_{\mathrm{loc}}(M\times\R)$, for all $(\omega,s)\in\Omega_T$ 
and $p\in [1, \infty)$. Furthermore, as in the proof of 
Lemma \ref{lemma: properties of A kappa epsilon}, we 
see that $\norm{A_\varepsilon(\omega,s)}_{L^1(\supp\theta)}$ is 
bounded by an $\varepsilon$-independent 
constant $C(M,h,\mathcal{A},f',\theta)$.  
The dominated convergence theorem thus implies  
$$
J_{12} 
\overset{\varepsilon\downarrow 0}{\longrightarrow} 0
\qquad
J_5\overset{\varepsilon\downarrow 0}{\longrightarrow} 0. 
$$

Let us consider the terms involving the $\Div$ operator. 
Since $\rho^+_{\varepsilon}(\omega,s)$ is smooth (in $x$) 
and the $\xi$-derivative of the flux is geometry-compatible 
\eqref{eq: gc}, it follows (after an integration by parts) that 
for any $(\omega,s,\xi)\in\Omega_T\times\R$,
\begin{align*}
& -2\int_M\theta\,\rho^+_\varepsilon(\omega,s)(x,\xi)
\,\Div(\rho^+_{\varepsilon}(\omega,s)(x,\xi)f'_{x}(\xi))\,dV_h(x)
\\ & \qquad  = \int_M\rho^+_\varepsilon(\omega,s)(x,\xi)^2
\left(f'_x(\xi),\nabla \theta\right)_h\,dV_h(x). 
\end{align*}
Moreover, coupling Lemma \ref{properties of the regularized global objects} 
with $\abs{\left(f'_x(\xi),\nabla \theta\right)}_h 
\leq C(f',\theta)$ on $\supp\theta'$, we obtain
\begin{align*}
&\int_{M\times\R}\rho^+_\varepsilon(\omega,s)(x,\xi)^2
\left(f'_x(\xi),\nabla \theta\right)_h\,d\xi\,dV_h(x)
\\ & \qquad 
\overset{\varepsilon \downarrow 0}{\longrightarrow}
\int_{M\times\R}\rho^+(\omega,s,x,\xi)^2
\left(f'_x(\xi),\nabla \theta\right)_h\,d\xi\,dV_h(x),  
\end{align*}
for all $(\omega,s)\in\Omega_T$. Again we can find a bound that is uniform 
in $\varepsilon$ and $(\omega,s)\in\Omega_T$, which permits us to apply the 
dominated convergence theorem to conclude that 
\[
\lim_{\varepsilon}J_2=\EE\int_0^t\int_{M\times\R}
\rho^+(\omega,s,x,\xi)^2\left(f'_x(\xi),\nabla \theta\right)_h
\,d\xi\,dV_h(x)\,ds.  
\]
A similar result holds for $J_9$.

We now treat the term $J_{11}$. 
By Lemma \ref{properties of the regularized global objects}, 
\[
-\int_{M\times\R} \theta\,\partial_\xi
\left[m_\varepsilon(\omega,t)(x,\xi)\right] 
\,d\xi\,dV_h(x)
\overset{\varepsilon\downarrow 0}{\longrightarrow}
\int_{[0,t]\times M\times\R}\partial_\xi\theta 
\, m(ds,dx,d\xi), 
\]
for all $(\omega,t)\in\Omega_T$. Furthermore, by 
Lemma \ref{lemma: properties of alpha m  epsilon}, 
summing over $\kappa\in\mathcal{A}$, we obtain 
\begin{align*}
\bigg|\int_{M\times\R}\theta\,\partial_\xi
\left[m_\varepsilon(\omega,t)(x,\xi)\right]
\,d\xi\,dV_h(x) \bigg| & 
\leq \norm{\partial_\xi\theta}_{L^\infty(M\times\R)}
\int_{[0,t]\times M\times\R} \,m(ds,dx,d\xi)
\\ &  
\leq\norm{\partial_\xi\theta}_{L^\infty(M\times\R)}\,
m(\omega)([0,T]\times M\times\R). 
\end{align*}
In view of the integrability property 
$\EE \, m([0,T]\times M\times\R)<\infty$, we 
are allowed to apply the dominated convergence theorem 
to arrive at
\[
\lim_{\varepsilon\to 0} J_{11}
=\EE\int_{[0,t]\times M\times\R}\partial_\xi\theta
\, m(ds,dx,d\xi). 
\]

We now deal with the term $J_{10}$. Arguing as 
above, Lemma \ref{properties of the regularized global objects} 
implies
\begin{align*}
&\int_{M\times\R}\theta\, 
\partial_\xi\left[(G^2\nu)_\varepsilon(\omega,s)(x,\xi)\right] 
\,d\xi\,dV_h(x)
\\ & \qquad\quad 
\overset{\varepsilon\downarrow 0}{\longrightarrow}
-\int_M\left(\int_\R\partial_\xi\theta \,G^2(x,\xi)
\nu_{\omega,s,x}(d\xi)\right)dV_h(x), 
\end{align*}
for all $(\omega,s)\in\Omega_T$. Furthermore, by 
Lemma \ref{lemma: properties of G^2 nu kappa epsilon}, summing 
over $\kappa\in\mathcal{A}$, we obtain 
\begin{align*}
\bigg|\int_{M\times\R}\theta\, \partial_\xi & 
\left[(G^2\nu)_\varepsilon(\omega,s)(x,\xi)\right] 
\,d\xi\,dV_h(x)\bigg|
\\ & \leq \norm{\partial_\xi\theta}_{L^\infty(M\times\R)}
\int_M\left(\int_\R \,G^2(x,\xi)\,\nu_{\omega,s,x}(d\xi)\right)\, dV_h(x), 
\end{align*}
which is integrable over $\Omega\times[0,t]$ by the ``vanishing 
at infinity" assumption, cf.~Definition \ref{def: generalized solution}. 
Again by dominated convergence, we conclude
\[
\lim_{\varepsilon\to 0}J_{10}=-\frac12\,
\EE\int_0^t\int_M\left(\int_\R\partial_\xi\theta 
\,G^2(x,\xi)\, \nu_{\omega,s,x}(d\xi)\right)\, dV_h(x)\,ds. 
\]

Summarizing, sending $\varepsilon\to 0$ in 
\eqref{eq: difference between rho_epsilon^2 and rho_epsilon} yields
\begin{equation}\label{eq: difference between rho_epsilon^2 and rho_epsilon (2)}
\begin{split}
\EE & \int_{M\times\R}\left[(\rho^+(\omega,t,x,\xi))^2
-\rho^+(\omega,t,x,\xi)\right]\theta\,d\xi\,dV_h(x)
\\ & =\EE\int_{M\times\R}\left[(\rho_{0}(\omega,x,\xi))^2
-\rho_{0}(\omega,x,\xi)\right]\theta\,d\xi\,dV_h(x)
\\ & \quad 
+\EE\int_0^t\int_{M\times\R}\left[(\rho^+(\omega,s,x,\xi))^2
-\rho^+(\omega,s,x,\xi)\right]\left(f'_x(\xi),\nabla \theta\right)_h
\,d\xi\,dV_h(x)\,ds
\\ & \quad 
+\limsup_{\varepsilon\to 0}\bigg\{\EE\int_0^t\int_{M\times\R}
\partial_\xi\theta\, \rho^+_{\varepsilon}(\omega,s)(x,\xi)
(G^2\nu)_\varepsilon(\omega,s)(x,\xi)\,d\xi\,dV_h(x)\,ds
\\ & \qquad \qquad\quad\quad
-\EE\int_0^t\int_{M\times\R}\theta\, \nu_\varepsilon(\omega,s)(x,\xi)
(G^2\nu)_\varepsilon(\omega,s)(x,\xi)\,d\xi\,dV_h(x)\,ds
\\ & \qquad \qquad\quad\quad
+2\EE\int_{M\times\R}\int_0^t \rho^+_{\varepsilon}(\omega,s_-)(x,\xi) 
\,\theta\,\partial_\xi\left[m_\varepsilon(\omega,ds)(x,\xi)\right] 
\,d\xi\,dV_h(x)
\\ & \qquad \qquad\quad\quad
+\EE\int_{M\times\R}\int_0^t\sum_{k\geq 1}
\left((g_k\nu)_\varepsilon(\omega,s)(x,\xi)\right)^2\, ds\,\theta\,d\xi\,dV_h(x)
\\ & \qquad \qquad\quad\quad
+\EE\int_{M\times\R}\theta\int_0^t
\, d\left[I_\varepsilon(s)(x,\xi)\right]\,d\xi\,dV_h(x)\bigg\}
\\ &\quad
-\frac12\,\EE\int_0^t\int_M
\left(\int_\R\partial_\xi\theta \,G^2(x,\xi)
\nu_{\omega,x,s}(d\xi)\right)\, dV_h(x)\,ds
\\ &\quad +\EE\int_{[0,t]\times M\times\R}\partial_\xi\theta
\,m(ds,dx,d\xi), 
\end{split}
\end{equation}
where we have integrated by parts in the $J_3$ term and made 
use of Lemma \ref{properties of the regularized global objects} 
to express the $\xi$-derivative of 
$\rho^+_{\varepsilon}(\omega,s)$ as $-\nu_\varepsilon(\omega,s)$. 

Before continuing we state three technical lemmas, whose proofs
are postponed until the end of the section.

\begin{lemma}\label{lemma: stima per G2 nu epsilon contro rho epsilon e test}
For $\gamma\in C^0(M)$ and $\psi\in C^1_c(\R)$,
\begin{align*}
&\bigg|\int_{M\times\R}\gamma(x)\,
\partial_\xi\psi(\xi) \, \rho^+_{\varepsilon}(\omega,s)
(x,\xi)(G^2\nu)_\varepsilon(\omega,s)(x,\xi)\,d\xi\,dV_h(x)\bigg|
\\ & \; \leq 
C(M,h,\mathcal{A})\norm{\gamma}_{L^\infty(M)}
\norm{\partial_\xi\psi}_{L^\infty(\R)}
\int_M\int_{\supp\partial_\xi \psi
+[-\varepsilon,\varepsilon]} \!\!\! \!\!\! G^2(x,\xi)
\,\nu_{\omega,s,x}(d\xi)\,dV_h(x), 
\end{align*}
for all $(\omega,s)\in\Omega_T$ and $\varepsilon<\bar{\varepsilon}/2$.
\end{lemma}

The next lemma depends crucially on the regularity of the 
noise coefficients $(g_k)_k$, cf.~\eqref{eq: gk are lipschitz} 
and the motivational discussion around \eqref{intro:reg-error}. 
 
\begin{lemma}\label{lemma: simmetry argument}
Let $0\leq\psi\in\mathcal{D}(\R)$. Then 
\begin{align*}
-\EE\int_0^t&\int_{M\times\R}\psi(\xi)\, 
\nu_\varepsilon(\omega,s)(x,\xi)
(G^2\nu)_\varepsilon(\omega,s)(x,\xi)\,d\xi\,dV_h(x)\,ds
\\ & +\EE\int_{M\times\R}\int_0^t\sum_{k\geq 1}
\left((g_k\nu)_\varepsilon(\omega,s)(x,\xi)\right)^2 \,ds
\, \psi(\xi)\, \,d\xi\,dV_h(x)
\overset{\varepsilon\downarrow 0}{\longrightarrow} 0. 
\end{align*}
\end{lemma}

The last lemma controls a term that 
involves the measure $m_\varepsilon(\omega,ds)(x,\xi)$ 
for large values of $\xi$.

\begin{lemma}\label{lemma: un ennesimo lemma strumentale}
Let $\psi_N\in \mathcal{D}(\R)$, $N\in\N$ with 
$0\leq\psi_N\leq 1$, $\psi_N(\xi)=1$, if $\abs{\xi}\leq N$, $\psi_N(\xi)=0$, 
if $\abs{\xi}\geq N+1$, and $|\partial_\xi\psi_N(\xi)|\leq 2$. Then
\[
\lim_{N\to\infty}\limsup_{\varepsilon\to 0}
\abs{-2\EE\int_{M\times\R}\int_0^t 
\partial_\xi\psi_N\,\rho^+_{\varepsilon}(\omega,s_-)(x,\xi) 
\,m_\varepsilon(\omega,ds)(x,\xi) \,d\xi\,dV_h(x)}=0. 
\]
\end{lemma}

Let us continue with the proof 
of Proposition \ref{prop: Reduction and uniqueness}. 
Choose $\theta(x,\xi)=\psi_N(\xi)$, where 
$\psi_N\in \mathcal{D}(\R)$, $N\in\N$, and $0\leq\psi_N\leq 1$, 
$\psi_N(\xi)=1$, if $\abs{\xi}\leq N$, $\psi_N(\xi)=0$, 
if $\abs{\xi}\geq N+1$, and $|\partial_\xi\psi_N(\xi)|\leq 2$. 

With this choice of $\theta$, the second term on the right-hand side 
of \eqref{eq: difference between rho_epsilon^2 and rho_epsilon (2)} 
is zero. Since the kinetic measure $m$ vanishes 
for large $\xi$ (cf.~Definition \ref{def: kinetik measure}),
\begin{align*}
\bigg|\EE\int_{[0,t]\times M\times\R}\partial_\xi\psi_N\,
\, m(ds,dx,d\xi)\bigg|
&\leq 2 \EE\int_{[0,t]\times M\times
\left\{\xi: N\leq \abs{\xi}\leq N+1\right\}} 
\, m(ds,dx,d\xi)\\
&\leq 2\EE\int_{[0,T]\times M\times B _N^c} \, m(ds,dx,d\xi)
\overset{N\uparrow\infty}{\longrightarrow} 0. 
\end{align*}
Moreover, by \eqref{eq: definition of G},
\begin{align*}
\bigg|-\frac12\,\EE\int_0^t\int_M & 
\left(\int_\R\partial_\xi\psi_N \,G^2(x,\xi)
\nu_{\omega,s,x}(d\xi)\right)dV_h(x)\,ds\bigg|
\\ &\leq 
\EE\int_0^t\int_M\left(\int_{N\leq \abs{\xi}\leq N+1}D_1
\left(1+\abs{\xi}^2\right)
\nu_{\omega,s,x}(d\xi)\right)dV_h(x)\,ds. 
\end{align*}
By the ``vanishing at infinity" assumption \eqref{def: Young measure}, 
applying twice the dominated convergence theorem, we 
conclude that the above term vanishes as $N\to\infty$. 

Therefore, taking into account Lemma \ref{lemma: simmetry argument} 
as well, the equation 
\eqref{eq: difference between rho_epsilon^2 and rho_epsilon (2)}
becomes
\begin{equation}\label{eq: difference between rho_epsilon^2 and rho_epsilon (3)}
\begin{split}
\EE& \int_{M\times\R}\left[(\rho^+(\omega,t,x,\xi))^2
-\rho^+(\omega,t,x,\xi)\right]\psi_N\,d\xi\,dV_h(x)
\\ &=\EE\int_{M\times\R}\left[(\rho_{0}(\omega,x,\xi))^2
-\rho_{0}(\omega,x,\xi)\right]\psi_N\,d\xi\,dV_h(x)
\\ &\quad 
+\limsup_{\varepsilon\to 0}\bigg\{\EE\int_0^t
\int_{M\times\R}\partial_\xi\psi_N\, 
\rho^+_{\varepsilon}(\omega,s)(x,\xi)
(G^2\nu)_\varepsilon(\omega,s)(x,\xi)\,d\xi\,dV_h(x)\,ds
\\ & \qquad\qquad\qquad\quad
+2\EE\int_{M\times\R}\int_0^t \rho^+_{\varepsilon}(\omega,s_-)(x,\xi) 
\,\psi_N\,\partial_\xi\left[m_\varepsilon(\omega,ds)(x,\xi)\right] 
\,d\xi\,dV_h(x)
\\ & \qquad\qquad\qquad\quad
+\EE\int_{M\times\R}\psi_N\int_0^td
\left[I_\varepsilon(s)(x,\xi)
\right]\,d\xi\,dV_h(x)\bigg\}+o(1),
\end{split}
\end{equation}
where the error $o(1)$ tends to zero as $N\to\infty$.

The following integration by parts formula can be verified:
\begin{equation}\label{eq: nasty IP}
\begin{split}
J_4 & =2\EE\int_{M\times\R}\int_0^t \rho^+_{\varepsilon}(\omega,s_-)(x,\xi) 
\,\psi_N\,\partial_\xi\left[m_\varepsilon(\omega,ds)(x,\xi)\right] 
\,d\xi\,dV_h(x)
\\ & =-2\EE\int_{M\times\R}\int_0^t 
\partial_\xi\left(\rho^+_{\varepsilon}(\omega,s_-)(x,\xi) 
\,\psi_N\right) \, m_\varepsilon(\omega,ds)(x,\xi) \,d\xi\,dV_h(x).
\end{split}
\end{equation}
Note that in the deterministic setting one also mollifies 
in time, in which case \eqref{eq: nasty IP} holds trivially. 
In the present context the relevant functions 
are not smooth in the time variable and, as a 
result, \eqref{eq: nasty IP} is no longer immediate. 
However, to avoid making this paper even longer, we do not 
give a detailed proof of \eqref{eq: nasty IP}.

Therefore,
\begin{align*}
J_4 &=-2\EE\int_{M\times\R}\int_0^t \partial_\xi\psi_N\,
\rho^+_{\varepsilon}(\omega,s_-)(x,\xi) 
\,m_\varepsilon(\omega,ds)(x,\xi) \,d\xi\,dV_h(x)
\\ &\hspace{0.5cm}+2\EE\int_{M\times\R}\int_0^t 
\psi_N\,\nu_\varepsilon(\omega,s_-)(x,\xi) 
\,m_\varepsilon(\omega,ds)(x,\xi) \,d\xi\,dV_h(x)
\\ & \geq -2\EE\int_{M\times\R}\int_0^t 
\partial_\xi\psi_N\,\rho^+_{\varepsilon}(\omega,s_-)(x,\xi) 
\,m_\varepsilon(\omega,ds)(x,\xi) \,d\xi\,dV_h(x), 
\end{align*}
because $m_\varepsilon(\omega,ds)(x,\xi)$ is a positive $s$-measure 
(indeed, the function $s\mapsto m_\varepsilon(\omega,s)(x,\xi)$ is 
non-decreasing), and $\psi_N$, $\nu_\varepsilon(\omega,s_-)(x,\xi)$
are non-negative functions.

By using Lemma \ref{lemma: stima per G2 nu epsilon contro rho epsilon e test} 
with $\gamma\equiv 1$ and $\psi=\psi_N$, for  
$\varepsilon<\bar{\varepsilon}/4,\varepsilon<1$ and 
$(\omega,s)\in\Omega_T$, we obtain
\begin{align*}
\bigg| &\int_{M\times\R} \partial_\xi\psi_N\, 
\rho^+_{\varepsilon}(\omega,s)(x,\xi)
(G^2\nu)_\varepsilon(\omega,s)(x,\xi)\,d\xi\,dV_h(x) \bigg|
\\ & \leq 2\,D_1 C(M,h,\mathcal{A}) 
\int_M \int_{N-1\leq \abs{\xi}\leq N+2}\left(1+\abs{\xi}^2\right)
\,\nu_{\omega,s,x}(d\xi)\,dV_h(x). 
\end{align*}
Hence,
\begin{align*}
\bigg|\limsup_{\varepsilon\to 0}
& \,\EE\int_0^t\int_{M\times\R}\partial_\xi\psi_N\, 
\rho^+_{\varepsilon}(\omega,s)(x,\xi)(G^2\nu)_\varepsilon(\omega,s)(x,\xi)
\,d\xi\,dV_h(x)\,ds\bigg|
\\ & \leq \limsup_{\varepsilon}
\bigg| \, \EE\int_0^t\int_{M\times\R}\partial_\xi\psi_N\, 
\rho^+_{\varepsilon}(\omega,s)(x,\xi)
(G^2\nu)_\varepsilon(\omega,s)(x,\xi)\,d\xi\,dV_h(x)\,ds\bigg|
\\ & \leq  2\,D_1C(M,h,\mathcal{A})\,
\EE \int_0^t\int_M\int_{N-1\leq \abs{\xi}\leq N+2}\left(1+\abs{\xi}^2\right)
\,\nu_{\omega,s,x}(d\xi)\,dV_h(x)\,ds, 
\end{align*}
and once again by the ``vanishing at infinity" assumption on $\nu$, 
cf.~\eqref{def: Young measure}, the last 
quantity goes to zero as $N\to\infty$.

Using Lemma \ref{lemma: un ennesimo lemma strumentale} 
and the non-negativity of 
$\int_0^td\left[I_\varepsilon(s)(x,\xi)\right]$, we arrive at 
\begin{align*}
& \EE \int_{M\times\R} \left[(\rho^+(\omega,t,x,\xi))^2
-\rho^+(\omega,t,x,\xi)\right]\psi_N\,d\xi\,dV_h(x)
\\ & \qquad \geq 
\EE\int_{M\times\R}\left[(\rho_{0}(\omega,x,\xi))^2
-\rho_{0}(\omega,x,\xi)\right]\psi_N\,d\xi\,dV_h(x) +o(1), 
\end{align*}
where $o(1)$ tends to zero as $N\to\infty$.

We use the monotone convergence theorem to send $N\to\infty$, obtaining 
\begin{equation}\label{eq: pre-contraction inequality}
\begin{split}
& 0\leq\EE\int_{M\times\R}\left[\rho^+(\omega,t,x,\xi)
-(\rho^+(\omega,t,x,\xi))^2\right]\,d\xi\,dV_h(x)
\\ & \qquad \leq 
\EE\int_{M\times\R}\left[\rho_{0}(\omega,x,\xi)
-(\rho_{0}(\omega,x,\xi))^2\right]\,d\xi\,dV_h(x), 
\end{split}
\end{equation}
for all $t\in[0,T]$. This concludes the proof of 
Proposition \ref{prop: Reduction and uniqueness}.

\subsection{Uniqueness part of Theorem \ref{thm:well-posed}}\label{subsec:wellposed}

Let us consider a kinetic initial function $\rho_0=\En_{u_0>\xi}$, 
with $u_0\in L^\infty(\Omega,\mathcal{F}_0;L^\infty(M,h))$. 
Then the right-hand side of \eqref{eq: pre-contraction inequality} is zero. 
Repeating an argument from \cite{DV2010}, there is a 
function $u^+(\omega,t,x)$ such that 
\[
\rho^+(\omega,t,x,\xi)=\En_{u^+(\omega,t,x)>\xi},\qquad 
\text{for a.e.~$(\omega,t,x,\xi)$.}
\]
The function $u^+$ constitutes a kinetic solution according 
to Definition \ref{def: solution}. A similar result holds for 
$\rho^-$. Since $\rho^+(t)=\rho(t)$ for a.e.~$t$, it follows 
that a generalized kinetic solution $\rho$ is in fact a kinetic solution.  

Finally, following \cite{Perthame}, let us establish the $L^1$ contraction 
principle in Theorem \ref{thm:well-posed}. 
Let $u_1, u_2$ be two kinetic solutions with corresponding 
initial data $u_{1,0}, u_{2,0}$ and kinetic measures $m_1, m_2$. 
Set $\rho=\frac12 \left(\En_{u_1>\xi}+\En_{u_2>\xi}\right)$ 
and $\rho_0=\frac12 \left(\En_{u_{1,0}>\xi}+\En_{u_{2,0}>\xi}\right)$. 
Then $\rho$ is a generalized kinetic solution with initial data $\rho_0$, kinetic 
measure $m=\frac12 (m_1+m_2)$, and $\partial_\xi \rho 
=-\frac12 \left(\delta_{u_1}+\delta_{u_2}\right)$.
By \eqref{eq: pre-contraction inequality},
$$
\EE \int_{M\times\R} \left( \rho-\rho^2\right)(t)
\,d\xi \,dV_h(x)
\leq \EE\int_{M\times\R}\left(\rho_0-\rho_0^2\right)
\,d\xi \,dV_h(x),
$$
for a.e.~$t\in [0,T]$.  Using that $\rho_i^2=\rho_i$ 
for $i=1,2$, a simple computation reveals that 
$\rho-\rho^2=\frac14(\rho_1-\rho_2)^2=\frac14 \abs{\rho_1-\rho_2}$ and 
so $\int_{\R}(\rho-\rho^2)\, d\xi=\frac14\abs{u_1-u_2}$.
Similarly, $\int_{\R}(\rho_0-\rho_0^2)\, d\xi
=\frac14\abs{u_{1,0}-u_{2,0}}$. 
Therefore, \eqref{L1-contraction} holds for a.e.~$t\in[0,T]$. 

The trajectories of a kinetic solution $u$ are continuous. 
Indeed, as $u=u^+$ a.e.~with 
respect to $(\omega,t,x)$, it is enough 
to show that $u^+$ admits $\P$-a.s.~continuous 
trajectories in $L^p(M)$. This is achieved by 
replicating the proof in \cite[Cor.~16]{DV2010}. 
Thanks to the continuity of the trajectories, \eqref{L1-contraction} 
must hold for all $t\in[0,T]$.

\subsection{Proofs of technical lemmas}

\begin{proof}[Proof of Lemma 
\ref{lemma: stima per G2 nu epsilon contro rho epsilon e test}]
As usual, we make use of a local argument: fix $\kappa\in\mathcal{A}$ 
and consider $\gamma\in C^0_c(X_\kappa)$. 
Hence, $\gamma\,\,\partial_\xi\psi \, 
\rho^+_{\varepsilon}(\omega,s)\in  C^0_c(X_\kappa\times\R)$ 
and, by means of the coordinates given by $\kappa$,
\begin{align*}
I & :=\bigg|\int_{M\times\R}\gamma(x)\,\partial_\xi\psi(\xi) 
\, \rho^+_{\varepsilon}(\omega,s)(x,\xi)
\frac{(G^2\nu_\kappa)_\varepsilon(\omega,s)(x,\xi)}
{\abs{h_\kappa(x)}^{1/2}}\,d\xi\,dV_h(x)\bigg|\\
&=\bigg|\int_{\tilde{X}_\kappa \times\R}
\gamma(z)\,\partial_\xi\psi(\xi) \, 
\rho^+_{\varepsilon}(\omega,s)(z,\xi)
(G^2\nu_\kappa)_\varepsilon(\omega,s)(z,\xi)\,d\xi\,dz\bigg|
\\ & =\abs{(G^2\nu_\kappa)_{\omega,s}\left((\gamma\,\partial_\xi\psi 
\, \rho^+_{\varepsilon}(\omega,s))_\varepsilon\right)}
\\ & \leq \norm{(\gamma\,\partial_\xi\psi \, 
\rho^+_{\varepsilon}(\omega,s))_\varepsilon}_{L^\infty(\tilde{X}_\kappa \times\R)} 
\\ &\hspace{1cm}\times\int_{S_\varepsilon}
\alpha_\kappa(z)G^2(z,\xi)\,\abs{h_\kappa(z)}^{1/2}\, 
\nu_{\omega,s,\kappa^{-1}(z)}(d\xi)\,dz, 
\end{align*}
where (as before) $(\cdots)_\varepsilon$ means convolution 
between $\cdots$ and $\phi_\varepsilon$, and 
\[
S_\varepsilon:=\supp\,(\gamma\,\partial_\xi\psi 
\, \rho^+_{\varepsilon}(\omega,s))_\varepsilon 
\cap \left(\kappa(\supp\alpha_\kappa)\times\R\right). 
\]
Since $\supp(\gamma\,\partial_\xi\psi \, 
\rho^+_\varepsilon(\omega,s))\subset 
\supp\gamma\times\supp\partial_\xi\psi$,
\[
S_\varepsilon\subset \kappa(\supp\alpha_\kappa)
\times\left(\supp\partial_\xi\psi+[-\varepsilon,\varepsilon]\right). 
\]
Thus, by basic convolution estimates, we obtain
\begin{align*}
I& \leq \norm{\gamma\,\partial_\xi\psi \, 
\rho^+_{\varepsilon}(\omega,s)}_{L^\infty(\tilde{X}_\kappa \times\R)}
\\ &\hspace{0.7cm}\times\int_{\kappa(\supp\alpha_\kappa)}
\alpha_\kappa(z)\int_{\supp\partial_\xi\psi+[-\varepsilon,\varepsilon]} 
G^2(z,\xi)\,\abs{h_\kappa(z)}^{1/2}\,\nu_{\omega,s,\kappa^{-1}(z)}(d\xi)\,dz
\\ &\leq C(M,h,\mathcal{A})\norm{\gamma}_{L^\infty(M)} 
\norm{\partial_\xi\psi}_{L^\infty(\R)} \\
&\hspace{0.7cm}\times\int_M\alpha_\kappa(x)
\int_{\supp\partial_\xi\psi+[-\varepsilon,\varepsilon]}G^2(x,\xi)
\,\nu_{\omega,s,x}(d\xi)\,dV_h(x), 
\end{align*}
which holds for any global $\gamma\in C^0(M)$ as well, since the 
support of $\frac{(G^2\nu_\kappa)_\varepsilon(\omega,s)}
{\abs{h_\kappa}^{1/2}}$ is contained in 
$\kappa^{-1}(\mathcal{U}_\kappa)\times\R$. Summing over 
$\kappa\in\mathcal{A}$, we obtain the desired result.
\end{proof}

\begin{proof}[Proof of Lemma \ref{lemma: simmetry argument}]
Consider two charts $\kappa,\kappa'\in\mathcal{A}$ 
such that $X_\kappa\cap X_{\kappa'}\neq \emptyset$ (to avoid trivialities). 
The coordinates given by $\kappa$ will be denoted $z$ and the ones 
given by $\kappa'$ will be called $w$ instead. Our aim is to 
compute the following quantity for each fixed $k\in\N$, $\xi\in\R$, 
and $(\omega,s)\in F$ (cf.~discussion before Lemma 
\ref{lemma: properties of gk nu kappa epsilon}):
\begin{align*}
\mathcal{S}_k(\kappa,\kappa'):=\int_M\bigg(- & 
\frac{(\nu_\kappa)_\varepsilon(\omega,s)(x,\xi)}{\abs{h_\kappa(x)}^{1/2}}
\frac{(g^2_k\nu_{\kappa'})_\varepsilon(\omega,s)(x,\xi)}{|h_{\kappa'}(x)|^{1/2}}
\\ & +\frac{(g_k\nu_\kappa)_\varepsilon(\omega,s)(x,\xi)}{\abs{h_\kappa(x)}^{1/2}}
\frac{(g_k\nu_{\kappa'})_\varepsilon(\omega,s)(x,\xi)}
{|h_{\kappa'}(x)|^{1/2}}\bigg)
\,dV_h(x). 
\end{align*}
To alleviate the notation, we have not made 
explicit the dependence of $\mathcal{S}_k(\kappa,\kappa')$ 
on $\xi$ and $(\omega,s)$. To integrate we use the 
coordinates given by $\kappa$. Being explicit about the 
definition of the integrand (specifically 
the role of $\kappa, \kappa'$), we have
\begin{align*}
\mathcal{S}_k(\kappa,\kappa') & = \int_M
\bigg(-\frac{(\nu_\kappa)_\varepsilon(\omega,s)(\kappa(x),\xi)}
{|h_\kappa(\kappa(x))|^{1/2}}
\frac{(g^2_k\nu_{\kappa'})_\varepsilon(\omega,s)(\kappa'(x),\xi)}
{|h_{\kappa'}(\kappa'(x))|^{1/2}}
\\ &\hspace{1.6cm}
+\frac{(g_k\nu_\kappa)_\varepsilon(\omega,s)(\kappa(x),\xi)}
{h_\kappa(\kappa(x))|^{1/2}}\frac{(g_k\nu_{\kappa'})_\varepsilon(\omega,s)
(\kappa'(x),\xi)}{|h_{\kappa'}(\kappa'(x))|^{1/2}}\bigg)\,dV_h(x)\\
&=\int_{\kappa(X_\kappa\cap X_{\kappa'})}
\bigg(-(\nu_\kappa)_\varepsilon(\omega,s)(z,\xi)
\frac{(g^2_k\nu_{\kappa'})_\varepsilon(\omega,s)
(\kappa'\circ\kappa^{-1}(z),\xi)}
{|h_{\kappa'}(\kappa'\circ\kappa^{-1}(z))|^{1/2}}
\\ &\hspace{1.6cm}
+(g_k\nu_\kappa)_\varepsilon(\omega,s)(z,\xi)
\frac{(g_k\nu_{\kappa'})_\varepsilon(\omega,s)(\kappa'\circ\kappa^{-1}(z),\xi)}
{|h_{\kappa'}(\kappa'\circ\kappa^{-1}(z))|^{1/2}}\bigg)\,dz
\\ & = \int_{\kappa(X_\kappa\cap X_{\kappa'})}
\frac{|\jac(\kappa'\circ\kappa^{-1})(z)|}{\abs{h_\kappa(z)}^{1/2}}
\\ & \qquad \qquad \qquad
\times \bigg((-(\nu_\kappa)_\varepsilon(\omega,s)(z,\xi)
 (g^2_k\nu_{\kappa'})_\varepsilon(\omega,s)(\kappa'\circ\kappa^{-1}(z),\xi)
\\ & \qquad \qquad \qquad\qquad
+ (g_k\nu_\kappa)_\varepsilon(\omega,s)(z,\xi)
(g_k\nu_{\kappa'})_\varepsilon(\omega,s)
(\kappa'\circ\kappa^{-1}(z),\xi)\bigg) \,dz. 
\end{align*}
Note the attendance of $\kappa, \kappa'$. Indeed, recall that our 
functions are defined on the Euclidean space 
and then lifted to $M$ via the charts $\kappa$ and $\kappa'$.

For convenience set $\Phi:=\kappa'\circ\kappa^{-1}$. 
By definition of the involved quantities, 
\begin{align*}
-& (\nu_\kappa)_\varepsilon(\omega,s)(z,\xi)
(g^2_k\nu_{\kappa'})_\varepsilon(\omega,s)(\kappa'\circ\kappa^{-1}(z),\xi)\\
& \hspace{0.3cm}+(g_k\nu_\kappa)_\varepsilon(\omega,s)(z,\xi)
(g_k\nu_{\kappa'})_\varepsilon(\omega,s)(\kappa'\circ\kappa^{-1}(z),\xi))
\\ & = -\int_{\kappa(\supp \alpha_\kappa)\times\R}\alpha_\kappa(\bar{z})
\, \abs{h_\kappa(\bar{z})}^{1/2}\phi_\varepsilon(z-\bar{z},\xi-\bar{\xi})
\,\nu_{\omega,s,\kappa^{-1}(\bar{z})}(d\bar{\xi})\,d\bar{z} 
\\ & \qquad \quad
\times\int_{\kappa'(\supp \alpha_{\kappa'})\times\R}\alpha_{\kappa'}(\bar{w})
\,\abs{h_{\kappa'}(\bar{w})}^{1/2}g^2_k(\bar{w},\bar{\bar{\xi}})
\\ & \qquad\qquad\quad \qquad\qquad\qquad\qquad 
\times \phi_\varepsilon(\Phi(z)-\bar{w},\xi-\bar{\bar{\xi}})
\,\nu_{\omega,s,\kappa'^{-1}(\bar{w})}(d\bar{\bar{\xi}})\,d\bar{w}
\\ & \qquad
+\int_{\kappa(\supp \alpha_\kappa)\times\R}\alpha_\kappa(\bar{z})
\,\abs{h_\kappa(\bar{z})}^{1/2}g_k(\bar{z},\bar{\xi})
\phi_\varepsilon(z-\bar{z},\xi-\bar{\xi})
\,\nu_{\omega,s,\kappa^{-1}(\bar{z})}(d\bar{\xi})\,d\bar{z} 
\\ & \qquad \qquad \quad
\times\int_{\kappa'(\supp \alpha_{\kappa'})\times\R}\alpha_{\kappa'}(\bar{w})
\,\abs{h_{\kappa'}(\bar{w})}^{1/2}g_k(\bar{w},\bar{\bar{\xi}})
\\ & \qquad\qquad\qquad\qquad\qquad\qquad\qquad 
\times \phi_\varepsilon(\Phi(z)-\bar{w},\xi-\bar{\bar{\xi}})
\,\nu_{\omega,s,\kappa'^{-1}(\bar{w})}(d\bar{\bar{\xi}})\,d\bar{w}
\\ & =
\int_{(\kappa(\supp \alpha_\kappa)\times\R)
\times (\kappa'(\supp \alpha_{\kappa'})\times\R)}
\left[-g^2_k(\bar{w},\bar{\bar{\xi}}) 
+ g_k(\bar{z},\bar{\xi})\,g_k(\bar{w},\bar{\bar{\xi}})\right]
\\ &\hspace{1cm}
\times \phi_\varepsilon(z-\bar{z},\xi-\bar{\xi})
\,\phi_\varepsilon(\Phi(z)-\bar{w},\xi-\bar{\bar{\xi}})
\,(\nu_\kappa)_{\omega,s}\otimes(\nu_{\kappa'})_{\omega,s}
(d\bar{z},d\bar{\xi},d\bar{w},d\bar{\bar{\xi}}). 
\end{align*}

Summarizing our findings so far,
\begin{align*}
\mathcal{S}_k(\kappa,\kappa') & =
\int_{\kappa(X_\kappa\cap X_{\kappa'})}
\frac{\abs{\jac(\Phi)(z)}}{\abs{h_\kappa(z)}^{1/2}} 
\\ & \qquad \times
\bigg(\int_{(\kappa(\supp \alpha_\kappa)\times\R)\times (\kappa'(\supp 
\alpha_{\kappa'})\times\R)}\left[-g^2_k(\bar{w},\bar{\bar{\xi}})
+ g_k(\bar{z},\bar{\xi})\,g_k(\bar{w},\bar{\bar{\xi}})\right]
\\ & \qquad\qquad\qquad\qquad
\times \phi_\varepsilon(z-\bar{z},\xi-\bar{\xi})
\,\phi_\varepsilon(\Phi(z)-\bar{w},\xi-\bar{\bar{\xi}})
\\ & \qquad\qquad\qquad\qquad\qquad 
\times (\nu_\kappa)_{\omega,s}\otimes(\nu_{\kappa'})_{\omega,s}
(d\bar{z},d\bar{\xi},d\bar{w},d\bar{\bar{\xi}})\bigg)\,dz, 
\end{align*}
with $\kappa,\kappa'\in\mathcal{A}$ such that 
$X_\kappa\cap X_{\kappa'}\neq \emptyset$. 

Reversing the roles of $\kappa$ and $\kappa'$, we make 
the following crucial observation:
\begin{align*}
\mathcal{S}_k(\kappa',\kappa)
& = \int_{\kappa(X_\kappa\cap X_{\kappa'})}
\frac{\abs{\jac(\Phi)(z)}}{\abs{h_\kappa(z)}^{1/2}} 
\\ & \qquad \times
\bigg(\int_{(\kappa(\supp \alpha_\kappa)\times\R)
\times (\kappa'(\supp \alpha_{\kappa'})\times\R)}
\left[-g^2_k(\bar{z},\bar{\xi}) + g_k(\bar{z},\bar{\xi}) 
\, g_k(\bar{w},\bar{\bar{\xi}})\right]
\\ & \qquad\qquad\qquad\qquad 
\times \phi_\varepsilon(z-\bar{z},\xi-\bar{\xi})
\,\phi_\varepsilon(\Phi(z)-\bar{w},\xi-\bar{\bar{\xi}})
\\ & \qquad\qquad\qquad\qquad\qquad \times
(\nu_\kappa)_{\omega,s}\otimes(\nu_{\kappa'})_{\omega,s}
(d\bar{z},d\bar{\xi},d\bar{w},d\bar{\bar{\xi}})\bigg)\,dz, 
\end{align*}
where we are again using the coordinates $z$ 
given by $\kappa$ to integrate. 

Therefore, we arrive at the following remarkable identity:
\[
\begin{split}
& \mathcal{S}_k(\kappa,\kappa') + \mathcal{S}_k(\kappa',\kappa)
=\int_{\kappa(X_\kappa\cap X_{\kappa'})}
\frac{\abs{\jac(\Phi)(z)}}{\abs{h_\kappa(z)}^{1/2}}\,
\\ & \quad 
\times \bigg(\int_{(\kappa(\supp \alpha_\kappa)\times\R)\times 
(\kappa'(\supp \alpha_{\kappa'})\times\R)}
-\abs{g_k(\bar{z},\bar{\xi})-g_k(\bar{w},\bar{\bar{\xi}})}^2\\
&\qquad\qquad\qquad  \times \phi_\varepsilon(z-\bar{z},\xi-\bar{\xi})\,
\phi_\varepsilon(\Phi(z)-\bar{w},\xi-\bar{\bar{\xi}})
\\ & \qquad\qquad\qquad\qquad \times 
(\nu_\kappa)_{\omega,s} \otimes(\nu_{\kappa'})_{\omega,s}
(d\bar{z},d\bar{\xi},d\bar{w},d\bar{\bar{\xi}})\bigg)\,dz. 
\end{split}
\]
In view of the previous computations, we can 
write (for $k\in\N$) 
\begin{align*}
\mathcal{C}_k(\varepsilon) & 
:= \int_M  
\left(
-\nu_\varepsilon(\omega,s)(x,\xi)(g_k^2\nu)_\varepsilon(\omega,s)(x,\xi)
+\left((g_k\nu)_\varepsilon(\omega,s)(x,\xi)\right)^2
\right) \,dV_h(x)
\\ & =\sum_{\substack{\kappa,\kappa'\in\mathcal{A} \\ 
X_\kappa\cap X_{\kappa'}\neq \emptyset}}\mathcal{S}_k(\kappa,\kappa')
=\frac12\sum_{\substack{\kappa,\kappa'\in\mathcal{A} \\
X_\kappa\cap X_{\kappa'}\neq \emptyset}}\mathcal{S}_k(\kappa,\kappa')
+\frac12\sum_{\substack{\kappa,\kappa'\in\mathcal{A}\\ 
X_\kappa\cap X_{\kappa'}\neq \emptyset}}\mathcal{S}_k(\kappa',\kappa)
\\ & =-\frac12 \sum_{\substack{\kappa,\kappa'\in\mathcal{A}\\
X_\kappa\cap X_{\kappa'}\neq \emptyset}}
\int_{\kappa(X_\kappa\cap X_{\kappa'})}
\frac{\abs{\jac(\Phi)(z)}}{\abs{h_\kappa(z)}^{1/2}}
\\ &\hspace{0.8cm}\times
\bigg(\int_{(\kappa(\supp \alpha_\kappa)\times\R)
\times (\kappa'(\supp \alpha_{\kappa'})\times\R)}
\abs{g_k(\bar{z},\bar{\xi})-g_k(\bar{w},\bar{\bar{\xi}})}^2
\\ &\hspace{3.0cm}
\times \phi_\varepsilon(z-\bar{z},\xi-\bar{\xi})
\,\phi_\varepsilon(\Phi(z)-\bar{w},\xi-\bar{\bar{\xi}})
\\ &\hspace{4.0cm}
\times (\nu_\kappa)_{\omega,s}\otimes(\nu_{\kappa'})_{\omega,s}
(d\bar{z},d\bar{\xi},d\bar{w},d\bar{\bar{\xi}})\bigg)\,dz, 
\end{align*}
where we have suppressed the dependency on $\omega, s, \xi$ 
in $\mathcal{C}_k(\varepsilon)$. 
Note $\mathcal{C}_k(\varepsilon)\le 0$.

Let us assume for the moment that each 
$g_k\in C^0(\R)$ is independent of $x\in M$. 
Then, summing over $k\in\N$, using \eqref{eq: gk are lipschitz} and 
Lemma \ref{properties of the regularized global objects}, we obtain
\begin{equation}\label{eq: simmetry argument}
\begin{split}
\mathcal{C}(\varepsilon) & :=
\int_M \left(-\nu_\varepsilon(\omega,s)(x,\xi)
(G^2\nu)_\varepsilon(\omega,s)(x,\xi)
+\sum_{k\geq 1}\left((g_k\nu)_\varepsilon
(\omega,s)(x,\xi)\right)^2\right)\,dV_h(x)
\\ &=\sum_{k\geq 1} \mathcal{C}_k(\varepsilon)
\\ & \geq-\frac{D_2}{2} \sum_{\substack{\kappa,\kappa'\in\mathcal{A}
\\ X_\kappa\cap X_{\kappa'}\neq \emptyset}}
\int_{\kappa(X_\kappa\cap X_{\kappa'})}
\frac{\abs{\jac(\Phi)(z)}}{\abs{h_\kappa(z)}^{1/2}}
\\ &\hspace{1.6cm}
\times\bigg(\int_{(\kappa(\supp \alpha_\kappa)\times\R)\times 
(\kappa'(\supp \alpha_{\kappa'})\times\R)}
\abs{\bar{\xi}-\bar{\bar{\xi}}}^2\phi_\varepsilon(z-\bar{z},\xi-\bar{\xi})
\\ & \hspace{2.3cm}
\times \phi_\varepsilon(\Phi(z)-\bar{w},\xi-\bar{\bar{\xi}})
\,(\nu_\kappa)_{\omega,s}\otimes(\nu_{\kappa'})_{\omega,s}
(d\bar{z},d\bar{\xi},d\bar{w},d\bar{\bar{\xi}}) \bigg)\,dz\\
&= -\frac{D_2}{2}\sum_{\substack{\kappa,\kappa'\in\mathcal{A}
\\ X_\kappa\cap X_{\kappa'}\neq \emptyset}}
\int_{\kappa(X_\kappa\cap X_{\kappa'})}
\frac{\abs{\jac(\Phi)(z)}}{\abs{h_\kappa(z)}^{1/2}}
\\ &\hspace{0.5cm}
\times \int_{\kappa(\supp \alpha_\kappa)\times 
\kappa'(\supp \alpha_{\kappa'})}\alpha_\kappa(\bar{z})\alpha_{\kappa'}(\bar{w})
\abs{h_\kappa(\bar{z})}^{1/2}\abs{h_{\kappa'}(\bar{w})}^{1/2}
\\ & \hspace{1.1cm} 
\times \varepsilon^{-n}\phi_1\left(\frac{z-\bar{z}}{\varepsilon}\right)
\varepsilon^{-n}\phi_1\left(\frac{\Phi(z)-\bar{w}}{\varepsilon}\right)
\Lambda_{\varepsilon,\kappa,\kappa'}(\omega,s,\bar{z},\bar{w},\xi)
\,d\bar{z}\,d\bar{w}\,dz,
\end{split}
\end{equation}
where 
\begin{align*}
& \Lambda_{\varepsilon,\kappa,\kappa'}(\omega,s,\bar{z},\bar{w},\xi)
\\ & \quad := \int_{\R^2}\abs{\bar{\xi}-\bar{\bar{\xi}}}^2
\varepsilon^{-1}\phi_2\left(\frac{\xi-\bar{\xi}}{\varepsilon}\right)
\varepsilon^{-1}\phi_2\left(\frac{\xi-\bar{\bar{\xi}}}{\varepsilon}\right)
\nu_{\omega,s,\kappa^{-1}(\bar{z})}\otimes\nu_{\omega,s,\kappa'^{-1}(\bar{w})}
(d\bar{\xi},d\bar{\bar{\xi}}). 
\end{align*}

Let $\psi\in \mathcal{D}(\R)$, $\psi\geq 0$. 
In view of Lemma \ref{lemma: strumentale 2} (in the appendix), 
\[
0\leq\int_\R\psi(\xi)\,
\Lambda_{\varepsilon,\kappa,\kappa'}(\omega,s,\bar{z},\bar{w},\xi)
\,d\xi\leq 4\norm{\psi}_{L^\infty}\norm{\phi_2}_{L^\infty} \varepsilon, 
\]
uniformly in $\omega,s,\bar{z},\bar{w},\kappa,\kappa'$.

Hence, integrating \eqref{eq: simmetry argument} 
with respect to $\psi\,d\xi$ leads to
\begin{equation}\label{eq: simmetry argument 2}
\begin{split}
0 & \geq \int_{\R} \psi(\xi) \mathcal{C}(\varepsilon)\,d\xi
\geq -\varepsilon\, C_\psi
\sum_{\substack{\kappa,\kappa'\in\mathcal{A}
\\ X_\kappa\cap X_{\kappa'}\neq \emptyset}}
\int_{\kappa(X_\kappa\cap X_{\kappa'})}
\frac{\abs{\jac(\Phi)(z)}}{\abs{h_\kappa(z)}^{1/2}}
\\ &\hspace{0.6cm}
\times\int_{\kappa(\supp \alpha_\kappa)\times \kappa'(\supp 
\alpha_{\kappa'})}\alpha_\kappa(\bar{z})\alpha_{\kappa'}(\bar{w})
\abs{h_\kappa(\bar{z})}^{1/2}\abs{h_{\kappa'}(\bar{w})}^{1/2}
\\ &\hspace{2.5cm}\times
\varepsilon^{-n}\phi_1\left(\frac{z-\bar{z}}{\varepsilon}\right)
\varepsilon^{-n}\phi_1\left(\frac{\Phi(z)-\bar{w}}{\varepsilon}\right)
\,d\bar{z}\,d\bar{w}\,dz
\\ &\geq-\varepsilon\, C_\psi\, C_{M,h,\mathcal{A}}
\sum_{\substack{\kappa,
\kappa'\in\mathcal{A}\\X_\kappa\cap X_{\kappa'}\neq \emptyset}}
\int_{\kappa(X_\kappa\cap X_{\kappa'})}
\frac{\abs{\jac(\Phi)(z)}}{\abs{h_\kappa(z)}^{1/2}}\,dz
\ge -\varepsilon\,C',
\end{split}
\end{equation}
for some ($\varepsilon$-independent) constants 
$C_\psi, C_{M,h,\mathcal{A}},C'\ge 0$. 
This inequality holds for a.e.~$(\omega,s)\in\Omega_T$. 
Integrating \eqref{eq: simmetry argument 2} 
over $\Omega\times[0,t]$, we obtain
$$
0\geq \EE\int_0^t \int_{\R} \psi(\xi) 
\mathcal{C}(\varepsilon)\,d\xi\geq -\varepsilon\, C'\,t
$$
concluding the proof of Lemma \ref{lemma: simmetry argument} 
in the case where each $g_k$ is independent of $x$. 
 
Let us consider the general ($x$-dependent case).  First of all, we 
may assume that for any $\kappa\in\mathcal{A}$ there exists 
$C_\kappa>1$ such that $h_\kappa(x)\leq C_\kappa \mbox{Id}$ in the sense 
of bilinear forms for all $x\in X_\kappa$, where $\mbox{Id}$ is the 
identity matrix on $\R^n$. Indeed, if it is not the case, we can 
choose for any $x\in M$ normal coordinates centered at $x$, and by 
using the fact the $M$ is compact, we obtain 
a finite atlas $\mathcal{A}$ for which the above property is fulfilled, and 
furthermore $\tilde{X}_\kappa $ is convex. It follows that 
$$
d_h(x,y)\leq C_\kappa \abs{\kappa(x)-\kappa(y)},
\qquad x,y\in X_\kappa,
$$
where $\abs{\cdot}$ is the Euclidean distance on $\R^n$.  
As a result, \eqref{eq: gk are lipschitz} takes the form 
\begin{equation}\label{eq: gk are lipschitz wrt the euclidean distance}
\sum_{k\geq 1}
\abs{g_k(\bar{z},\bar{\xi})-g_k(\bar{w},\bar{\bar{\xi}})}^2
\leq D_2\left(C_\kappa\,\abs{\bar{z}-\Phi^{-1}(\bar{w})}^2 
+ \abs{\bar{\xi}-\bar{\bar{\xi}}}^2\right), 
\end{equation}
for all $\bar{z}\in X_\kappa,\bar{w}\in\kappa'(X_\kappa\cap X_{\kappa'})$, 
and $\bar{\xi},\bar{\bar{\xi}}\in\R$.

By direct inspection of the expression for 
$\mathcal{S}_k(\kappa,\kappa')$, the integral is actually zero if 
$\kappa^{-1}(\mathcal{U}_\kappa)\cap \kappa'^{-1}(\mathcal{U}_{\kappa'})
=\emptyset$. Therefore, the expression must be studied 
only for $z\in\mathcal{U}_\kappa\cap\Phi^{-1}(\mathcal{U}_{\kappa'})$, 
$\bar{z}\in B_\varepsilon(z)$, and $\bar{w}\in B_\varepsilon(\Phi(z))$. 
By the construction of $\mathcal{U}_\kappa$ and the choice of $\varepsilon$, 
$B_\varepsilon(z)\subset\tilde{X}_\kappa$ and $
B_\varepsilon(\Phi(z))\subset\kappa'(X_\kappa\cap X_{\kappa'})$. 
Hence, for all such points, \eqref{eq: gk are lipschitz wrt the 
euclidean distance} clearly holds.

Restarting from \eqref{eq: simmetry argument}, we now obtain 
\begin{equation*}
\begin{split}
\mathcal{C}(\varepsilon)
& \geq \mathcal{I}(\varepsilon) -\frac{D_2}{2}
\max_\kappa C_\kappa \sum_{\substack{\kappa,\kappa'\in\mathcal{A}\\
X_\kappa\cap X_{\kappa'}\neq \emptyset}}
\int_{\mathcal{U}_\kappa\cap\Phi^{-1}(\mathcal{U}_{\kappa'})}
\frac{\abs{\jac(\Phi)(z)}}{\abs{h_\kappa(z)}^{1/2}}
\\ &\hspace{0.6cm}\times\bigg(
\int_{(B_\varepsilon(z)\times\R)\times (B_\varepsilon(\Phi(z))\times\R)}
\abs{\bar{z}-\Phi^{-1}(\bar{w})}^2 \, 
\phi_\varepsilon(z-\bar{z},\xi-\bar{\xi})
\\ &\hspace{1.3cm}
\times \phi_\varepsilon(\Phi(z)-\bar{w},\xi-\bar{\bar{\xi}})
\,(\nu_\kappa)_{\omega,s}\otimes(\nu_{\kappa'})_{\omega,s}
(d\bar{z},d\bar{\xi},d\bar{w},d\bar{\bar{\xi}})\bigg)\,dz
\\ & =\mathcal{I}(\varepsilon)
-\frac{D_2}{2}\max_\kappa C_\kappa\sum_{\substack{\kappa,
\kappa'\in\mathcal{A}\\X_\kappa\cap X_{\kappa'}\neq \emptyset}}
\int_{\mathcal{U}_\kappa\cap\Phi^{-1}(\mathcal{U}_{\kappa'})}
\frac{\abs{\jac(\Phi)(z)}}{\abs{h_\kappa(z)}^{1/2}}\\
&\hspace{0.5cm} 
\times \int_{B_\varepsilon(z)\times B_\varepsilon(\Phi(z))}
\abs{\bar{z}-\Phi^{-1}(\bar{w})}^2\,\alpha_\kappa(\bar{z})
\alpha_{\kappa'}(\bar{w})\abs{h_\kappa(\bar{z})}^{1/2}
\abs{h_{\kappa'}(\bar{w})}^{1/2} \\
&\hspace{1.1cm}\times\varepsilon^{-n}
\phi_1\left(\frac{z-\bar{z}}{\varepsilon}\right)
\varepsilon^{-n}\phi_1\left(\frac{\Phi(z)-\bar{w}}{\varepsilon}\right)
\bar{\Lambda}_{\varepsilon,\kappa,\kappa'}(\omega,s,\bar{z},\bar{w},\xi)
\,d\bar{z}\,d\bar{w}\,dz, 
\end{split}
\end{equation*}
where $\mathcal{I}(\varepsilon)$ denotes the corresponding term in 
the $x$-independent case, and
\begin{align*}
& \bar{\Lambda}_{\varepsilon,\kappa,\kappa'}(\omega,s,\bar{z},\bar{w},\xi)
\\ & \quad
:=  \int_{\R^2}\varepsilon^{-1}
\phi_2\left(\frac{\xi-\bar{\xi}}{\varepsilon}\right)
\varepsilon^{-1}\phi_2\left(\frac{\xi-\bar{\bar{\xi}}}{\varepsilon}\right)
\nu_{\omega,s,\kappa^{-1}(\bar{z})}\otimes\nu_{\omega,s,\kappa'^{-1}(\bar{w})}
(d\bar{\xi},d\bar{\bar{\xi}}). 
\end{align*}

Let $0\le \psi\in \mathcal{D}(\R)$. 
Arguing as in the proof of Lemma \ref{lemma: strumentale 2}, 
\[
0\leq \int_\R 
\psi(\xi)\,\bar{\Lambda}_{\varepsilon,\kappa,\kappa'}
(\omega,s,\bar{z},\bar{w},\xi)\,d\xi
\leq \norm{\psi}_{L^\infty}\norm{\phi_2}_{L^\infty}
\,\varepsilon^{-1}, 
\]
uniformly in $\omega,s,\bar{z},\bar{w},\kappa, \kappa'$. 
There is a constant $C_\Phi$ 
such that $|\bar{z}-\Phi^{-1}(\bar{w})|
\leq (C_\Phi+1)\varepsilon$, for all $\bar{z}\in B_\varepsilon(z)$ 
and $\bar{w}\in B_\varepsilon(\Phi(z))$. Hence, integrating 
the above lower bound on $\mathcal{C}(\varepsilon)$ 
against $\psi\,d\xi$ we obtain
\begin{equation*}
\begin{split}
0 & \geq \int_{\R} \psi(\xi) \mathcal{C}(\varepsilon)\,d\xi
\\ & \geq-\varepsilon C'-C_\psi\frac{D_2}{2}\max_\kappa C_\kappa
\sum_{\substack{\kappa,\kappa'\in\mathcal{A}\\
X_\kappa\cap X_{\kappa'}\neq \emptyset}}
\int_{\mathcal{U}_\kappa\cap\Phi^{-1}(\mathcal{U}_{\kappa'})}
\frac{\abs{\jac(\Phi)(z)}}{\abs{h_\kappa(z)}^{1/2}}
\,C_{M,h,\mathcal{A}}\, \varepsilon^2\,\varepsilon^{-1}\,dz
\\ & \geq -\varepsilon\,C'-\varepsilon C'', 
\end{split}
\end{equation*}
for some constants $C_{M,h,\mathcal{A}}, C''$, whereas 
$C'$ derives from \eqref{eq: simmetry argument 2}. 
Integrating over $\Omega\times [0,t]$, we conclude 
as in the $x$-independent case.
\end{proof}

\begin{proof}[Proof of Lemma 
\ref{lemma: un ennesimo lemma strumentale}]

Set $J_N:=\left\{\xi\in\R: N\leq\abs{\xi}\leq N+1\right\}$. 
We have
\begin{align*}
\mathcal{I}(\varepsilon) 
& := \abs{-2\EE\int_{M\times\R}\int_0^t \partial_\xi\psi_N\,
\rho^+_{\varepsilon}(\omega,s_-)(x,\xi)
\,m_\varepsilon(\omega,ds)(x,\xi) \,d\xi\,dV_h(x)}
\\ &\hspace{0.5cm}\leq 4\,C(M,h,\mathcal{A})\,
\EE\int_{M\times J_N}\int_0^t 
\abs{m_\varepsilon(\omega,ds)(x,\xi)} \,d\xi\,dV_h(x), 
\end{align*}
where $\abs{m_\varepsilon(\omega,ds)(x,\xi)}$ is the total 
variation of the $s$-measure $m_\varepsilon(\omega,ds)(x,\xi)$. 
Since $s\mapsto m_\varepsilon(\omega,s)(x,\xi)$ is 
non-decreasing, the expression on the right-hand side equals
\[
4\,C(M,h,\mathcal{A})\,\EE\int_{M\times J_N} 
m_\varepsilon(\omega,t)(x,\xi) \,d\xi\,dV_h(x). 
\] 
Since $M\times J_N$ is compact, weak convergence 
of measures implies that
\[
\limsup_{\varepsilon\to 0}\int_{M\times J_N} 
m_\varepsilon(\omega,t)(x,\xi) \,d\xi\,dV_h(x)
\leq \int_{[0,t]\times M\times J_N} m(ds,dx,d\xi), 
\]
for any $(\omega,t)\in\Omega_T$ and $N\in\N$. 
Suppose that there exists a dominant integrable 
$H=H(\omega)$ such that
\[
\int_{M\times J_N} m_\varepsilon(\omega,t)(x,\xi) 
\,d\xi\,dV_h(x) \leq H(\omega), 
\]
uniformly in $\varepsilon$. Then, by the reverse Fatou lemma, we infer
\begin{align*}
\limsup_{\varepsilon\to 0} \mathcal{I}(\varepsilon)
& \leq 4\,C(M,h,\mathcal{A})\,\limsup_{\varepsilon\to 0}
\EE\int_{M\times J_N} m_\varepsilon(\omega,t)(x,\xi) 
\,d\xi\,dV_h(x)
\\ &
\leq 4\,C(M,h,\mathcal{A})\,\EE\limsup_{\varepsilon\to 0}
\int_{M\times J_N} m_\varepsilon(\omega,t)(x,\xi) \,d\xi\,dV_h(x)
\\ &
\leq 4\,C(M,h,\mathcal{A})\,
\EE\int_{[0,t]\times M\times J_N} m(ds,dx,d\xi), 
\end{align*}
valid for any $t\in[0,T]$ and $N\in\N$. Since the kinetic measure 
$m$ vanishes for large $\xi$, the last quantity converges 
to zero as $N\to\infty$. 

We are hence left to find a suitable $H$. We argue like this: let 
$\kappa\in\mathcal{A}$, $(\omega,t)\in\Omega_T$, 
and $\varepsilon<\bar{\varepsilon}$. 
For $\psi \in C^0_c(\R^n\times\R)$, $0\le \psi\le 1$, we 
compute (functions are identified with measures)
\begin{align*}
((\alpha_\kappa m)_\sharp)_\varepsilon(\omega,t)(\psi)
&=\int_{\R^n\times\R}\psi_\varepsilon(z,\xi)
(\alpha_\kappa m)_\sharp(\omega,t)(dz,d\xi)\\
&\leq \norm{\psi_\varepsilon}_{L^\infty}
(\alpha_\kappa m)_\sharp(\omega,t)(\R^n\times\R)\\
&\leq \norm{\psi}_{L^\infty}
(\alpha_\kappa m)_\sharp(\omega,t)(\R^n\times\R)\\
&\leq (\alpha_\kappa m)_\sharp(\omega,t)(\R^n\times\R), 
\end{align*}
where, as usual, $(\cdots)_\varepsilon$ 
means convolution between $\cdots$ and $\phi_\varepsilon$. 
Thus, taking the supremum over such $\psi$, we infer
\[
((\alpha_\kappa m)_\sharp)_\varepsilon(\omega,t)(\R^n\times\R)
\leq (\alpha_\kappa m)_\sharp(\omega,t)(\R^n\times\R). 
\]
On the other hand, 
$(\alpha_\kappa m)_\sharp(\omega,t)(\R^n\times\R)
=(\alpha_\kappa m)_\sharp(\omega,t)(\tilde{X}_\kappa \times\R)$, and thus
\[
((\alpha_\kappa m)_\sharp)_\varepsilon(\omega,t)(\tilde{X}_\kappa \times\R)
\leq (\alpha_\kappa m)_\sharp(\omega,t)(\tilde{X}_\kappa \times\R), 
\]
for any $(\omega,t)\in\Omega_T$, $\varepsilon<\bar{\varepsilon}$ 
and $\kappa\in\mathcal{A}$. As a consequence,
\begin{align*}
&\int_{M\times\R}\frac{((\alpha_\kappa m)_\sharp)_\varepsilon
(\omega,t)(x,\xi)}{\abs{h_\kappa(x)}^{1/2}}\, d\xi\,dV_h(x) 
=((\alpha_\kappa m)_\sharp)_\varepsilon(\omega,t)(\tilde{X}_\kappa \times\R) 
\\ & \qquad \leq (\alpha_\kappa m)_\sharp
(\omega,t)(\tilde{X}_\kappa \times\R) 
= \int_{[0,t]\times M\times\R}\alpha_\kappa(x) m(ds,dx,d\xi), 
\end{align*}
and, by summing over $\kappa\in\mathcal{A}$,
\[
\int_{M\times\R} m_\varepsilon(\omega,t)(x,\xi)
\,d\xi\,dV_h(x) \leq\int_{[0,t]\times M\times\R}m(ds,dx,d\xi).
\]
Therefore, for all $(\omega,t)\in\Omega_T$, 
$\varepsilon<\bar{\varepsilon}$, and $N\in\N$,
\begin{align*}
& \int_{M\times J_N} m_\varepsilon(\omega,t)(x,\xi) \,d\xi\,dV_h(x) 
\leq \int_{M\times\R} m_\varepsilon(\omega,t)(x,\xi)\,d\xi\,dV_h(x)
\\ & \qquad 
\leq \int_{[0,t]\times M\times\R}m(ds,dx,d\xi)
\leq\int_{[0,T]\times M\times\R}m(ds,dx,d\xi), 
\end{align*}
which is integrable in $\omega$.
\end{proof}

\section{Stochastic parabolic problem}\label{sec: Vanishing}

As a step towards constructing solutions to \eqref{eq:target}, we 
study a problem where we add to \eqref{eq:target} 
a small diffusion term $\varepsilon\, \Delta_h$ ($\varepsilon>0$) 
involving the Laplace-Beltrami operator $\Delta_h$ on $(M,h)$. 
For a $C^2(M)$ function $u$,
$$
\Delta_h u := \Div\Grad u = 
h^{ab}\left(\partial_{ab}u-\Gamma_{ab}^c\partial_cu\right),
\qquad \text{in local coordinates}.
$$
For fixed $\varepsilon>0$, we consider the parabolic problem
\begin{equation}\label{eq:target2}
\begin{split}
& du^\varepsilon + \Div \, f_x(u^\varepsilon) \, dt 
= \varepsilon\, \Delta_hu^\varepsilon \, dt 
+ B(u^\varepsilon)\, dW(t), \;\;\; x\in M, \; t\in (0,T),\\
& u^\varepsilon(0,x) = u_0^\varepsilon(x), 
\;\;\; \text{$x\in M$}.
\end{split}
\end{equation} 

In this section we assume the following strengthened 
conditions on the flux $f$, in comparison with 
\eqref{eq: growth condition 0}:
\begin{equation}\label{eq: growth condition 1}
\begin{cases}
\abs{f_x(\xi)}_h \leq C_0(1+\abs{\xi}^r), & 
\abs{\xi}\leq L, \; x\in M, \\
\abs{f_x(\xi)}_h \leq C_0\abs{\xi}, & 
\abs{\xi}> L, \; x\in M,
\end{cases}
\end{equation}
for some positive constants $C_0,r,L$. Moreover, we suppose
\begin{equation}\label{eq: growth condition 2}
\abs{f_x(\xi_1)-f_x(\xi_2)}_h \leq C_1\abs{\xi_1-\xi_2},  
\quad  \xi_1,\xi_2\in\R,\; x\in M.
\end{equation}
for some constant $C_1>0$. 

We wish to establish the existence and uniqueness of a 
solution for suitable initial data $u_0^\varepsilon$. 
We interpret \eqref{eq:target2} in the
``variational framework'', and look for 
so-called variational solutions \cite{Krylov/Rozowskii}. 
Since $\Ltwo$ is a separable Hilbert space 
and $\Hone$ is reflexive and continuously 
and densely embedded in $\Ltwo$ \cite{Hebey}, 
we consider the Gelfand triple
\[
V\subset H = H^\ast \subset V^\ast, 
\]
with $V:=\Hone$ and $H:=\Ltwo$. 
As usual, the duality pairing between $V$ and its 
dual $V^\ast$ will be denoted by $\coppia{\cdot}{\cdot}$; 
moreover, observe that
\[
\coppia{u}{v}= (u,v)_H,  \;\;\; u\in H, \,v\in V,
\]
where $(u,v)_H=\int_Mu\,v\,dV_h$ is the $\Ltwo$-scalar product.

For $\varepsilon>0$ and $u\in\Hone$, we 
define $A^\varepsilon(u)\in V^\ast$ by
\begin{equation*}
\coppia{A^\varepsilon(u)}{v}
:= \int_M\left(-\varepsilon\nabla u 
+ f_x(u),\nabla v\right)_h\, dV_h(x),  
\quad v\in\Hone. 
\end{equation*}
We observe that, in view of \eqref{eq: growth condition 1}, the 
right-hand side is well-defined, because for any 
$v\in\Hone$ with $\norm{v}_V\leq 1$,
\begin{equation}\label{eq: boundness of Aepsilon}
\begin{split}
\left|\coppia{A^\varepsilon(u)}{v}\right| 
& \leq \varepsilon\norm{\nabla u}_{L^2(M)}\norm{\nabla v}_{L^2(M)} 
+ C_0(1+L^r)\norm{\nabla v}_{L^1(M)}\\
&\hspace{1cm}+C_0\norm{u}_{L^2(M)}\norm{\nabla v}_{L^2(M)}\\
&\leq \varepsilon\norm{\nabla u}_{L^2(M)} + C_0(1+L^r)+C_0\norm{u}_{L^2(M)}, 
\end{split}
\end{equation}
where in the first line we used $\norm{\nabla v}_{L^1(M)}
\leq \norm{\nabla v}_{L^2(M)}$ (recall $\mathrm{Vol}(M,h)=1$). 
We also observe that, in view of \eqref{eq: growth condition 2}, 
the map $A^\varepsilon$ is Lipschitz from $V$ to $V^\ast$, since
\begin{align*}
\norm{A^\varepsilon(u_1)-A^\varepsilon(u_2)}_{V^\ast}
&=\sup_{\norm{V}_V\leq 1}\coppia{A^\varepsilon(u_1)-A^\varepsilon(u_2)}{v}\\
&\leq \varepsilon \norm{\nabla(u_1-u_2)}_{L^2(M)}+C_1\norm{u_1-u_2}_{L^2(M)}\\
&\leq (\varepsilon+ C_1)\, \norm{u_1-u_2}_{\Hone}. 
\end{align*}

In particular, $A^\varepsilon$ is 
$\mathcal{B}(V)/\mathcal{B}(V^\ast)$ measurable, and 
hence progressively measurable, being independent 
of $\omega\in\Omega$ and $t\in [0,T]$. The next two lemmas show that the 
assumptions \eqref{eq: growth condition 1}, \eqref{eq: 
growth condition 2}, \eqref{eq: definition of G}, 
and \eqref{eq: gk are lipschitz} guarantee that we 
may use the variational framework of \cite{PR} to solve \eqref{eq:target2}.

\begin{lemma}\label{lemma: Aepsilon satisfies H1-H4}
Assume that $f=f_x(\xi)$ satisfies \eqref{eq: growth condition 1} 
and \eqref{eq: growth condition 2}. Then, for 
any $\varepsilon>0$, the map $A^\varepsilon:V\to V^\ast$ is 
hemicontinuous, weakly monotone, coercive, and bounded.
\end{lemma}

\begin{proof}
Let us check hemicontinuity. Fix $u,v,w\in \Hone$, 
and let $\lambda\in\R$. Then
\begin{align*}
\coppia{A^\varepsilon(u+\lambda v)}{w} 
& = \int_M-\varepsilon\left(\nabla u + 
\lambda\nabla v,\nabla w\right)_h \, dV_h(x) \\
&\hspace{0.5cm} 
+ \int_M\left(f_x(u+\lambda v),\nabla w\right)_h \, dV_h(x) 
=: I_1(\lambda) + I_2(\lambda).   
\end{align*}
Clearly, $I_1(\lambda)$ is linear in $\lambda$ and 
hence continuous. For $I_2(\lambda)$ 
we argue as follows: if $\lambda \to \lambda_0$, then, by 
smoothness of $f_x(\xi)$, it clearly follows that 
\[
\left(f_x(u(x)+\lambda v(x)),\nabla w(x)\right)_h 
\overset{\lambda\to\lambda_0}
{\longrightarrow} \left(f_x(u(x)+\lambda_0 v(x)),
\nabla w(x)\right)_h,
\]
for every $x\in M$. In view of \eqref{eq: growth condition 1}, we 
can find a dominating function and conclude via the 
dominated convergence theorem. Hence, $I_2(\lambda)$ is continuous.

To verify weak monotonicity, we take $u,v\in\Hone$ 
and observe that
\begin{align*}
{}_{V^\ast}\langle & 
A^\varepsilon(u)-A^\varepsilon(v),u-v\rangle_{V}\\
& = -\varepsilon \norm{\nabla u-\nabla v}^2_{L^2(M)}
+ \int_M\left(f_x(u)-f_x(v),\nabla u-\nabla v\right)_h dV_h(x)\\
& \leq -\varepsilon \norm{\nabla u-\nabla v}^2_{L^2(M)} 
+ C_1\norm{u-v}_{L^2(M)}\norm{\nabla u -\nabla v}_{L^2(M)}
\end{align*}
by means of \eqref{eq: growth condition 2} and 
H\"older's inequality. With the help of Cauchy's 
inequality ``with $\varepsilon$'', we obtain
\begin{align*}
\coppia{A^\varepsilon(u)-A^\varepsilon(v)}{u-v}
& \leq -\varepsilon \norm{\nabla u-\nabla v}^2_{L^2(M)} 
+ \frac{C_1^2}{2\varepsilon}\norm{u-v}^2_{L^2(M)}\\
&\hspace{1cm} +\frac{\varepsilon}{2}
\norm{\nabla u -\nabla v}^2_{L^2(M)}\\
&\leq \frac{C_1^2}{2\varepsilon}\norm{u-v}^2_{L^2(M)}. 
\end{align*}

Next we examine coercivity. For any $v\in V$, thanks 
to \eqref{eq: growth condition 1},
\begin{align*}
\coppia{A^\varepsilon(v)}{v} & = 
-\varepsilon \norm{\nabla v}^2_{L^2(M)} 
+ \int_M\left(f_x(v),\nabla v\right)_h dV_h(x)
\\ &\leq -\varepsilon \norm{\nabla v}^2_{L^2(M)} 
+ C_0(1+L^r)\norm{\nabla v}_{L^1(M)}\\
&\hspace{0.5cm} +C_0\norm{V}_{L^2(M)}\norm{\nabla v}_{L^2(M)} . 
\end{align*}
Once again by Cauchy's inequality ``with $\varepsilon$'',
\begin{align*}
\coppia{A^\varepsilon(v)}{v} 
&\leq -\varepsilon \norm{\nabla v}^2_{L^2(M)} 
+ \frac{\varepsilon}{4} \norm{\nabla v}^2_{L^2(M)}
+ \frac{C_0^2(1+L^r)^2}{\varepsilon}\\
& \hspace{1.5cm} + \frac{\varepsilon}{4} \norm{\nabla v}^2_{L^2(M)} 
+ \frac{C_0^2}{\varepsilon}\norm{v}^2_{L^2(M)} 
\\ &= \left(\frac{C_0^2}{\varepsilon}+\frac{\varepsilon}{2}\right)
\norm{v}^2_{L^2(M)} -\frac{\varepsilon}{2}\norm{v}^2_{\Hone}
+\frac{C_0^2(1+L^r)^2}{\varepsilon}.   
\end{align*}
Finally, it is clear that \eqref{eq: boundness of Aepsilon} 
implies, for any $u\in\Hone$,
\begin{align*}
\norm{A^\varepsilon(u)}_{V^\ast} & \leq 
\varepsilon\norm{\nabla u}_{L^2(M)} 
+ C_0(1+L^r)+C_0\norm{u}_{L^2(M)}\\
&\leq (\varepsilon + C_0)\norm{u}_{\Hone} + C_0(1+L^r). 
\end{align*}
This concludes the proof of the lemma.
\end{proof}

\begin{lemma}\label{lemma: B is Lipschitz on H1}
Assume \eqref{eq: definition of G} 
and \eqref{eq: gk are lipschitz} hold. 
Then $B$ is a Lipschitz mapping from $\Hone$ to $\hS$.
\end{lemma}

\begin{proof}
From \eqref{eq: B is Lipschitz on L2} we immediately obtain
\[
\norm{B(z_1)-B(z_2)}_{\hS}^2 \leq D_2\norm{z_1-z_2}_{\Hone}^2, 
\]
for any $z_1,z_2\in\Hone$; hence $B$ is 
Lipschitz with constant $\sqrt{D_2}$.
\end{proof}

Therefore, in view of the general results 
in \cite[Theorems 2.1 \& 2.2, p.~1253]{Krylov/Rozowskii}, we infer the existence 
and uniqueness of a variational solution to the 
stochastic parabolic problem \eqref{eq:target2}.

\begin{theorem}[well-posedness]\label{thm: existence variational solution}
Suppose conditions \eqref{eq: growth condition 1}, 
\eqref{eq: growth condition 2}, \eqref{eq: definition of G}, 
and \eqref{eq: gk are lipschitz} hold. 
Fix $\varepsilon>0$, and let $u_0^\varepsilon\in L^2(\Omega;\Ltwo)$ 
be an $\mathcal{F}_0$-measurable initial function. 
Then there exists a unique solution $u^\varepsilon$ to \eqref{eq:target2}, 
namely a continuous $\Ltwo$-valued $\{\mathcal{F}_t\}_{t\in[0,T]}$-adapted 
process $(u^\varepsilon(t))_{t\in[0,T]}$ such that 
$u^\varepsilon\in\Hone$ for $\P\otimes dt$-a.e.~$(\omega,t)\in\Omega_T$, 
$u^\varepsilon\in\LOT{2}{\Hone}\cap\LOT{2}{\Ltwo}$, and $\P$-a.s.
\begin{equation*}
\begin{split}
& u^\varepsilon(t)= u^\varepsilon_0 
+ \int_0^t\left( \int_M\left(-\varepsilon\nabla 
u^\varepsilon(s) + f_x(u^\varepsilon(s)),\cdot\right)_h 
\, dV_h(x)\right) \, ds  
\\ & \qquad\qquad\qquad
+\int_0^t B(u^\varepsilon(s)) \,dW(s), 
\qquad t\in[0,T], 
\end{split} 
\end{equation*}
where the equation is understood as 
equality between elements of $V^\ast$. Moreover, 
\[
\EE\left[\sup_{t\in[0,T]}\norm{u^\varepsilon(t)}_{L^2(M)}^2\right]<\infty, 
\]
and the following It\^{o} formula holds for 
the square of the $\Ltwo$-norm:
\begin{equation}\label{eq: Ito p=2}
\begin{split}
& \norm{u^\varepsilon(t)}_{L^2(M)}^2 
= \norm{u^\varepsilon_0}_{L^2(M)}^2
\\ & \quad + 2\int_0^t\left(-\varepsilon 
\norm{\nabla u^\varepsilon(s)}_{L^2(M)}^2 
+ \int_M\left(f_x(u^\varepsilon(s)),\nabla u^\varepsilon(s)\right)_h 
\, dV_h(x)\right)\, ds \\
&\quad + \int_0^t \norm{B(u^\varepsilon(s))}_{\hS}^2 \,ds\\
&\quad +2\int_0^t
\left( u^\varepsilon(s), B(u^\varepsilon(s))\,dW(s)\right)_{L^2(M)}, 
\quad \text{$\P$-a.s., for any $t\in[0,T]$}.
\end{split}
\end{equation}
\end{theorem}

\begin{remark}\label{rmk: analytically weak solution}
The expression $\int_M\left(-\varepsilon\nabla 
u^\varepsilon(s) + f_x(u^\varepsilon(s)),\cdot\right)_h 
\, dV_h(x)$ above denotes the linear functional on $V$ induced 
by $-\varepsilon\nabla u^\varepsilon(s) + f_x(u^\varepsilon(s))$.

Fix any $\theta\in\Hone$. It is immediate to 
see that the following equation holds $\P$-a.s., 
for any $t\in[0,T]$:
\begin{equation}\label{eq: analytically weak solution}
\begin{split}
& \left( u^\varepsilon(t),\theta\right)_{L^2(M)}
\\ & \quad = \left(u^\varepsilon_0,\theta\right)_{L^2(M)} 
+ \int_0^t\left(\, \int_M\left(-\varepsilon\nabla u^\varepsilon(s) 
+ f_x(u^\varepsilon(s)),\nabla\theta\right)_h \, dV_h(x)\right) 
\, ds \\ & \quad \qquad \qquad\qquad\qquad
+ \left(\int_0^t B(u^\varepsilon(s))\,dW(s),\theta\right)_{L^2(M)}.
\end{split}
\end{equation}
Thus, $u^\varepsilon$ is a weak solution in the ordinary sense.
\end{remark}

\section{Generalized It\^{o} formula}\label{sec: generalized ito}

Our aim is to establish a generalized It\^{o} formula for weak 
solutions to a general class of SPDEs on Riemannian manifolds. 
This result will be used to derive the kinetic 
formulation as well as to establish a priori $\Lp{p}$-estimates. 
An analogous formula for the Euclidean case is 
presented in \cite{DHV2016} (see also \cite{Krylov13}). 

The stochastic equations take the general form
\begin{equation}\label{eq: wide class}
\begin{split}
& du = F(t)\, dt + \Div G(t)\,dt +\Delta_h I(t)\,dt 
+ H(t)\,dW(t), \\
& u(0,x) = u_0(x), 
\end{split}
\end{equation}
where $t\in (0,T)$, $x\in M$, and $(M,h)$
is an $n$-dimensional ($n\ge 1$) compact smooth Riemannian 
manifold, which is connected, oriented, and 
without boundary. The cylindrical 
Wiener process $W$ is defined in Section \ref{sec: Background}. 

We assume $F\in\LOT{2}{\Ltwo}$ is predictable. 
Let us write $\overrightarrow{\Ltwo}$ for the separable Hilbert space of 
measurable vector fields on $M$ for which $|\cdot|_h$ belongs to $\Ltwo$. 
The vector field $G\in\LOTV$ is assumed to be predictable. 
We assume $I\in\LOT{2}{\Hone}$ is $\Ltwo$-predictable. 
Moreover, we ask that $H\in\LOT{2}{\hS}$ 
is predictable, and denote $H_k:=He_k,\,k\in\N$. 

Given an $\mathcal{F}_0$-measurable initial function 
$u_0\in L^2(\Omega;\Ltwo)$, a process 
$u\in \LC{2}{\Ltwo}\cap \LOT{2}{\Hone}$ is a weak 
solution of \eqref{eq: wide class} if, for 
any $\theta\in C^2(M)$, the following equation 
holds $\P$-a.s., for any $t\in[0,T],$
\begin{align*}
\left(u(t),\theta\right)_{L^2(M)}& 
= \left(u_0,\theta\right)_{L^2(M)} 
+ \int_0^t\left(\, \int_MF(s)\,\theta(x) \,dV_h(x)\right)\, ds\\
&\hspace{0.6cm}- \int_0^t\left( \,\int_M\left(G(s),\nabla\theta(x)\right)_h 
\, dV_h(x)\right) \, ds \\
&\hspace{0.6cm}+\int_0^t\int_MI(s)\Delta_h\theta(x)dV_h(x)\,ds
+  \left(\,\int_0^t H(s)\,dW(s),\theta\right)_{L^2(M)}.
\end{align*}

\begin{proposition}[generalized It\^{o} formula]\label{prop: generalized Ito}
Let 
$$
u\in \LC{2}{\Ltwo}\cap \LOT{2}{\Hone}
$$ 
be a weak solution of \eqref{eq: wide class}. Fix 
$$
\psi\in C^1(M), \qquad 
S\in C^2(\R), \quad 
S''\in L^\infty(\R).
$$ 
Then, $\P$-a.s., for any $t\in[0,T],$ 
\begin{equation}\label{eq: generalized Ito}
\begin{split}
& \left( S(u(t))\right.,\left.\psi\right)_{L^2(M)}\\
& \quad = \left(S(u_0),\psi\right)_{L^2(M)} 
+ \int_0^t\left(\, \int_MS'(u(s))F(s)\,\psi(x) \,dV_h(x)\right)\, ds\\
& \quad \qquad 
- \int_0^t\left(\, \int_M\big(G(s),\nabla\left(S'(u(s))\psi(x) 
\right)\big)_h \, dV_h(x)\right)\, ds \\
& \quad \qquad \quad
-\int_0^t \int_M\big( \nabla I(s),\nabla\left(S'(u(s))\psi\right) \big)_h 
\, dV_h(x)\,ds\\
& \quad \qquad\quad \quad
+  \left(\,\int_0^t S'(u(s))H(s)\,dW(s),\psi\right)_{L^2(M)}\\
& \quad \qquad \quad\quad\quad
+\frac12\sum_{k\geq 1}\int_0^t\left(\,
\int_MS''(u(s))H_k^2(s)\psi(x) \,dV_h(x) \right) \, ds. 
\end{split}
\end{equation}
\end{proposition}

\begin{proof}
The proof uses ideas from \cite{DHV2016}. 
In our context further complications arise, 
since the underlying space is not 
Euclidean but a general Riemannian manifold. 
We use regularization by means of the heat kernel on $(M,h)$. 
To this end, let $(P_\tau)_{\tau\geq 0}$ be the 
heat semigroup on $\Ltwo$ and $p_\tau(x,y)$ its associated 
heat kernel (for its definition, construction 
and main properties, we refer to \cite{Grig,Strichartz}). 

Here, we recall that $(x,y,\tau)\mapsto p_\tau(x,y)$ 
is in $C^\infty(M\times M\times (0,\infty))$, it is 
symmetric in $x$ and $y$ for any $\tau>0$, and it is positive. 
For a function $w\in\Ltwo$,
\begin{equation}\label{eq: action of the heat kernel}
P_\tau w(x)=\int_M p_\tau(x,y)w(y)\,dV_h(y), 
\qquad x\in M,  \; \tau>0,
\end{equation}
and $P_\tau w\in C^\infty(M)$. Moreover,
\begin{equation}\label{eq: continuity of Ptau}
P_\tau w \stackrel{L^2(M)}{\longrightarrow} w,  
\quad \text{as $\tau\to 0^+$}, 
\end{equation}
and
\[
\norm{P_\tau w}_{L^2(M)}\leq \norm{w}_{L^2(M)},  \quad \tau> 0. 
\]
Finally, the following pointwise bounds hold \cite{Davies}:
\begin{equation}\label{eq: pointwise bound for ptau}
	p_\tau(x,y)\lesssim \tau^{-n/2}, 
	\quad  
	\abs{\nabla_yp_\tau(x,y)}_h\lesssim \tau^{-\frac{n+1}{2}},
	\quad
	\abs{\nabla^2_yp_\tau(x,y)}_h\lesssim \tau^{-\frac{n+4}{2}}, 
\end{equation}
for $x,y\in M$ and $\tau>0$, where $n$ is the dimension of $M$. 

In the course of the proof, we will also make use of the heat kernel on forms. 
We hence begin by making a brief digression (for further details, 
we see \cite{Elworthy,Bakry,deRham}). 
Denote by $(\mathcal{E}_\tau)_{\tau\geq 0}$ the de Rham-Hodge 
semigroup on 1-forms, associated to the de Rham-Hodge Laplacian, 
which by elliptic regularity has a kernel $e(\tau,\cdot,\cdot)$. 
More precisely, for any $\tau>0$, $e(\tau,
\cdot,\cdot)$ is a double form on $M\times M$, such 
that for any 1-form $\omega\in \Ltwo$ and any $p\in M$,
\[
(\mathcal{E}_\tau\omega)(p)
= \int_M e(\tau,p,q)\wedge\star_q\,\omega(q), 
\]
where $\star$ is the Hodge star operator, $q$ is a 
typical point in $M$ and $\wedge$ is the wedge product between forms. 
Concretely, in a coordinate patch $(U,(x^i))$ around $p$ 
and in a coordinate patch $(U',(y^j))$ around $q$, 
if we write the double form $e(\tau,\cdot,\cdot)$ as
\[
e(\tau,x,y)=\left(e(\tau,x,y)_{ij}\,dx^i\right)\, dy^j
\]
and $\omega$ as $\omega(y)=\omega_k(y)\, dy^k$, then 
the above integral becomes
\[
(\mathcal{E}_\tau\omega)(x)= \left(\int_M e(\tau,x,y)_{ij}\,h^{jk}(y)
\,\omega_k(y) \,dV_h(y)\right)\, dx^i. 
\]
For a vector field $V$, we denote by $V^\flat$ the 1-form 
obtained by lowering an index via the metric $h$; analogously, 
for a 1-form $\omega$, we denote by $\omega^\sharp$ the 
vector field obtained by raising an index always via the metric.

We define for a vector field $V$ the following quantity
\[
\mathcal{E}_\tau V := ((\mathcal{E}_\tau V^\flat))^\ast. 
\] 
We have the following remarkable properties: 
for $V\in \overrightarrow{\Ltwo}$,
\begin{itemize}

\item $\mathcal{E}_\tau V$ is a smooth vector field for any $\tau>0$,

\item $\mathcal{E}_\tau V\to V$ in $\overrightarrow{\Ltwo}$ as $\tau\to 0^+$,

\item $\norm{\mathcal{E}_\tau V}_{\overrightarrow{L^2(M)}}
\leq \norm{V}_{\overrightarrow{L^2(M)}}$, for any $\tau\geq 0$. 
\end{itemize}
Furthermore, in analogy with \eqref{eq: action of the heat kernel}, 
the following local expression holds
\[
(\mathcal{E}_\tau V)(x)= 
\left(\int_M e(\tau,x,y)_{ij} 
\,V^{j}(y) \,dV_h(y)\right)h^{ik}(x)\partial_k. 
\]
Our aim now is to compute $\Div \,\mathcal{E}_\tau V$, for fixed $\tau>0$. 
To this end, we observe that the kernel $e(\tau,x,y)$ is jointly 
smooth in $x$ and $y$, and that by assumption $V^j$ is integrable. 
Therefore, we are allowed to interchange differentiation 
and integration, to obtain
\begin{equation}\label{eq: div del campo vettoriale regolarizzato}
\begin{split}
\Div \,\mathcal{E}_\tau V(x) &= \int_M 
\partial_k e(\tau,x,y)_{ij}\,V^{j}(y)\,dV_h(y)\,h^{ik}(x)
\\ & \hspace{0.7cm}+\int_M e(\tau,x,y)_{ij}\,V^{j}(y)
\,dV_h(y)\,\partial_kh^{ik}(x)
\\ &\hspace{0.7cm}
+\Gamma^\rho_{\rho k}(x)
\int_M e(\tau,x,y)_{ij}\,V^{j}(y)\,dV_h(y)h^{ik}(x), 
\end{split}
\end{equation}
in local coordinates $x$ (differentiation is carried out in $x$).

We now fix $\tau>0$ and $x\in M$, and 
use $y\mapsto p_\tau(x,y)$ as a test function in \eqref{eq: wide class}, 
obtaining, $\P$-a.s., for any $t\in[0,T]$,
\begin{align*}
P_\tau u(t)(x)=& \, P_\tau u_0(x) 
\\ & + \int_0^tP_\tau F(s)(x)\,ds
-\int_0^t\left( \,\int_M\left(G(s),\nabla_yp_\tau(x,y)\right)_h 
\, dV_h(y)\right)\, ds \\
& +\int_0^t\int_MI(s)\Delta_yp_\tau(x,y)\,dV_h(y)\,ds
+  \left(\,\int_0^t H(s)\,dW(s),p_\tau(x,\cdot)\right)_{L^2(M)},
\end{align*}
where we have employed the stochastic Fubini theorem \cite[Thm 4.33]{DZ}. 
The last term may be written as
\[
\left(\int_0^t H(s)\,dW(s),p_\tau(x,\cdot)\right)_{L^2(M)}
= \sum_{k\geq 1}\int_0^t
P_\tau H_k(s)(x)\,d\beta_k(s).
\]
In view of the predictability and integrability assumptions 
made on $F,G,I, H$, we are allowed to apply 
the classical It\^{o} formula \cite[Thm.~4.32]{DZ} 
to $S(P_\tau u(t))$. We multiply 
the result by $\psi(x)$, obtaining, $\P$-a.s., 
for any $t\in[0,T]$, $x\in M$, and $\tau>0$, 
\begin{align*}
S(P_\tau u(t))\psi 
=   S &( P_\tau  u_0)  \psi  
+\int_0^t S'(P_\tau u(s))\psi P_\tau F(s)\,ds 
\\ & -\int_0^t S'(P_\tau u(s))\psi
\left(\, \int_M\left(G(s),\nabla_yp_\tau(\cdot,y)\right)_h 
\,dV_h(y)\right)\, ds 
\\ & +\int_0^tS'(P_\tau u(s))\psi
\left(\int_MI(s)\Delta_yp_\tau(\cdot,y)\,dV_h(y)\right) \, ds 
\\ & +\sum_{k\geq 1}\int_0^t S'(P_\tau u(s))\psi  
P_\tau H_k(s)\,d\beta_k(s) 
\\ & + \frac12 \sum_{k\geq 1}\int_0^t S''(P_\tau u(s))
\psi \left(P_\tau H_k(s)\right)^2\,ds.
\end{align*}
In view of
\begin{equation}\label{eq: growth S 1}
\abs{S(\xi)}\leq C\left(1+\xi^2\right),  
\qquad \abs{S'(\xi)}\leq C\left( 1+\abs{\xi} \right),  
\qquad \xi\in\R, 
\end{equation}
the pointwise estimates \eqref{eq: pointwise bound for ptau} 
on $p_\tau$ and the integrability assumptions on $u$, we can 
integrate the expression above w.r.t.~$dV_h(x)$, arriving at 
\begin{equation}\label{eq: proto-Ito}
\begin{split}
\int_M & S(P_\tau u(t))\psi(x)\, dV_h(x)
= \, \int_MS(P_\tau u_0)\psi(x)\,dV_h(x) \\
&+\int_0^t \int_MS'(P_\tau u(s))\psi(x)P_\tau F(s)\,dV_h(x)\,ds \\
&- \int_0^t \int_MS'(P_\tau u(s))\psi(x)
\left( \int_M\big(G(s),\nabla_yp_\tau(x,y)\big)_h 
\, dV_h(y)\right)\, dV_h(x)\,ds\\
&+\int_0^t\int_MS'(P_\tau u(s))\psi(x)
\left(\int_MI(s)\Delta_yp_\tau(x,y)\,dV_h(y)\right)\, dV_h(x)\,ds \\
&+\int_M\sum_{k\geq 1}\int_0^t 
S'(P_\tau u(s))\psi(x)  P_\tau H_k(s)\,d\beta_k(s)\,dV_h(x)\\
&+ \frac12\sum_{k\geq 1} \int_0^t \int_MS''(P_\tau u(s))\psi(x) 
\left(P_\tau H_k(s)\right)^2\,dV_h(x)\,ds, 
\end{split}
\end{equation}
$\P$-a.s., for any $t\in[0,T]$, and $\tau>0$. 
Moreover, we can interchange the integrals 
appearing in the formula above thanks to the 
Fubini and stochastic Fubini theorems.

The aim is to prove that, as $\tau\to 0^+$, each term 
in \eqref{eq: proto-Ito} converges $\P$-a.s.~to its corresponding 
term in \eqref{eq: generalized Ito}. Let us begin with the term 
on the left hand side: in view of \eqref{eq: growth S 1}, it 
follows that $S(u(t))\in\Lp{1}$, $\P$-a.s., for any $t\in[0,T]$.
Moreover, since $\abs{S(r_1)-S(r_2)}
\leq C\left( (1+\abs{r_2})\abs{r_1-r_2}+\abs{r_1-r_2}^2\right)$, 
for any $r_1,r_2\in\R,$ 
\begin{align*}
\int_M  & \abs{S(P_\tau u(t))-S(u(t))}\,\abs{\psi(x)}\,dV_h(x) \\
&\leq C\norm{\psi}_{L^\infty} \int_M 
\abs{P_\tau u(t)-u(t)}(1+|u(t)|)\,dV_h(x) 
\\ & \hspace{0.75cm}
+C\norm{\psi}_{L^\infty} \int_M \abs{P_\tau u(t)-u(t)}^2\,dV_h(x)\\
&\leq C \norm{P_\tau u(t)-u(t)}_{L^2(M)}
\norm{1+\abs{u(t)}}_{L^2(M)}
+ C \norm{P_\tau u(t)-u(t)}^2_{L^2(M)}, 
\end{align*}
and thus, by \eqref{eq: continuity of Ptau}, we obtain
\[
\int_M  S(P_\tau u(t))\psi(x)\,dV_h(x) 
\overset{\tau\downarrow 0}{\to} 
\int_M  S(u(t))\psi(x)\,dV_h(x),  
\quad \text{$\P$-a.s., $t\in[0,T]$}.
\]
Similarly, $\P$-a.s., $\int_M  S(P_\tau u_0)\psi(x)\,dV_h(x)
\to \int_M  S(u_0)\psi(x)\, dV_h(x)$.

In view of \eqref{eq: growth S 1}, $S'(u(t))\in\Ltwo$, $\P$-a.s., 
$t\in[0,T]$, and
\begin{align*}
\int_0^t\int_M  &
\abs{S'(P_\tau u(t))P_\tau F(s)-S'(u(t))F(s)}\,\abs{\psi(x)}
\,dV_h(x)\, ds \\
&\leq \norm{\psi}_{L^\infty} 
\int_0^t\int_M \abs{S'(P_\tau u(s))-S'(u(s))}\,\abs{P_\tau F(s)}
\,dV_h(x)\, ds \\
&\hspace{0.75cm}+\norm{\psi}_{L^\infty} \int_0^t
\int_M \abs{S'(u(s))}\,\abs{P_\tau F(s)-F(s)}\,dV_h(x)\, ds.
\end{align*}
Since $\abs{S'(r_1)-S'(r_2)}\leq 
\norm{S''}_{L^\infty}\abs{r_1-r_2}$, 
for any $r_1,r_2\in\R$, we infer
\begin{align*}
\int_0^t\int_M  & 
\abs{S'(P_\tau u(t))P_\tau F(s)-S'(u(t))F(s)}\,\abs{\psi(x)}
\,dV_h(x)\, ds\\
&\leq C \int_0^t\int_M\abs{P_\tau u(s) -u(s)}\,\abs{P_\tau F(s)}
\,dV_h(x)\, ds \\
&\hspace{0.75cm}+C \int_0^t\int_M (1+\abs{(u(s)})\,
\abs{P_\tau F(s)-F(s)}\,dV_h(x)\, ds\\
&\leq C \int_0^t\norm{P_\tau u(s)-u(s)}_{L^2(M)}
\,\norm{F(s)}_{L^2(M)}\,ds\\
&\hspace{0.75cm}
+C \int_0^t\norm{P_\tau F(s)-F(s)}_{L^2(M)}\,ds\\
&\hspace{0.75cm}
+C \int_0^t\norm{u(s)}_{L^2(M)}\,
\norm{P_\tau F(s)-F(s)}_{L^2(M)}\,ds,
\end{align*}
where we have used H\"older's inequality to 
obtain the second inequality. Since 
$\norm{P_\tau u(s) -u(s)}_{L^2(M)}\to 0$, $\P$-a.s., for each $s$, 
and $\norm{P_\tau F(s)-F(s)}_{L^2(M)}\to 0$ for 
$\P\otimes dt$-a.e.~$(\omega,s)$, as $\tau\to 0^+$, we 
can use the dominated convergence theorem to conclude that 
the terms above go to zero. Thus, $\P$-a.s., for any 
$t\in [0,T]$,
\[
\int_0^t \int_MS'(P_\tau u(s))\psi(x)P_\tau F(s)\,dV_h(x)\, ds
\overset{\tau\downarrow 0}{\to} 
\int_0^t \int_MS'(u(s))\psi(x)F(s)\,dV_h(x)\, ds. 
\]

Let us now deal with the last term in \eqref{eq: proto-Ito}. 
We have
\begin{align*}
\sum_{k\geq 1}&\int_0^t\int_M
\abs{S''(P_\tau u(s))(P_\tau H_k(s))^2-S''(u(s))H_k^2(s)}
\,\abs{\psi(x)} \,dV_h(x)\,ds\\
&\leq \norm{\psi}_{L^\infty} \sum_{k\geq 1}\int_0^t\int_M
\abs{S''(P_\tau u(s))}\, \abs{(P_\tau H_k(s))^2-H_k^2(s)}
\,dV_h(x)\,ds \\
&\qquad +\norm{\psi}_{L^\infty}\sum_{k\geq 1}
\int_0^t\int_M \abs{S''(P_\tau u(s))-S''(u(s))}
\,\abs{H_k^2(s)}\,dV_h(x)\,ds\\
&=: I_1+I_2. 
\end{align*}
By the general fact $\norm{a^2-b^2}_{L^1} 
\leq \norm{a-b}_{L^2}\norm{a+b}_{L^2}$,
\begin{align*}
I_1 &\leq \norm{\psi}_{L^\infty} \norm{S''}_{L^\infty}\sum_{k\geq 1}
\int_0^t \norm{(P_\tau H_k(s))^2-H_k^2(s)}_{L^1(M)}\,ds \\
& \leq C \sum_{k\geq 1}\int_0^t \norm{P_\tau H_k(s)-H_k(s)}_{L^2(M)}
\norm{P_\tau H_k(s)+H_k(s)}_{L^2(M)}\,ds. 
\end{align*}
We have $\norm{P_\tau H_k(s)-H_k(s)}_{L^2(M)}\to 0$ 
as $\tau\to 0^+$ for $\P\otimes dt$-a.e.~$(\omega,s)$.
Furthermore, $\norm{P_\tau H_k(s)\pm H_k(s)}_{L^2(M)}
\leq 2 \norm{H_k(s)}_{L^2(M)}$ for 
$\P\otimes dt$-a.e.~$(\omega,s)$ and 
(by assumption) $H\in\LOT{2}{\hS}$. Therefore, by the 
dominated convergence theorem, 
$I_1\to 0$ as $\tau\to 0^+$, $\P$-a.s., $t\in[0,T]$. 

For the term $I_2$ we proceed in this way: since 
$P_\tau u(s)\to u(s)$ in $\Ltwo$ 
as $\tau\to 0^+$, $\P$-a.s., for a.e.~$s$, we 
infer by Lemma \ref{lemma: strumentale} that
\[
S''(P_\tau u(s))\overset{\tau\downarrow 0}{\to} 
S''(u(s))\;\; 
\mbox{in }\,\, \Lp{1}, 
\quad \text{$\P$-a.s., $s\in [0,t]$.}
\] 
Let $(\tau_j)_j$ be a sequence tending to $0$ as $j\to\infty$. 
We extract a subsequence $(\tau_{j_i})_i$. By the reverse 
dominated convergence theorem, we can extract a 
further subsequence (not relabelled) such that 
$\abs{S''(P_{\tau_{j_i}} u(s)(x))- S''(u(s)(x))}\to 0$ 
as $i\to \infty$, $\P$-a.s., for a.e.~$x\in M$ and $s\in [0,t]$. 
Thus, $\sum_{k}\abs{H_k^2(s)(x)}\,
\abs{S''(P_{\tau_{j_i}} u(s)(x))- S''(u(s)(x))}$ converges to 
zero. Besides, $\sum_{k}\abs{H_k^2(s)(x)}\,
\abs{S''(P_{\tau_{j_i}} u(s)(x))- S''(u(s)(x))}$ 
can be bounded by $2\norm{S''}_{L^\infty}\sum_{k\geq 1}H_k^2(s)(x)$, 
$\P$-a.s., for a.e.~$x\in M$ and $s\ge 0$. 
In view of our assumptions, the latter term belongs to $\Lp{1}$, $\P$-a.s., 
for a.e.~$s\ge 0$. By the dominated convergence theorem, we infer 
$\sum_{k}\abs{H_k^2(s)}\, \abs{S''(P_{\tau_{j_i}}u(s))- S''(u(s))}\to 0$ 
in $\Lp{1}$ as $i\to\infty$, $\P$-a.s., for a.e.~$s\in [0,t]$. 
This is true for the original sequence $(\tau_j)_j$ as well. 
In other words, 
\[
\norm{\, \sum_{k\geq 1}\abs{H_k^2(s)} \,
\abs{S''(P_{\tau_{j_i}} u(s))- S''(u(s))}}_{L^1(M)}
\overset{\tau\downarrow 0}{\to} 0,
\quad \text{$\P$-a.s., $s\in [0,t]$}.
\]
This last quantity is bounded by $2\norm{S''}_{L^\infty} 
\norm{H(s)}^2_{\hS}$, which is $\P\otimes dt$ integrable by assumption. 
By dominated convergence, we eventually conclude that 
\[
I_2\overset{\tau\downarrow 0}{\to} 0, 
\quad \text{$\P$-a.s., $t\in[0,T]$}. 
\]

Next we consider the stochastic term. First of all, by 
the Burkholder-Davis-Gundy inequality \eqref{eq: BDG} 
and \eqref{eq: growth S 1}, 
\begin{align*}
\mathcal{I} & :=\EE\sup_{0\leq t\leq T}
\abs{\sum_{k\geq 1}\int_0^t\int_M S'(u(s))H_k(s)\psi(x)
\,dV_h(x)\, d\beta_k(s)}
\\ &\leq C \EE\left(\int_0^T\sum_{k\geq 1}
\abs{\int_M S'(u(s))H_k(s)\psi(x)\, dV_h(x)}^2 \, ds \right)^\frac12  
\\ &\leq C \EE\left(\int_0^T\norm{S'(u(s))}^2_{L^2(M)}\norm{H(s)}^2_{\hS}
\,ds \right)^\frac12  
\\ &\leq C \EE\left(\int_0^T\left(1+ \norm{u(s)}^2_{L^2(M)}\right)
\norm{H(s)}^2_{\hS}\,ds \right)^\frac12  
\\ & \leq C \EE\left(\int_0^T\norm{u(s)}^2_{L^2(M)}
\,\norm{H(s)}^2_{\hS}\,ds \right)^\frac12  \\
&\hspace{1cm}+ C \EE\left(\int_0^T\norm{H(s)}^2_{\hS}\,ds \right)^\frac12,
\end{align*}
where we have used the inequality $\sqrt{x+y}\leq\sqrt{x}+\sqrt{y}$, 
valid for $x,y\geq 0$. Therefore, by H\"older's inequality and the 
integrability assumptions on $u$ and $H$,
\begin{align*}
\mathcal{I} & \leq C \EE\left(\sup_{0\leq s\leq T}
\norm{u(s)}^2_{L^2(M)}\,\int_0^T\norm{H(s)}^2_{\hS}\,ds \right)^\frac12  \\
&\hspace{1cm}+ C \left(\EE\int_0^T\norm{H(s)}^2_{\hS}\,ds\right)^\frac12 \\
& \leq C 
\left(\EE\int_0^T\norm{H(s)}^2_{\hS}\,ds \right)^\frac12
\left[ 1+
\left(\EE\sup_{0\leq s\leq T}
\norm{u(s)}^2_{L^2(M)}\right)^\frac12\right]\, ,
\end{align*}
which is finite. Thus, the stochastic integral 
in \eqref{eq: generalized Ito} is well defined. 
Therefore, arguing as above,
\begin{align*}
\EE\sup_{0\leq t\leq T}& 
\abs{\sum_{k\geq 1} \int_0^t\int_M 
\left[S'(P_\tau u(s))P_\tau H_k(s)-S'(u(s))H_k(s)\right]\psi(x)\,dV_h(x)
\, d\beta_k(s) }\\
&\leq C \EE\left(\int_0^T\norm{S'(P_\tau u(s))-S'(u(s))}^2_{L^2(M)}
\norm{P_\tau H(s)}^2_{\hS}\,ds \right)^\frac12  \\
&\hspace{0.5cm}+ C \EE\left(\int_0^T
\norm{S'(u(s))}^2_{L^2(M)}
\norm{P_\tau H(s)-H(s)}^2_{\hS}\,ds \right)^\frac12\\
&=: J_1 + J_2. 
\end{align*}
We have
\[
J_1 \leq \,C \EE\left(\int_0^T
\norm{P_\tau u(s))-u(s)}^2_{L^2(M)}\norm{H(s)}^2_{\hS}
\,ds \right)^\frac12.   
\]
The term $\norm{P_\tau u(s))-u(s)}^2_{L^2(M)}\norm{H(s)}^2_{\hS}$ 
is dominated by the quantity $2\sup_{0\leq s\leq T}
\norm{u(s)}^2_{L^2(M)}\norm{H(s)}^2_{\hS}$ 
for $\P\otimes ds$ a.e.~$(\omega,s)$, which is integrable 
in time, $\P$-almost surely. Moreover, 
$\norm{P_\tau u(s))-u(s)}^2_{L^2(M)}\to 0$ 
as $\tau\to 0^+$, for $\P\otimes ds$ a.e.~$(\omega,s)$. 
Hence, by the dominated convergence theorem,
\[
\left(\int_0^T\norm{P_\tau u(s))-u(s)}^2_{L^2(M)}
\norm{H(s)}^2_{\hS}\, ds \right)^\frac12 
\overset{\tau\downarrow 0}{\to} 0,  
\quad \text{$\P$-almost surely}. 
\]
This last quantity is bounded uniformly in $\tau$ by 
\[
2\left(\sup_{0\leq s\leq T}\norm{u(s)}^2_{L^2(M)}\right)^{\frac12}
\left(\int_0^T\norm{H(s)}^2_{\hS}\, ds\right)^{\frac12},
\quad \text{$\P$-a.s.},
\]
which is $\P$-integrable by assumption. 
Applying again the dominated convergence theorem, the 
conclusion is that
\[
J_1\to 0,  \;\; \mbox{as } \tau\to 0^+. 
\]
Because $\norm{P_\tau H(s)-H(s)}_{\hS}^2\to 0$ 
for $\P\otimes ds$ a.e.~$(\omega,s)$, the same 
conclusion holds true for $J_2$ as well. 
Therefore, up to a subsequence, we obtain 
\begin{align*}
\sum_{k\geq 1}\int_0^t & 
\int_M S'(P_\tau u(s))P_\tau H_k(s)
\,\psi(x)\,dV_h(x)\, d\beta_k(s)
\\&
\overset{\tau\downarrow 0}{\to}
\left(\int_0^t S'(u(s))H(s)\,dW(s),\psi\right)_{L^2(M)},  
\quad \text{$\P$-almost surely}.
\end{align*}

We are going to examine the remaining terms:
\[
A(\tau):=- \int_0^t \int_MS'(P_\tau u(s)(x))\psi(x)
\left( \int_M\left(G(s),
\nabla_yp_\tau(x,y)\right)_h \, dV_h(y)\right) \, dV_h(x)\,ds
\]
and 
\[
B(\tau):=\int_0^t\int_MS'(P_\tau u(s))\psi(x)
\left(\int_MI(s)\Delta_yp_\tau(x,y) \, dV_h(y)\right)\, dV_h(x)\,ds. 
\]

We temporarily assume that
\begin{itemize}
\item  for $\P\otimes dt$-a.e.~$(\omega,t)\in\Omega_T$, $G(\omega,t)$ 
is a $C^1$ vector field (so that $\Div G$ exists 
for $\P\otimes dt$-a.e.~$(\omega,t)\in\Omega_T$),

\item $\Div G\in\LOT{2}{\Ltwo}$.

\item  for $\P\otimes dt$-a.e.~$(\omega,t)\in\Omega_T$, $I(\omega,t)$ 
is in $C^2(M)$ (so that $\Delta_h I$ exists for 
$\P\otimes dt$-a.e.~$(\omega,t)\in\Omega_T$), and

\item $\Delta_h I\in\LOT{2}{\Ltwo}$.
\end{itemize}
As a result we are now allowed to integrate by parts in 
the expressions above. 

For the first term we obtain, $\P$-a.s.,
\begin{align*}
A_\tau &=\int_0^t \int_MS'(P_\tau u(s)(x))\,\psi(x)
\left( \int_M\mbox{div}_{h,y} G(s,y)p_\tau(x,y)
\, dV_h(y)\right)\, dV_h(x)\,ds
\\ & 
=\int_0^t \int_MS'(P_\tau u(s)(x))\,\psi(x)\, 
P_\tau\left(\mbox{div}_{h,y} G(s)\right)(x)\, dV_h(x)\,ds.
\end{align*}
For this reason, arguing exactly as was done with the $F$-term, we 
conclude that this quantity converges as $\tau\to 0^+$ to
\begin{align*}
\int_0^t &\int_MS'(u(s)(x))\,\psi(x)\, 
\left(\mbox{div}_{h,y} G(s)\right)(x)\, dV_h(x)\, ds\\
& =\int_0^t \int_MS'(u(s)(x))\,\psi(x)\, \Div G(s)(x)\, dV_h(x)\, ds, 
\end{align*}
$\P$-a.s., $t\in[0,T]$.
Observe that, in view of the integrability 
assumptions made on $u$ and \eqref{eq: growth S 1}, 
the chain rule for Sobolev functions implies 
$S'(u(s))\psi\in\Hone$ 
for $\P\otimes dt$ a.e.~$(\omega,s)\in\Omega_T$, a fact that enables us 
to integrate by parts in the last expression. 
Summarizing, under the additional smoothness assumptions 
made above on $G$, we have proved that
\[
A_\tau\overset{\tau\downarrow 0}{\to}
-\int_0^t \int_M\left( G(s),\nabla\left(S'(u(s))
\psi\right) \right)_h \, dV_h(x)\, ds, 
\quad
\text{$\P$-a.s., $t\in[0,T]$}. 
\]
A completely analogous result holds for $B_\tau$:
\[
B_\tau
\overset{\tau\downarrow 0}{\to}
-\int_0^t \int_M\left( \nabla I(s),\nabla\left(S'(u(s))
\psi\right) \right)_h \, dV_h(x)\, ds, 
\quad
\text{$\P$-a.s., $t\in[0,T]$}. 
\]

Recapping the story so far, the generalized It\^{o} formula 
holds under additional smoothness assumptions on $G$ and $I$. 
To establish the formula in the general case, we must 
apply a regularization procedure, using the heat kernel for $I$ and 
the heat kernel on forms for $G$.

Assume $G\in\LOTV$, $I\in\LOT{2}{\Hone}$ are $\Ltwo$-predictable. 
Consider an arbitrary sequence $(\tau_l)_l\subset(0,1)$ with 
$\tau_l\to 0^+$ as $l\to \infty$. We start by regularizing $G$. 
Indeed, $\mathcal{E}_{\tau_l}G(s)$ is a smooth vector field 
for $\P\otimes dt$ a.e.~$(\omega,s)\in\Omega_T$. 
Obviously, $\mathcal{E}_{\tau_l}G\in\LOTV$ is predictable. 
Finally, $\Div \,\mathcal{E}_{\tau_l}G$ exists and is in $\Ltwo$ 
for $\P\otimes dt$ a.e.~$(\omega,s)\in\Omega_T$. 
To show that $\Div \,\mathcal{E}_{\tau_l}G$ is in $\LOT{2}{\Ltwo}$ 
we proceed like this: for any $\gamma\in C^\infty(M)$, 
consider the $\R$-valued stochastic process
\[
\Omega_T \ni (\omega,t) \mapsto\Lambda_\gamma
= (\Div \,\mathcal{E}_{\tau_l}G(t),\gamma)_{L^2(M)}
=-(\mathcal{E}_{\tau_l}G(t),\nabla\gamma)_{\overrightarrow{L^2(M)}}.  
\]
Since (by assumption) $G$ is predictable and 
$\mathcal{E}_{\tau_l}$ is a continuous from 
$\overrightarrow{\Ltwo}$ to itself, it follows 
that $\mathcal{E}_{\tau_l}G$ is predictable. 
Moreover, by continuity of the scalar product, we 
immediately infer that $\Lambda_\gamma$ is predictable 
for any $\gamma\in C^\infty(M)$ as well. 
Because $\Ltwo$ is separable and $\overline{C^\infty(M)}^{L^2(M)}=\Ltwo$, 
the image of $C^\infty(M)$ under the Riesz isomorphism 
of $\Ltwo$ is a norming set, and thus, in view 
of \cite[Cor.~4, Ch.~2]{DiestelUhl}, we conclude 
that $\Div \,\mathcal{E}_{\tau_l}G$ is predictable. In particular, it 
is $\mathcal{F}\otimes\mathcal{B}([0,T])/\mathcal{B}(L^2(M))$ 
measurable. Given a Banach space $X$, we recall 
that a subset $G\subset X^\ast$ is norming if 
$\norm{x}=\sup\limits_{x^\ast\in G}
\abs{\langle x^\ast,x\rangle}/\norm{x^\ast}$ for each $x\in X$. 

By means of \eqref{eq: div del campo vettoriale regolarizzato}, we 
may write locally, for $\P\otimes dt$-a.e.~$(\omega,s)\in\Omega_T$,
\begin{align*}
&\Div  \,\mathcal{E}_{\tau_l} G(s)(x) 
= \int_M \partial_k e(\tau_l,x,y)_{ij}\,G^{j}(s,y)\,dV_h(y)\,h^{ik}(x)
\\ & \qquad +\int_M  e(\tau_l,x,y)_{ij}\,G^{j}(s,y)\,dV_h(y)
\left(\partial_kh^{ik}(x) 
+ \Gamma^\rho_{\rho k}(x)h^{ik}(x)\right).
\end{align*}
Hence, for each fixed $l\in\N$, there is a 
constant $C(M,h,\tau_l)>0$ such that
\[
\abs{\Div \mathcal{E}_{\tau_l}G(s)(x)}\leq C(M,h,\tau_l)\,
\norm{G(s)}_{\overrightarrow{\Ltwo}}, 
\]
for $\P\otimes dt$-a.e.~$(\omega,s)\in\Omega_T$ 
and all $x$ in the coordinate patch. Therefore,
\[
\int_M\abs{\Div \mathcal{E}_{\tau_l}G(s)(x)}^2 \, dV_{h}(x)
\leq C(M,h)\,\norm{G(s)}_{\overrightarrow{\Ltwo}}^2,
\]
for $\P\otimes dt$ a.e.~$(\omega,s)\in\Omega_T$, and 
the required integrability is a consequence of $G\in\LOT{2}{\Ltwo}$. 

Concerning the $I$ term, we apply $P_{\tau_l}$ to it. 
Exploiting the general fact
\[
\norm{P_{\tau_l}w}_{\Hone}\leq \norm{w}_{\Hone}, 
\qquad w\in\Hone, 
\]
and arguing as before, we can show that
$P_{\tau_l}I\in\LOT{2}{\Hone}$ and 
$P_{\tau_l}I$, $\Delta_h P_{\tau_l}I$ 
are $L^2(M)$-predictable. 
To verify the required integrability, we need access to an 
explicit expression for $\Delta_h P_{\tau_l}I$. 
According to \cite[Theorem 7.20]{Grig}, locally it holds
\begin{align*}
\Delta_h P_{\tau_l}I(s)(x)& 
= h^{ab}(x)(\partial_{ab}P_{\tau_l}I(s)(x)
-\Gamma_{ab}^c(x)\partial_cP_{\tau_l}I(s)(x))\\
& = h^{ab}(x)\int_M \partial_{ab}p_{\tau_l}(x,y)I(s,y)\,dV_h(y)\\
&\hspace{.75cm}-h^{ab}(x)\Gamma_{ab}^c(x)
\int_M\partial_cp_{\tau_l}(x,y)I(s,y)\,dV_h(y)\\
& =h^{ab}(x)\int_M \left(\partial_{ab}p_{\tau_l}(x,y) 
-\Gamma_{ab}^c(x)\partial_cp_{\tau_l}(x,y)\right)I(s,y)\,dV_h(y)\\
& =\int_M \Delta_{h,x}p_{\tau_l}(x,y) I(s,y)\,dV_h(y). 
\end{align*}
Hence, in view of the integrability assumptions made on $I$ 
and repeating earlier arguments, we infer 
$\Delta_h P_{\tau_l}I\in\LOT{2}{\Ltwo}$.

We are thus allowed to apply It\^{o}'s formula \eqref{eq: generalized Ito} 
with $\mathcal{E}_{\tau_l}G$ and $P_{\tau_l}I$, which will hold 
for any $t\in[0,T]$ and all $\omega\in\Omega_l$ for some 
$\Omega_l$ with $\P(\Omega_l)=1$. Set
\[
\Omega_0:=\bigcap_{l\in\N}\Omega_l, \qquad 
\P(\Omega_0)=1, 
\]
and fix $\omega\in\Omega_0$ and $t\in[0,T]$. 
To conclude the proof, we are required to check that 
\[
\int_0^t \int_M\big( \mathcal{E}_{\tau_l}G(s),
\nabla\left(S'(u(s))\psi\right) \big)_h \, dV_h\,ds
\overset{\tau\downarrow 0}{\to}
\int_0^t \int_M\big( G(s),
\nabla\left(S'(u(s))\psi\right) \big)_h \,dV_h\,ds
\]
and 
\[
\int_0^t \int_M\big( \nabla P_{\tau_l}I(s),
\nabla\left(S'(u(s))\psi\right) \big)_h \, dV_h\,ds
\overset{\tau\downarrow 0}{\to}
\int_0^t \int_M\big( \nabla I(s),
\nabla\left(S'(u(s))\psi\right) \big)_h \,dV_h\,ds.
\]
Regarding the first convergence claim, 
\begin{align*}
\int_0^t & \int_M \abs{\left( \mathcal{E}_{\tau_l}G(s)-G(s),
\nabla\left(S'(u(s))\psi\right) \right)_h} 
\,dV_h\,ds\\
&\leq\int_0^t \int_M \abs{\left( \mathcal{E}_{\tau_l}G(s)-G(s), 
\nabla u(s) \right)_hS''(u(s))\psi}\, dV_h\,ds\\
&\hspace{0.75cm}
+\int_0^t \int_M \abs{\left( \mathcal{E}_{\tau_l}G(s)-G(s), 
\nabla \psi \right)_hS'(u(s))} \,dV_h\,ds\\
&\leq\norm{\psi}_{L^\infty} \norm{S''}_{L^\infty}\int_0^t 
\norm{\mathcal{E}_{\tau_l}G(s)-G(s)}_{\overrightarrow{L^2(M)}}
\norm{u(s)}_{\Hone}\,ds
\\ &\hspace{0.75cm}
+C(M,h)\norm{\nabla \psi}_{L^\infty}\int_0^t 
\norm{\mathcal{E}_{\tau_l}G(s)-G(s)}_{\overrightarrow{L^2(M)}}
\norm{S'(u(s))}_{L^2(M)}\,ds
\\ &\leq C\int_0^t 
\norm{\mathcal{E}_{\tau_l}G(s)-G(s)}_{\overrightarrow{L^2(M)}}
\left(1+\norm{u(s)}_{\Hone}\right)\, ds. 
\end{align*}
The integrand in the last term converges to zero as $l\to\infty$, 
for $\P\otimes dt$-a.e.~$(\omega,s)$, and it is dominated 
by $C(M,h)\norm{G(s)}_{\overrightarrow{L^2(M)}}
\left(1+\norm{u(s)}_{\Hone}\right)$, 
which is integrable on $\Omega_T$. 
Thus, the integral converges to zero by the 
dominated convergence theorem. 

In a completely analogous way, using the fact that 
$P_{\tau_l}w\to w$ in $\Hone$, for $w\in\Hone$, we see that 
also the second convergence claim holds. 
This concludes the proof of Proposition \ref{prop: generalized Ito}.
\end{proof}

\section{Existence result}\label{Sec: Existence}

In this section we will prove the existence part of 
Theorem \ref{thm:well-posed}. We will achieve this by arguing 
that (weak) solutions to the parabolic problem \eqref{eq:target2} 
converge as $\varepsilon\to 0$ to a generalized kinetic 
solution of \eqref{eq:target}. In view of the already established 
rigidity/uniqueness results (cf.~Section \ref{sec: reduction and uniqueness}), 
the existence claim will follow from this.

Complying with Section \ref{sec: Vanishing}, we will without loss 
generality assume that the flux $f$ satisfies 
the conditions \eqref{eq: growth condition 1} and \eqref{eq: growth condition 2}. 
These conditions can be replaced by \eqref{eq: growth condition 0} 
at the expense of carrying out an additional approximation 
(and convergence) argument. We leave the details to the interested reader.

\subsection{Kinetic formulation of parabolic SPDE}

We start by writing the kinetic equation (in weak 
form) linked to the parabolic problem \eqref{eq:target2}.

\begin{proposition}[kinetic formulation]
\label{prop: kinetic formulation of the parabolic problem}
Fix $\varepsilon>0$ and $u_0\in L^\infty(M,h)$. 
Let $u^\varepsilon$ be the unique weak solution of \eqref{eq:target2}, 
with initial data $u^\varepsilon|_{t=0}=u_0$, given by 
Theorem \ref{thm: existence variational solution}. 
Then $\rho^\varepsilon:=\En_{u^\varepsilon>\xi}$ satisfies 
\begin{equation*}
\begin{split}
\int_0^T&\int_M\int_\R\rho^\varepsilon\,\partial_t\psi
\, d\xi \,dV_h(x)\, dt
+\int_M\int_\R\rho_0\,\psi(0,x,\xi)
\, d\xi \,dV_h(x) \\ & 
+\int_0^T\int_M\int_\R\rho^\varepsilon
\left(f'_x(\xi),\nabla \psi\right)_h\, d\xi \,dV_h(x)\, dt\\
&+\varepsilon\int_0^T\int_M\int_\R\rho^\varepsilon\Delta_h\psi
\,d\xi\,dV_h(x)\,dt =m^\varepsilon(\partial_\xi\psi)
\\ & -\sum_{k\geq 1}\int_0^T\int_M\int_\R g_k(x,\xi) \,\psi\,
\nu^\varepsilon_{\omega,t,x}(d\xi)\,dV_h(x)\,d\beta_k(t)
\\ &-\frac12\int_0^T\int_M\int_\R\partial_\xi\psi \,G^2(x,\xi)
\,\nu^\varepsilon_{\omega,t,x}(d\xi)\,dV_h(x)\, dt, 
\qquad \text{$\P$-almost surely},
\end{split}
\end{equation*}
$\forall \psi\in C^\infty_c([0,T)\times M\times\R)$, 
where $\rho_0:=\En_{u_0>\xi}$ and, for 
$\Psi\in C^0_b([0,T]\times M\times\R)$,
\[
\nu^\varepsilon_{\omega,t,x}=\delta_{u^\varepsilon(\omega,t,x)},  
\;
m^\varepsilon(\Psi)=\int_0^T\int_M\int_\R 
\varepsilon\, \abs{\nabla u^\varepsilon(t)}^2_h
\,\Psi(t,x,\xi)\, \delta_{u^\varepsilon(t)}(d\xi)\,dV_h(x)\,dt.
\] 
\end{proposition}

\begin{proof}
As the proof is fairly standard, we keep the level of 
detail at a minimum. We use the generalized It\^{o} formula 
(cf.~Proposition \ref{prop: generalized Ito}) with 
$F\equiv 0$, $G(s)=-f_x(u^\varepsilon(s))$, 
$I(s)=\varepsilon u^\varepsilon(s)$, $H(s)=B(u^\varepsilon(s))$, 
$\psi\in \mathcal{D}(M)$, and the nonlinear 
function $S(\xi)=\int_{-\infty}^\xi\lambda(r)\,dr$, 
$\lambda\in \mathcal{D}(\R)$. Taking into account \eqref{eq: gc}, 
the result is the following equation 
which holds $\P$-a.s., for any $t\in[0,T]$:
\begin{align*}
& \int_M \big\langle \rho^\varepsilon(t),
\lambda \big\rangle_{\mathcal{D}'(\R)}\psi(x) \,dV_h(x)
-\int_M \big\langle \rho_0,\lambda \big\rangle_{\mathcal{D}'(\R)} 
\psi(x) \,dV_h(x)
\\ & \quad =
\int_0^t \int_M \Big\langle \partial_\xi
(\varepsilon\,\abs{\nabla u^\varepsilon(s)}_h^2
\,\delta_{u^\varepsilon(s)}),\lambda \Big \rangle_{\mathcal{D}'(\R)}
\psi(x) \,dV_h(x) \,ds
\\ & \quad\qquad
-\varepsilon \int_0^t\int_M
\left( \nabla\big \langle \rho^\varepsilon(s),
\lambda\big\rangle_{\mathcal{D}'(\R)}, \nabla\psi\right)_h\,dV_h(x)\,ds
\\ & \quad \qquad\quad
-\int_0^t\int_M \Div \big\langle\rho^\varepsilon(s),f'_x(\cdot)\,
\lambda\big\rangle_{\mathcal{D}'(\R)}\psi(x)\,dV_h(x)\,ds
\\ & \quad \qquad\quad\quad
+\sum_{k\geq 1}\int_0^t\int_M \big\langle 
\delta_{u^\varepsilon(s)},\lambda\,g_k(x,\cdot)
\big \rangle_{\mathcal{D}'(\R)} \psi(x)\,dV_h(x)\,d\beta_k(s)
\\ & \quad \qquad\quad\quad\quad
-\frac12\int_0^t\int_M\big \langle\partial_\xi
\left(\delta_{u^\varepsilon(s)}G^2(x,\cdot)\right),
\lambda\big\rangle_{\mathcal{D}'(\R)} \,\psi(x)\,dV_h(x)\,ds. 
\end{align*}

As the topological tensor product 
$\mathcal{D}(M)\otimes\mathcal{D}(\R)$ is 
a dense subspace of $\mathcal{D}(M\times\R)$ (cf.~\cite[p.~38]{deRham}), 
the formula remains valid for any $\psi\in\mathcal{D}(M\times\R)$.
Indeed, after some simple manipulations, the following 
equation holds $\P$-a.s., for any $t\in[0,T]$, 
and for any $\psi\in\mathcal{D}(M\times\R)$: 
\begin{align*}
& \int_{M\times\R}\rho^\varepsilon(t)\psi(x,\xi)\,d\xi\,dV_h(x)
-\int_{M\times\R} \rho_0\psi(x,\xi) \,d\xi\,dV_h(x)
\\ & \qquad 
= -\int_0^t\int_M \varepsilon\,\abs{\nabla u^\varepsilon(s)}_h^2
\,(\partial_\xi\psi)(x,u^\varepsilon(s)) \,dV_h(x)\,ds
\\ & \qquad\quad +\varepsilon\int_0^t
\int_{M\times\R} \rho^\varepsilon(s)\Delta_h\psi
\,d\xi\,dV_h(x)\,ds\\ & \qquad\quad\quad 
+\int_0^t\int_{M\times\R} \rho^\varepsilon(s) 
\left(f'_x(\xi)\,\nabla\psi(x,\xi)\right)_h\,d\xi\,dV_h(x)\,ds
\\ & \qquad\quad\quad\quad
+\sum_{k\geq 1}\int_0^t\int_M g_k(x,u^\varepsilon(s))\,
\psi(x,u^\varepsilon(s))\,dV_h(x)\,d\beta_k(s)
\\ & \qquad\quad\quad\quad\quad
+\frac12\int_0^t\int_M(\partial_\xi\psi)(x,u^\varepsilon(s))
\,G^2(x,u^\varepsilon(s))\,dV_h(x)\,ds. 
\end{align*}
We multiply this equation by $\partial_t\theta$ and integrate over $[0,T]$, 
with $\theta\in \mathcal{D}(-T,T)$. The stochastic 
term is handled using ``integration by parts". 
The final result is
\begin{align*}
& \int_0^T\int_{M\times\R} \rho^\varepsilon(t)\psi(x,\xi)
\partial_t\theta(t) \,d\xi\,dV_h(x)\,dt
+\int_{M\times\R} \rho_0\psi(x,\xi)\theta(0) \,d\xi\,dV_h(x)
\\ & \quad 
= \int_0^T\int_M \varepsilon\,\abs{\nabla u^\varepsilon(t)}_h^2
\,(\partial_\xi\psi)(x,u^\varepsilon(t))\theta(t) \,dV_h(x)\,dt
\\ & \quad \quad
-\varepsilon\int_0^T\int_{M\times\R}\rho^\varepsilon(t)\Delta_h\psi
\,\theta(t)\,d\xi\,dV_h(x)\,dt
\\ & \quad \quad \quad 
-\int_0^T\int_{M\times\R} \rho^\varepsilon(t)(f'_x(\xi)
\,\nabla\psi(x,\xi))_h\theta(t)\,d\xi\,dV_h(x)\,dt
\\ & \quad\quad\quad\quad
-\sum_{k\geq 1}\int_0^T\int_M g_k(x,u^\varepsilon(t))
\,\psi(x,u^\varepsilon(t))\,\theta(t)\,dV_h(x)\,d\beta_k(t)
\\ & \quad\quad\quad\quad\quad
-\frac12\int_0^T\int_M(\partial_\xi\psi)(x,u^\varepsilon(t))
\,G^2(x,u^\varepsilon(t))\,\theta(t)\,dV_h(x)\,dt. 
\end{align*}
By density of tensor products, this equation continue to 
hold for any test function $\psi\in\mathcal{D}((-T,T)\times M\times\R)$. 
Finally, if we are given a function $\psi\in\mathcal{D}([0,T)\times M\times\R)$, we
extend it to a function belonging to $\mathcal{D}((-T,T)\times M\times\R)$ using 
the Whitney extension theorem. The equation holds for this extension too.
\end{proof}

\subsection{A priori $L^p$ estimates}
Our goal is to use the generalized It\^{o} formula in 
Proposition \ref{prop: generalized Ito} to establish an $\varepsilon$-independent 
$L^p$ bound on the weak solution to the parabolic SPDE \eqref{eq:target2}. 
Before we can do that we need two auxiliary lemmas, which make use of 
the geometric compatibility condition \eqref{eq: gc}. 

\begin{lemma}\label{lemma: gc1}
Suppose $f=f_x(\xi)$ is a smooth geometry-compatible vector 
field on $M$, smoothly depending on $\xi\in\R$. For $u\in C^1(M)$ 
and $S\in C^2(\R)$ with $S''\in L^\infty(\R)$,
\[
\int_M \left( f_x(u),\nabla S'(u) \right)_h \, dV_h(x)
=\int_M\left( f_x(u),\nabla u \right)_h S''(u)\,dV_h(x)=0. 
\]
\end{lemma}

\begin{proof}
Following \cite{Ben-Artzi/LeFloch}, we define 
\[
F_x(\xi):=\int_0^\xi S'(w)\partial_w f_x(w)\,dw\in T_xM,  
\qquad x\in M,  \,\xi\in\R. 
\]
The geometry compatibility condition \eqref{eq: gc} implies the 
following pointwise identity: 
\[
S'(u(x))\, \Div f_x(u(x))
=\Div F_x(u(x)),  \qquad x\in M. 
\]
Therefore, the divergence theorem implies 
\begin{align*}
&\int_M \left( f_x(u),\nabla u \right)_h S''(u)\,dV_h(x)
=\int_M\left( f_x(u),\nabla S'(u) \right)_h \, dV_h(x)
\\ & \quad = -\int_M S'(u(x))\, \Div f_x(u(x)) \,dV_h(x) 
=-\int_M \Div F_x(u(x)) \,dV_h(x)=0.
\end{align*}
\end{proof}

\begin{lemma}\label{lemma: gc2}
Suppose $f=f_x(\xi)$ is a smooth geometry-compatible vector 
field on $M$, smoothly depending on $\xi\in\R$, such that 
\eqref{eq: growth condition 1} and \eqref{eq: growth condition 2} hold. 
The conclusion of Lemma \ref{lemma: gc1} remains valid 
for any $u\in\Hone$.
\end{lemma}

\begin{proof}
By definition of $\Hone$, we can find a sequence 
$(u_m)_m\subset C^1(M)$ such that $u_m\rightarrow u$ in 
$\Hone$ as $m\to\infty$. By Lemma \ref{lemma: gc1},
\[
\int_M\left(f_x(u_m), \nabla u_m \right)_h S''(u_m)\,dV_h(x)=0,  
\qquad m\in\N. 
\] 
Let us write this equation as
\begin{align*}
0 & = 
\int_M\left( f_x(u_m),\nabla u_m \right)_h S''(u_m)\,dV_h(x)
\\ & = 
\int_M\left( f_x(u_m)-f_x(u),\nabla u_m \right)_h S''(u_m)\,dV_h(x)
\\ & \qquad 
+ \int_M\left( f_x(u),\nabla u_m- \nabla u \right)_h S''(u_m)\,dV_h(x)
\\ & \qquad \quad
+ \int_M\left( f_x(u),\nabla u \right)_h (S''(u_m)-S''(u))\,dV_h(x)
\\ & \qquad \quad
+ \int_M\left( f_x(u),\nabla u \right)_h S''(u)\,dV_h(x)\\
& =: A_1+A_2+A_3 + \int_M\left( f_x(u),\nabla u \right)_h S''(u)\,dV_h(x). 
\end{align*}
The lemma follows if we can show that 
$A_1,A_2,A_3\to 0$ as $m\to \infty$. 

Clearly,
\[
\abs{A_1}\leq C_1 \,\norm{S''}_{L^\infty}
\norm{u-u_m}_{L^2(M)}\sup_m \norm{\nabla u_m}_{L^2(M)},
\]
where $C_1$ is the constant in \eqref{eq: growth condition 2}. 
Therefore, $A_1\to 0$ as $m\to\infty$.

In a similar way, employing \eqref{eq: growth condition 1},
\begin{align*}
\abs{A_2} \leq C_0\,\norm{S''}_{L^\infty} 
& \Big ( (1+L^r) \norm{\nabla u_m-\nabla u}_{L^1(M)}
\\ & \qquad + \norm{u}_{L^2(M)}
\norm{\nabla u-\nabla u_m}_{L^2(M)} \Big), 
\end{align*}
and hence $A_2\to 0$ as $m\to \infty$.

For the remaining term, we argue like this:
\begin{align*}
\abs{A_3} & \leq C_0\int_M (1+L^r)\abs{\nabla u}_h
\abs{S''(u_m)-S''(u)}\,dV_h(x)
\\ &\qquad + C_0\int_M \abs{u}
\,\abs{\nabla u}_h \abs{S''(u_m)-S''(u)}\,dV_h(x) 
\\ & \leq C_0(1+L^r)\,\norm{\nabla u}_{L^2(M)}
\norm{S''(u_m)-S''(u)}_{L^2(M)}
\\ &\qquad 
+C_0\int_M \abs{u}\,\abs{\nabla u}_h
\abs{S''(u_m)-S''(u)}\,dV_h(x). 
\end{align*}
By invoking Lemma \ref{lemma: strumentale} (in the appendix), 
$S''(u_m)\to S''(u)$ in $\Lp{p}$ as $m\to\infty$, 
for any $p\in[1,\infty)$. Therefore, 
\[
\abs{A_3}\leq o(1)+
C_0 \int_M \abs{u}\,\abs{\nabla u}_h 
\abs{S''(u_m)-S''(u)}\,dV_h(x),  
\]
where $o(1)\to 0$ as $m\to\infty$. 
Let $p_n$ and $q_n$ be defined by
\[
p_n=\left\{
\begin{array}{ll}
2, & \text{if $n=1$} \\
\geq 1, & \text{if $n=2$} \\
n, & \text{if $n>2$}
\end{array}
\right.
\]
and $\frac{1}{q_n}+\frac12+\frac{1}{p_n}=1$, 
where $n$ is the dimension of the manifold. 
Therefore, by the generalized H\"older 
inequality and the Sobolev embedding,
\[
\abs{A_3}\leq o(1) + C_0\,
\norm{u}_{L^{q_n}(M)}\norm{\nabla u}_{L^2(M)}
\norm{S''(u_m)-S''(u)}_{L^{p_n}(M)}, 
\]
and thus $A_3\to 0$ as $m\to \infty$, again thanks 
to Lemma \ref{lemma: strumentale}.
\end{proof}

Fix $\varepsilon>0$, and consider 
$\mathcal{F}_0$-measurable initial datum 
$u_0^\varepsilon\in L^2(\Omega;\Ltwo)$. Denote by $u^\varepsilon$ 
the corresponding weak (variational) solution given by 
Theorem \ref{thm: existence variational solution}.  
In what follows, we derive some basic energy estimates that will allow us 
to show (weak) convergence of $u^\varepsilon$ as $\varepsilon\to 0$.

\begin{proposition}\label{prop: bound for the L2 norm}
There exists an $\varepsilon$-independent 
constant $K>0$ such that 
\begin{equation}\label{eq: bound for the L2 norm}
\EE\sup_{0\leq t\leq T} \norm{u^\varepsilon(t)}_{L^2(M)}^2
+\varepsilon \, \EE\int_0^T\norm{\nabla u^\varepsilon(t)}_{L^2(M)}^2\,dt 
\leq K \left(1+\EE\,\norm{u^\varepsilon_0}_{L^2(M)}^2\right). 
\end{equation}
\end{proposition}

\begin{proof}
From equation \eqref{eq: Ito p=2}, it follows that
\begin{align*}
&\EE\,\norm{u^\varepsilon(t)}_{L^2(M)}^2
= \EE\,\norm{u^\varepsilon_0}_{L^2(M)}^2 
\\ & \quad 
+2\int_0^t\EE\left(-\varepsilon \norm{\nabla u^\varepsilon(s)}_{L^2(M)}^2 
+\int_M\left( f_x(u^\varepsilon(s)),\nabla u^\varepsilon(s)\right)_h
\, dV_h(x)\right)\, ds 
\\ & \quad 
+\int_0^t \EE\,\norm{B(u^\varepsilon(s))}_{\hS}^2\, ds, 
\qquad t\in [0,T]
\end{align*}
For additional details, see \cite[Remark 4.2.8]{PR}. 
Because $u^\varepsilon$ belongs to $\Hone$ for 
$\P\otimes dt$-a.e.~$(\omega,s)\in\Omega_T$, it follows 
from Lemma \ref{lemma: gc2} with $S(\xi)=\xi^2$ that 
$$
\int_0^t\EE \int_M 
\left(f_x(u^\varepsilon(s)),\nabla u^\varepsilon(s)\right)_h 
\,dV_h(x)\,ds=0. 
$$
Using this and \eqref{eq: definition of G}, we obtain
\begin{align*}
\EE\,\norm{u^\varepsilon(t)}_{L^2(M)}^2 
& \leq \EE\,\norm{u^\varepsilon_0}_{L^2(M)}^2 
+ \int_0^t \EE\,\norm{B(u^\varepsilon(s))}_{\hS}^2 \,ds
\\ & \leq \EE\,\norm{u^\varepsilon_0}_{L^2(M)}^2 
+ D_1\int_0^t\EE \left(1+\norm{u^\varepsilon(s)}_{L^2(M)}^2\right)\, ds
\\ & =\EE\,\norm{u^\varepsilon_0}_{L^2(M)}^2 + D_1t
+\int_0^t D_1\,\EE\, \norm{u^\varepsilon(s)}_{L^2(M)}^2\, ds. 
\end{align*}
Note that $t\mapsto \EE\, \norm{u^\varepsilon(t)}_{L^2(M)}^2$ 
is continuous, and that 
$t\mapsto \EE\,\norm{u^\varepsilon_0}_{L^2(M)}^2 + D_1t$ 
is non-decreasing and positive. Hence, an application of 
Gronwall's inequality yields
\begin{equation}\label{eq: Gronwall-Bellman}
\EE\,\norm{u^\varepsilon(t)}_{L^2(M)}^2 
\leq e^{D_1t}\left(\EE\,\norm{u^\varepsilon_0}_{L^2(M)}^2 
+ D_1t\right), \qquad t\in[0,T].
\end{equation}

Let us now return to \eqref{eq: Ito p=2}. Using once 
more \eqref{eq: definition of G} and Lemma \ref{lemma: gc2}, 
taking the supremum and then the expectation, 
and making use of the elementary inequality 
$\sup_t\alpha(t)+\sup_t \beta(t)\leq 2\sup_t\gamma(t)$ 
valid for positive functions $\alpha(t),\beta(t)$ 
satisfying $\alpha(t)+\beta(t)\leq \gamma(t)$, we obtain
\begin{align*}
&\EE \sup_{0\leq t\leq T} \norm{u^\varepsilon(t)}_{L^2(M)}^2
+2\varepsilon\, \EE\int_0^T
\norm{\nabla u^\varepsilon(s)}_{L^2(M)}^2\,ds
\\ & \qquad 
\leq 2\EE\,\norm{u^\varepsilon_0}_{L^2(M)}^2 
+2D_1T+ 2D_1\int_0^T\EE\,\norm{u^\varepsilon(s)}_{L^2(M)}^2 \,ds
\\ &\qquad \qquad
+4\EE\sup_{0\leq t\leq T} \left|\, \sum_{k\geq 1}
\int_0^t \int_M u^\varepsilon(s) g_k(x,u^\varepsilon(s))
\,dV_h(x)\,d\beta_k(s)\right|. 
\end{align*}
By the Burkholder-Davis-Gundy inequality \eqref{eq: BDG}, arguing 
as in the proof of Proposition \ref{prop: generalized Ito} 
and employing Cauchy's inequality, we can bound the last term by
\[
\frac12 \EE\sup_{0\leq t\leq T}\norm{u^\varepsilon(t)}_{L^2(M)}^2 
+C\left[1+\int_0^T\EE\,\norm{u^\varepsilon(s)}_{L^2(M)}^2 ds\right], 
\]
where $C$ depends on $D_1,T$ but not $\varepsilon,C_0,r,L,C_1,D_2$. 
Thus, for a new constant $C$, 
\begin{align*}
& \frac12 \EE\sup_{0\leq t\leq T}
\norm{u^\varepsilon(t)}_{L^2(M)}^2 + 2\varepsilon\,\EE\int_0^T
\norm{\nabla u^\varepsilon(s)}_{L^2(M)}^2\, ds
\\ & \qquad 
\leq 2\EE\,\norm{u^\varepsilon_0}_{L^2(M)}^2 
+ C\left( 1 + \int_0^T\EE\,\norm{u^\varepsilon(s)}_{L^2(M)}^2 \,ds\right). 
\end{align*}
Combining this with \eqref{eq: Gronwall-Bellman}, we arrive at 
the claim \eqref{eq: bound for the L2 norm}.
\end{proof}

With similar reasoning, this time using our 
generalized It\^{o} formula \eqref{eq: generalized Ito}, it 
is possible to derive a related bound for the $L^p$-norm.

\begin{proposition}\label{prop: bound for the Lp norm}
Fix $p\in (2,\infty)$. There exists an $\varepsilon$-independent 
constant $K_p>0$ such that 
\begin{equation}\label{eq: bound for the Lp norm}
\begin{split}
& \EE\sup_{0\leq t\leq T}
\norm{u^\varepsilon(t)}_{L^p(M)}^p
+\varepsilon\,\EE\int_0^T\int_M 
\abs{u^\varepsilon(t)}^{p-2}
\abs{\nabla u^\varepsilon(t)}^2_h\,dV_h(x)\,dt \\ & \qquad 
\leq K_p\left(1+\EE\,\norm{u^\varepsilon_0}_{L^p(M)}^p\right). 
\end{split}
\end{equation}
\end{proposition}

\begin{remark}
The bound \eqref{eq: bound for the Lp norm} implies
\[
\EE\left(\esssup_{t\in[0,T]}
\int_M\int_\R \abs{\xi}^p \, \nu^\varepsilon_{\omega,t,x}(d\xi)\, 
dV_h(x)\right)\leq C_p
\]
for some $\varepsilon$-independent constant $C_p>0$, 
where the ``Dirac delta" Young measure $\nu^\varepsilon_{\omega,t,x}$ is defined 
in Proposition \ref{prop: kinetic formulation of the parabolic problem}.
\end{remark}

\begin{proof}
To show the assertion we would like to employ 
the generalized It\^{o} formula with $\psi\equiv 1$ 
and nonlinear function $S(\xi)=\abs{\xi}^p$.
Unfortunately this cannot be done directly as the 
second derivative of $S$ is not bounded. 
We are therefore forced to utilize a suitable 
approximation of $S$, following 
an idea from \cite{DMS2005}. 

For $m\in\N$, define $S_m\in C^2(\R)$ by
\[
S_m(\xi)=
\begin{cases}
\abs{\xi}^p, 
	& \abs{\xi}\leq m \\
	m^{p-2}\left\{\frac{p(p-1)}{2}\xi^2-p(p-2)m\abs{\xi}
	+\frac{(p-1)(p-2)}{2}m^2\right\},  
	& \abs{\xi}> m
\end{cases}.
\]
The following elementary properties can be easily verified:
\begin{equation}\label{eq: elementary inequalities}
\begin{split}
& \abs{\xi S_m'(\xi)} \leq p S_m(\xi), \\
& \abs{S_m'(\xi)} \leq p(1+ S_m(\xi)), \\
& \abs{S_m'(\xi)} \leq \abs{\xi} S_m''(\xi), \\
& \xi^2 S_m''(\xi) \leq p(p-1) S_m(\xi), \\
& S_m''(\xi) \leq p(p-1)(1+ S_m(\xi)). 
\end{split}
\end{equation}

Before applying Proposition \ref{prop: generalized Ito}, we make 
a few preliminary observations. Since $u^\varepsilon$ is $\Ltwo$-predictable 
and $B$ is Lipschitz on $\Ltwo$, it follows that $B(u^\varepsilon)$ 
is $\hS$-predictable. By means of \eqref{eq: growth condition 2}, we see that also
$f_x(u^\varepsilon)$ is $\overrightarrow{\Ltwo}$-predictable. 
Moreover, if we utilize test functions in \eqref{eq: analytically weak solution} 
that are in $C^2(M)$, we are allowed to integrate by parts the term involving 
$\nabla u^\varepsilon$, obtaining 
\[
\int_0^t \int_M\varepsilon u^\varepsilon(s) \Delta_h\theta \, dV_h(x)\, ds. 
\] 

Setting $F\equiv 0$, $G(s):=-f_x(u^\varepsilon(s))$, 
$I(s):=\varepsilon u^\varepsilon(s)$, and $H(s):=B(u^\varepsilon(s))$, we see 
that all the assumptions in Proposition \ref{prop: generalized Ito} 
are satisfied. As a result, we can apply the generalized It\^{o} 
formula with $S_m(\cdot)$ and $\psi\equiv 1$, with the result that
\begin{equation}\label{eq: Ito con Sm}
\begin{split}
& \int_MS_m(u^\varepsilon(t))\,dV_h(x)
=\int_MS_m(u^\varepsilon_0)\,dV_h(x)
\\ & \qquad + \int_0^t \int_M\left(-\varepsilon\nabla u^\varepsilon(s)
+f_x(u^\varepsilon(s)),\nabla S_m'(u^\varepsilon(s)) \right)_h 
\,dV_h(x)\,ds \\ & \qquad \quad 
+  \left(\,\int_0^t S_m'(u^\varepsilon(s))
B(u^\varepsilon(s))\,dW(s),1\right)_{L^2(M)}
\\ & \qquad \quad \quad 
+\frac12\sum_{k\geq 1}\int_0^t\,
\int_MS_m''(u^\varepsilon(s)) 
g_k^2(x,u^\varepsilon(s)) \,dV_h(x)\, ds,
\end{split}
\end{equation}
$\P$-a.s., for any $t\in[0,T]$.

In view of Lemma \ref{lemma: gc2}, also in the 
present case the integral involving the flux $f$ vanishes. 
Furthermore, since $S_m''\geq 0$, we see that
\begin{align*}
& \int_0^t \int_M\left(-\varepsilon\nabla u^\varepsilon(s),
\nabla S_m'(u^\varepsilon(s)) \right)_h \,dV_h(x)\,ds
\\ & \qquad = -\varepsilon\int_0^t \int_M 
\abs{\nabla u^\varepsilon(s)}_h^2 
S_m''(u^\varepsilon(s)) \,dV_h(x)\,ds\leq 0. 
\end{align*}
By making use of \eqref{eq: elementary inequalities} 
and \eqref{eq: definition of G}, 
\begin{align*}
& \frac12\int_0^t\int_M S_m''(u^\varepsilon(s)) 
G^2(x,u^\varepsilon(s)) \, dV_h(x)\, ds
\\ & \qquad \leq \frac{D_1}{2}\int_0^t\int_M 
S_m''(u^\varepsilon(s)) \left(1+\abs{u^\varepsilon(s)}^2\right) 
\, dV_h(x)\, ds
\\ & \qquad \leq D_1p(p-1)
\int_0^t\int_M\left(\frac12+ S_m(u^\varepsilon(s))\right) dV_h(x) \,ds.
\end{align*}

Utilizing our findings in \eqref{eq: Ito con Sm} and then 
taking the expectation, we obtain, for any $t\in[0,T]$,
\begin{align*}
& \EE \int_MS_m(u^\varepsilon(t))\,dV_h(x) 
\leq \EE\int_MS_m(u^\varepsilon_0)\,dV_h(x) 
\\ & \qquad
+\EE\, \sum_{k\geq 1}\int_0^t\int_M 
S_m'(u^\varepsilon(s))g_k(x,u^\varepsilon(s))
\,dV_h(x)\,d\beta_k(s)
\\ &\qquad \qquad 
+D_1p(p-1)\frac{t}{2}+\int_0^t D_1p(p-1)
\, \EE\int_MS_m(u^\varepsilon(s)) \,dV_h(x)\, ds. 
\end{align*}

In view of the Burkholder-Davis-Gundy 
inequality \eqref{eq: BDG}, arguing as in the 
proof of Proposition \ref{prop: bound for the L2 norm}, 
it follows that the second term on the right-hand side 
of the inequality can be bounded as follows:
\begin{align*}
&\EE\sup_{0\leq t\leq T}
\abs{\sum_{k\geq 1}\int_0^t\int_M S_m'(u^\varepsilon(s)) 
g_k(x,u^\varepsilon(s)) \,dV_h(x)\,d\beta_k(s)}
\\ & \qquad \leq C\left(1+\EE\sup_{0\leq t\leq T}
\norm{u^\varepsilon(t)}_{L^2(M)}^2\right), 
\end{align*}
where the constant $C$ depends on $D_1, T, p$ but not 
$m,\varepsilon,C_0,r,L,C_1,D_2$. 
From Proposition \ref{prop: bound for the L2 norm} and the 
simple fact that $\norm{u^\varepsilon}_{L^2(M)}^2\leq 1 
+\norm{u^\varepsilon}_{L^p(M)}^p$ ($p>2$), the 
last term may be bounded by
\[
C\left(1+\EE\, \norm{u^\varepsilon_0}_{L^p(M)}^p\right), 
\]
for some new constant $C$. 

Noting that $S_m(\xi)\leq \abs{\xi}^p$, we 
summarize our computations as follows:
\begin{align*}
&\EE\int_MS_m(u^\varepsilon(t))\,dV_h(x) 
\leq \EE \,\norm{u^\varepsilon_0}_{L^p(M)}^p
+C\left(1+\EE\,\norm{u^\varepsilon_0}_{L^p(M)}^p\right)
\\ &\qquad +D_1p(p-1)\frac{t}{2}
+\int_0^tD_1p(p-1)\EE\int_M S_m(u^\varepsilon(s)) \,dV_h(x) \,ds. 
\end{align*}
As in the proof of Proposition \ref{prop: bound for the L2 norm}, we 
apply the Gronwall inequality to obtain 
\[
\EE\int_M S_m(u^\varepsilon(t))\,dV_h(x) 
\leq C\left(1+\EE\,\norm{u^\varepsilon_0}_{L^p(M)}^p\right), 
\qquad t\in [0,T],
\]
and hence, by Fatou's lemma, 
\begin{equation}\label{eq: Gronwall-Bellman 2}
\EE\norm{u^\varepsilon(t)}_{L^p(M)}^p 
\leq C\left(1+\EE\,\norm{u^\varepsilon_0}_{L^p(M)}^p\right). 
\end{equation}

Let us return to \eqref{eq: Ito con Sm}. 
Taking into account the previous computations,
\begin{align*}
& \int_M S_m(u^\varepsilon(t))\,dV_h(x)
+\varepsilon \int_0^t \int_M\abs{\nabla u^\varepsilon(s)}_h^2
\, S_m''(u^\varepsilon(s)) \,dV_h(x)\,ds
\\ & \quad \leq
\norm{u^\varepsilon_0}_{L^p(M)}^p 
+\sum_{k\geq 1}\int_0^t
\int_M S_m'(u^\varepsilon(s))g_k(x,u^\varepsilon(s))\,dV_h(x)\,d\beta_k(s)
\\ &\quad\qquad +D_1p(p-1)\frac{t}{2}
+\int_0^t D_1p(p-1)\norm{u^\varepsilon(s)}_{L^p(M)}^p \,ds,
\end{align*}
$\P$-a.s., for any $t\in[0,T]$. 
Note that $S_m$ is convex and, for any $\xi\in \R$, 
$S_m(\xi)\to \abs{\xi}^p$ and $S_m''(\xi)\to p(p-1)\abs{\xi}^{p-2}$ 
as $m\to \infty$. As a result of this, the 
superadditivity of $\liminf$, Fatou's lemma, and 
taking the supremum over $[0,T]$, we obtain
\begin{equation}\label{Lp-bound-tmp1}
\begin{split}	
& \sup_{0\leq t\leq T} \norm{u^\varepsilon(t)}_{L^p(M)}^p
+\varepsilon p(p-1)\int_0^T 
\int_M \abs{u^\varepsilon(s)}^{p-2}
\abs{\nabla u^\varepsilon(s)}_h^2 \,dV_h(x)\,ds
\\ & \quad \leq 2\norm{u^\varepsilon_0}_{L^p(M)}^p 
+D_1p(p-1)T
+ \int_0^T2D_1p(p-1)\norm{u^\varepsilon(s)}_{L^p(M)}^p \,ds
+ \mathcal{I},
\end{split}
\end{equation}
$\P$-a.s., for any $t\in[0,T]$, where
$$
\mathcal{I}
:= 2\liminf_{m\to \infty}\sup_{0\leq t\leq T} 
\abs{\sum_{k\geq 1}\int_0^t\int_M S_m'(u^\varepsilon(s))
g_k(x,u^\varepsilon(s)) \,dV_h(x)\,d\beta_k(s)}.
$$

Note that $u^\varepsilon$ admits an 
$L^r(M,h)$-continuous modification for any $r\in(2,p)$, and thus 
for $p$ as well. This is a consequence of the fact that 
this is already known for $p=2$, the $L^p$ estimate 
\eqref{eq: Gronwall-Bellman 2}, and a standard interpolation argument.

Using the Fatou lemma and arguing as we have 
done several times before,
$$
\EE \, \mathcal{I}\le 
C\left(1+\EE\,\norm{u^\varepsilon_0}_{L^p(M)}^p\right).
$$
Taking the expectation in \eqref{Lp-bound-tmp1}, 
noting that $\varepsilon p(p-1)>\varepsilon$, we thus obtain 
\begin{align*}
& \EE\sup_{0\leq t\leq T}\norm{u^\varepsilon(t)}_{L^p(M)}^p
+\varepsilon \,\EE \int_0^T \int_M
\abs{u^\varepsilon(s)}^{p-2}\abs{\nabla u^\varepsilon(s)}_h^2 
\,dV_h(x)\,ds
\\ & \quad \leq 2\EE\, \norm{u^\varepsilon_0}_{L^p(M)}^p 
+D_1p(p-1)T+\int_0^T2D_1p(p-1)\EE\norm{u^\varepsilon(s)}_{L^p(M)}^p\,ds
\\ & \quad \qquad
+C\left(1+\EE\,\norm{u^\varepsilon_0}_{L^p(M)}^p\right). 
\end{align*}
Thanks to \eqref{eq: Gronwall-Bellman 2}, we conclude that there is 
a constant $K_p>0$, depending on $D_1$ and $T$ but independent of 
$\varepsilon,C_0,r,L,C_1,D_2$, such that 
\eqref{eq: bound for the Lp norm} holds.
\end{proof}

\subsection{Bounds on kinetic measure}
In view of Propositions \ref{prop: bound for the L2 norm} 
and \ref{prop: bound for the Lp norm}, there is an 
$\varepsilon$-independent constant $C_p$ such that
\[
\EE\int_{[0,T]\times M\times\R}
\abs{\xi}^p\,m^\varepsilon(ds,dx,d\xi)\leq C_p,
\qquad p\in [2,\infty),
\] 
where the ``parabolic" kinetic measure $m^\varepsilon$ is defined in 
Proposition \ref{prop: kinetic formulation of the parabolic problem}.

We also have the improved estimate
\begin{equation}\label{improved-diss}
\EE\abs{\int_{[0,T]\times M\times\R}
\abs{\xi}^{2p}\,m^\varepsilon(ds,dx,d\xi)}^2 \leq C_p, 
\quad p\in [0,\infty),
\end{equation}
To derive this estimate we replicate the 
proof of Proposition \ref{prop: bound for the Lp norm}, 
this time with $S_m(\xi)\to \abs{\xi}^q$ as $m\to\infty$, 
where $q:=2p+2$. Indeed, we obtain
\begin{align*}
& \int_0^T\int_M 
\varepsilon\abs{\nabla u^\varepsilon(s)}_h^2
\,S_m''(u^\varepsilon(s)) \,dV_h(x)\,ds
\\ & \qquad\leq 
\norm{u_0}_{L^q(M)}^q
+ \sum_{k\geq 1}\int_0^T\int_M S'_m(u^\varepsilon(s))
g_k(x,u^\varepsilon(s))\,dV_h(x)\,d\beta_k(s)
\\ & \qquad\qquad 
+ C_q\left(1+\int_0^T\norm{u^\varepsilon(s)}_{L^q(M)}^q\, ds\right), 
\qquad \text{$\P$-almost surely},
\end{align*}

We square and then take the expectation. 
By the It\^{o} isometry,
\begin{align*}
\mathcal{D}(m) & :=\EE\, \abs{\int_0^T \int_M \varepsilon 
\, \abs{\nabla u^\varepsilon(s)}_h^2\,S_m''(u^\varepsilon(s)) 
\,dV_h(x)\,ds}^2 \\ & \leq C+C \, \EE\, \norm{u_0}_{L^q(M)}
+ C\,\EE\int_0^T\sum_{k\geq 1} 
\abs{\int_MS'_m(u^\varepsilon(s))g_k(x,u^\varepsilon(s))\,dV_h(x)}^2\, ds
\\ & \qquad \qquad + C\, \EE\,
 \abs{\int_0^T\norm{u^\varepsilon(s)}_{L^q(M)}^q \,ds}^2, 
\end{align*}
where the constant $C$ is independent of $\varepsilon$. 
Making use of $\norm{\cdot}_{L^q(M)}\leq \norm{\cdot}_{L^{2q}(M)}$ and 
Proposition \ref{prop: bound for the Lp norm}, we arrive at
\begin{align*}
&\mathcal{D}(m)\leq C + C\,\EE\int_0^T\sum_{k\geq 1}
\abs{\int_M S'_m(u^\varepsilon(s))g_k(x,u^\varepsilon(s))\,dV_h(x)}^2\, ds. 
\end{align*}
for a new constant $C$. By Jensen's inequality, elementary 
properties of the function $S_m$, cf.~\eqref{eq: elementary inequalities}, 
\eqref{eq: definition of G}, and again Proposition 
\ref{prop: bound for the Lp norm}, we can bound the 
second term on the right-hand side by a constant. 
The final result is 
\begin{equation}\label{Dm-bound}
\mathcal{D}(m)\leq C_q,
\end{equation}
for some positive constant $C_q$ independent 
of $\varepsilon$ (and $m$).

Seeing that $0\leq S_m''(\xi) \uparrow q(q-1)\abs{\xi}^{q-2}$ 
pointwise as $m\to\infty$, the monotone 
convergence theorem implies
$$
\lim_{m\to \infty} \mathcal{D}(m) = \mathcal{D},
\qquad \text{$\P$-almost surely},
$$
where $\mathcal{D}=\mathcal{D}(\omega)\in [0,\infty]$ is defined by
$$
\mathcal{D}:=
\abs{\int_0^T\int_M \varepsilon\,
\abs{\nabla u^\varepsilon(s)}_h^2\, 
q(q-1)\abs{u^\varepsilon(s)}^{q-2} \,dV_h(x)\,ds}^2. 
$$
On the other hand, since $q=2p+2\ge 2$, Propositions 
\ref{prop: bound for the L2 norm} and 
\ref{prop: bound for the Lp norm} ensure that 
$\mathcal{D}<\infty$ ($\P$-a.s.). In particular,
$\liminf\limits_{m\to \infty} \mathcal{D}(m) 
= \mathcal{D}\in [0,\infty)$ ($\P$-a.s.). 
By means of Fatou's Lemma and \eqref{Dm-bound}, 
$\EE \, \mathcal{D} \leq C_q$, that is,
$$
\EE\abs{\int_0^T\int_M \varepsilon\,
\abs{\nabla u^\varepsilon(s)}_h^2\,
\abs{u^\varepsilon(s)}^{2p} \,dV_h(x)\,ds}^2
\leq C_p, \qquad p\in [0,\infty).
$$
The claim \eqref{improved-diss} follows from this.

\subsection{Existence part of Theorem \ref{thm:well-posed}}

Following \cite[Section 4]{DV2010} closely, using the a priori 
bounds derived in the preceding paragraphs, 
$\rho^\varepsilon=\En_{u^\varepsilon>\xi}$ converges weakly-$\star$ 
to some $\rho$ in $L^\infty(\Omega\times (0,T)\times M\times \R)$, 
along a subsequence as $\varepsilon\to 0$ (not relabeled). 
Moreover, $m^\varepsilon$ converges weakly-$\star$ to 
some $m$ in $L^2(\Omega;\mathcal{M})$, where 
$\mathcal{M}$ denotes the space of bounded Borel measures on 
$(0,T)\times M\times \R$. The limit $(\rho,m)$ 
constitutes a generalized kinetic solution 
according to Definition \ref{def: generalized solution}.
Thanks to Proposition \ref{prop: Reduction and uniqueness}, 
$\rho$ is actually a kinetic function, that is, $\rho=\En_{u>\xi}$ for 
some $u$. The function $u$ is a kinetic solution according to 
Definition \ref{def: solution}. Moreover, 
$u^{\varepsilon}=\int_\R\left(\rho^\varepsilon-\En_{0>\xi}\right)\, d\xi$ 
converges strongly to $u=\int_\R\left(\rho-\En_{0>\xi}\right)\, d\xi$ 
in $L^p(\Omega\times [0,T]\times M)$: 
$$
\lim_{\varepsilon\to 0}\, 
\EE \, \norm{u^\varepsilon-u}^p_{L^p([0,T]\times M)}=0,  
\qquad \forall p\in [1,\infty).
$$
Since the kinetic solution $u$ is unique, cf.~Theorem \ref{thm:well-posed}, 
the entire sequence converges. For the details, we refer again to \cite{DV2010}.

\section{Appendix}

\begin{lemma}\label{lemma: strumentale}
Let $F\in C^0_b(\R)$. Suppose $(u_j)_j\subset\Ltwo$ 
converges to $u$ in $\Ltwo$. Then $F(u_j)\to F(u)$ 
for $j\to\infty$ in $\Lp{p}$, for any $p\in[1,\infty[$.
\end{lemma}

\begin{proof}
We pick any subsequence $(j_k)_k$. Then 
the reverse dominated convergence theorem 
guarantees that there exists $(j_{k_l})_l$ such that we 
have pointwise convergence for a.e.~$x\in M$. 
Therefore, by continuity of $F$, we obtain
\[
\abs{F(u_{j_{k_l}}(x))-F(u(x))}^p\to 0,  \quad \text{as $l\to\infty$},
\]
for a.e.~$x\in M$. Moreover, 
$\abs{F(u_{j_{k_l}}(x))-F(u(x))}^p\leq 2^p\norm{F}_{L^\infty}^p$, and 
thus by dominated convergence $F(u_{j_{k_l}})$ tends 
in $\Lp{p}$ to $F(u)$ as $l\to\infty$. 
Since $(j_k)_k$ was arbitrary, the same 
result must hold for the whole sequence. 
\end{proof}


\begin{lemma}\label{lemma: strumentale 2}
Let $\psi$ be in $\mathcal{D}(\R)$ and $\psi\geq 0$. 
Let $\mu$ and $\nu$ be finite Borel measures on $\R$. 
Let $\eta$ be a standard mollifier on $\R$ and set 
$\eta_\varepsilon(\cdot):=\varepsilon^{-1}
\eta\left(\frac{\cdot}{\varepsilon}\right)$ 
for $\varepsilon$ positive. Then
\[
\int_\R\psi(u)\int_{\R^2}
\abs{\xi-\gamma}^2\eta_\varepsilon(u-\xi)
\,\eta_\varepsilon(u-\gamma)
\,\mu\otimes\nu(d\xi,d\gamma)\,du \leq C\varepsilon, 
\] 
where $C=4 \norm{\psi}_{L^\infty}\norm{\eta}_{L^\infty}
(\mu\otimes\nu)(\R^2)$
\end{lemma}

\begin{proof}
By direct computation, setting $S_y:=B_1(0)\cap B_1(y)$ 
for $y\in\R$, we see that
\begin{align*}
&\int_\R\psi(u)\int_{\R^2}\abs{\xi-\gamma}^2\eta_\varepsilon(u-\xi)
\,\eta_\varepsilon(u-\gamma)\,\mu\otimes\nu(d\xi,d\gamma)\,du \\
& =\int_{\R^2}\abs{\xi-\gamma}^2\left[\int_\R\psi(u)
\,\eta_\varepsilon(u-\xi)\,\eta_\varepsilon(u-\gamma)\,du\right]
\mu\otimes\nu(d\xi,d\gamma)\\
&=\int_{\R^2}\abs{\xi-\gamma}^2\left[
\int_{S_{\frac{\gamma-\xi}{\varepsilon}}}
\psi(\varepsilon v+\xi)\,\eta(v)\,\varepsilon^{-1}
\eta\left(v+\frac{\xi-\gamma}{\varepsilon}\right)\,dv\right]
\mu\otimes\nu(d\xi,d\gamma)\\
&=\int_{\abs{\gamma-\xi}\leq 2\varepsilon}
\abs{\xi-\gamma}^2\left[\int_{S_{\frac{\gamma-\xi}{\varepsilon}}}
\psi(\varepsilon v+\xi)\,\eta(v)\,\varepsilon^{-1}
\eta\left(v+\frac{\xi-\gamma}{\varepsilon}\right)\,dv\right]
\mu\otimes\nu(d\xi,d\gamma)
\\ & \leq\norm{\psi}_{L^\infty}
\norm{\eta}_{L^\infty}\,\varepsilon^{-1}
\int_{\abs{\gamma-\xi}\leq 2\varepsilon}\abs{\xi-\gamma}^2
\mu\otimes\nu(d\xi,d\gamma)\leq C\varepsilon. 
\end{align*}
\end{proof}

\end{document}